%% file: Flamm_CharHitchinReprRealClosedFields.tex
\documentclass[11pt, a4paper, UKenglish, reqno]{amsart}

\usepackage{amsmath,amssymb,amstext,amsthm,amsfonts} 
\usepackage[UKenglish]{babel}
\usepackage{hyperref}{\hypersetup{
 colorlinks,
 linkcolor=blue
}
\usepackage{graphicx}
\usepackage{graphics}
\usepackage{caption}
\usepackage{comment}
\usepackage{url}
\usepackage{mathtools}
\usepackage{mathrsfs}
\usepackage{color}
\usepackage{tikz-cd}
\usepackage{float}
\usepackage{enumerate}
\usepackage{enumitem}
\usepackage{mathdots}
\usepackage{cleveref}
\usepackage{stmaryrd}
\usepackage{todonotes}

\DeclareFontEncoding{LS2}{}{\noaccents@}
\DeclareFontSubstitution{LS2}{stix2}{m}{n}
\DeclareSymbolFont{arrows3}{LS2}{stix2tt}{m}{n}
\DeclareMathSymbol{\urtriangle}{\mathord}{arrows3}{"98}
\DeclareMathSymbol{\lltriangle}{\mathord}{arrows3}{"99}

\usepackage{thmtools, thm-restate}
\def\thmt@counterformatters{@Alph}

\newcommand{\Hom}{\mbox{\rm Hom}}
\newcommand{\GL}{\mbox{\rm GL}}
\newcommand{\SL}{\mbox{\rm SL}}
\newcommand{\PSL}{\mbox{\rm PSL}}
\newcommand{\PGL}{{\rm PGL}}

\newcommand{\tr}{\mbox{\rm tr}}

\newcommand{\Stab}{\mbox{\rm Stab}}
\newcommand{\Ad}{\textnormal{Ad}}

\newcommand{\Aut}{\textnormal{Aut}}
\newcommand{\Mat}{\textnormal{Mat}}
\newcommand{\Flag}{\textnormal{Flag}}
\newcommand{\Fix}{\textnormal{Fix}}

\newcommand{\pr}{\textnormal{pr}}
\newcommand{\Hit}{\textnormal{Hit}}
\newcommand{\red}{\textnormal{red}}
\newcommand{\ev}{\textnormal{ev}}
\newcommand{\Gr}{\textnormal{Gr}}

\newcommand{\rsp}{\textnormal{RSp}}
\newcommand{\CR}{\textnormal{CR}}
\newcommand{\BD}{\textnormal{BD}}

\newcommand{\Pos}{\textnormal{Pos}}
\newcommand{\Graph}{\textnormal{Graph}}
\newcommand{\cl}{\textnormal{cl}}

\newcommand{\Id}{\rm Id}
\newcommand*\from{\colon}
\newcommand{\Z}{\mathbb{Z}}
\newcommand{\K}{\mathbb{K}}

\newcommand{\Q}{\mathbb{Q}}
\newcommand{\N}{\mathbb{N}}
\newcommand{\R}{\mathbb{R}}
\renewcommand{\H}{\mathbb{H}}
\newcommand{\F}{\mathbb{F}}
\newcommand{\PP}{\mathbb{P}}

\renewcommand{\Im}{\textnormal{Im}}

\DeclareRobustCommand{\SkipTocEntry}[5]{}

\newtheorem*{theorem*}{Theorem}

\newtheorem*{lemma*}{Lemma}

\newtheorem{propo}{Proposition}[section]

\newtheorem{lem}[propo]{Lemma}
\newtheorem{corol}[propo]{Corollary}
\newtheorem{theor}[propo]{Theorem}

\theoremstyle{definition}
\newtheorem{dfn}[propo]{Definition}
\newtheorem*{dfn*}{Definition}

\newtheorem{examp}[propo]{Example}
\newtheorem{rem}[propo]{Remark}
\newtheorem*{rem*}{Remark}

\newtheorem*{ques*}{Question}

\hyphenation{ho-mo-mor-phism ho-mo-mor-phisms re-pre-sen-ta-tion re-pre-sen-ta-tions po-si-tive po-si-ti-vi-ty equi-va-lent}

\setcounter{tocdepth}{1}

\title[Real spectrum compactification of Hitchin components]{Real spectrum compactification of\\Hitchin components, Weyl chamber\\valued lengths,  and dual spaces}
\author{Xenia Flamm}
\address{Department of Mathematics, ETH Z\"{u}rich, Switzerland}
\email{xenia.flamm@math.ethz.ch}

\begin{document}

\def\subjclassname{\textup{2020} Mathematics Subject Classification}
\subjclass{
22E40, 
30F60, 
14P10
}

\begin{abstract}
The main result of this article is that Hitchin representations over real closed field extensions $\F$ of $\R$ correspond precisely to those representations of the fundamental group of a closed surface into $\PSL(n,\F)$ that are conjugate to $\F$-positive representations, i.e.\ representations that admit an equivariant limit map from the set of fixed points in the boundary of the universal cover of the surface into the set of full flags in $\F^n$ satisfying specific positivity properties.
As the theorem treats general real closed fields,  and not only the reals,  the tools of analysis are not available. 
Instead,  our proof is based on the Tarski--Seidenberg transfer principle and a multiplicative version of the Bonahon--Dreyer coordinates.

We use this result to prove that $\F$-positive representations form semi-algebraically connected components of the space of all representations,  that consist entirely of injective and discrete representations,  which are positively hyperbolic and weakly dynamics-preserving over $\F$.
Furthermore,  we show how to associate intersection geodesic currents to $\F$-positive representations,  and conclude with applications to the Weyl chamber length compactification and to dual spaces of geodesic currents.
\end{abstract}

\maketitle

\input{sections/01_intro.tex}

\tableofcontents
\input{sections/02_realalggeom.tex}
\input{sections/03_FHit.tex}

\input{sections/charactervarietieshigherTMspaces.tex}
\input{sections/04_flags.tex}
\input{sections/05_positivity.tex}
\input{sections/06_poshypFHitpos.tex}
\input{sections/07_BDcoordinates.tex}
\input{sections/08_proofthm.tex}
\input{sections/09_propertiesFHit.tex}

\input{sections/geodesiccurrents.tex}
\input{sections/associatinggeodcurrentsHitchin.tex}
\input{sections/associatinggeodcurrentsHitchin_continuity.tex}
\input{sections/11_proofapplications.tex}

\bibliographystyle{alpha}
\bibliography{bibl}

\end{document}

%% file: sections/01_intro.tex
\section{Introduction}
\label{section_Intro}
Let $S$ be a closed,  connected and oriented surface of genus at least two.
Higher Teichm\"uller theory concerns itself with the study of connected components of the character variety $\chi(S,G)$, for $G$ a semisimple Lie group of higher rank,  consisting entirely of injective representations with discrete image, called \emph{higher Teichm\"uller spaces} \cite{Wienhard_InvitationHigherTeichmuellerTheory}.
They are conjectured to parametrize geometric structures on $S$; see e.g.\ \cite{ChoiGoldman_ConvexRealProjStructuresClosedSurfaces, GuichardWienhard_ConvexFolProjStr}.
To understand these and degenerations thereof it is often useful to compactify them in a geometrically meaningful way; see for example \cite{MorganShalen_ValuationsTreesDegHypStructures, 
Bonahon_GeometryTeichmuellerSpaceGeodesicCurrents,
Loftin_CompModuliSpaceConvexProjStrSurfaces,
Alessandrini_CompactificationConvexProjStrMnfd,
Parreau_CompactificationsReprSpacesFinGenGroups,
OuyangTamburelli_LimitsBlaschkeMetric,
LoftinTamburelliWolf_LimitsCubicDifferentialsBuildings,
SagmanSmillie_UnstableMinSurfSymSpaces,
Reid_LimitsConvexProjectiveSurfacesFinslerMetrics}.

In this article, we concentrate on the \emph{real spectrum compactification} of $\chi(S,G)$, whose study was initiated by Burger--Iozzi--Parreau--Pozzetti in their seminal works \cite{BurgerIozziParreauPozzetti_RSCCharacterVarieties,BurgerIozziParreauPozzetti_RSCCharacterVarieties2}, inspired by the work of Brumfiel in the case of Teichm\"uller space \cite{Brumfiel_RSCTeichmullerSpace}.
Its boundary points are identified with equivalence classes of reductive representations from $\pi_1(S)$ into $G_\F$,  where $\F$ is a real closed field containing $\R$ and $G_\F$ is the $\F$-extension of $G$ (\Cref{dfn_ExtSemiAlgSets}) \cite[Theorem 1.1]{BurgerIozziParreauPozzetti_RSCCharacterVarieties2}.
Thus boundary points can be studied in terms of representations.
For a higher Teichm\"uller space $\mathcal{C}$ of $\chi(S,G)$ a natural question in this context is the following:
Given a reductive representation from $\pi_1(S)$ to $G_\F$,  can we determine ``geometrically'' whether or not it is in the boundary of $\mathcal{C}$?

We answer this question in the case when $\mathcal{C}$ is the Hitchin component---a higher Teichm\"uller space of the $\PSL(n,\R)$-character variety.
In analogy to the real case, see \cite{FockGoncharov_ModuliSpacesLocalSystemsHigherTeichmuellerTheory}, we give a characterization of $\F$-Hitchin representations in terms of positivity of an equivariant limit map from a subset of the boundary of the universal cover of $S$ to the set of full flags in $\F^n$.

\begin{restatable}{theor}{thmA}
\label{thm_FHitchinEquivFPositive}
Let $\F$ be a real closed extension of $\R$.
A representation $\rho \from \pi_1(S) \to \PSL(n,\F)$ is $\PGL(n,\F)$-conjugate to an $\F$-Hitchin representation if and only if it is $\F$-positive.
\end{restatable}

This characterization gives a semi-algebraic description of the set of $\F$-positive representations and thus allows a direct study of these representations via Hitchin representations.
For example,  our description allows to prove that for every $\F$-positive representation $\rho$ there exists a continuous path $\rho_t$ of real Hitchin representations,  such that the renormalized length functions of $\rho_t$ converge to the length function of $\rho$ (\Cref{corol_PiecewiseSemiAlgPath}).
As the theorem treats general $\F$ instead of just the reals,  the tools of analysis---crucially used for the analog result over $\R$---are not available. 
Instead,  our proof is based on carefully describing and characterizing all involved notions semi-algebraically and using the Tarski--Seidenberg transfer principle. 
Furthermore,  we provide a multiplicative variant of the Bonahon--Dreyer-coordinates for the set of $\F$-Hitchin representations.
For maximal components in $\chi(S,\textrm{Sp}(2n,\R))$ the analogue of the above result was announced in  \cite{BurgerIozziParreauPozzetti_RSCCharacterVarieties}.
More generally,  one can ask if the closure of the set of $\Theta$-positive representations in the real spectrum compactification of $\chi(S,G)$,  for $G$ a Lie group admitting a $\Theta$-positive structure,  can be characterized in a similar way \cite{ GuichardWienhard_PositivityAndHigherTeichmuellerTheory, GuichardLabourieWienhard_PositivityReprSurfceGroups, GuichardWienhard_GeneralizingLusztigsTotalPositivity,
BurgerIozziParreauPozzetti_RSCCharacterVarieties2}.

\medskip
Let us now explain the objects and results of this article in more detail.
In his seminal paper \cite{Hitchin_LieGroupsAndTeichmuellerSpace}, Hitchin used Higgs bundles to show that $\chi(S,\PSL(n,\R))$ for $n\geq 3$ contains three connected components, if $n$ is odd, and six connected components,  if $n$ is even.
In the odd case, one of the three components, and in the even case, two of the six components, are homeomorphic to $\R^{(2g-2)(n^2-1)}$.
They can be characterized as follows.
Denote by $\iota_n$ the irreducible $n$-dimensional representation of $\PSL(2,\R)$ into $\PSL(n, \R)$.

\begin{dfn}
\label{dfn_HitchinComp}
Fix $j \from \pi_1(S) \to \PSL(2,\R)$ a holonomy representation of a marked hyperbolic structure on $S$,  i.e.\ $j$ is faithful with discrete image.
The \emph{Hitchin component} $\Hit(S,n)$ is the connected component of $\chi(S,\PSL(n,\R))$ containing $\iota_n \circ j$.
A representation $\pi_1(S) \to \PSL(n,\R)$ whose $\PSL(n,\R)$-conjugacy class lies in the Hitchin component is a \emph{Hitchin representation}.
\end{dfn}

Denote by $\rsp(\chi(S,\PSL(n,\R)))$ the real spectrum compactification of the $\PSL(n,\R)$-character variety.
To study its points we need some notions from real algebraic geometry,  that will be introduced in detail in \Cref{section_BackgroundRealAlgGeom}.
A \emph{semi-algebraic set} is a subset of $\R^d$ that is cut out by finitely many polynomial equalities and inequalities.
We refer the interested reader to \cite[Sections 7.1, 7.2]{BochnakCosteRoy_RealAlgebraicGeometry} for the general definition of the real spectrum compactification of closed semi-algebraic sets.
The special case of the character variety, which by \cite{RichardsonSlodowy_MinimumVectorsForRealReductiveGroups} is semi-algebraic, is treated in \cite{BurgerIozziParreauPozzetti_RSCCharacterVarieties2}.
If $\F$ is a \emph{real closed field} extension of $\R$,  i.e.\ $\F$ is an orderable field for which $\F[\sqrt{-1}]$ is algebraically closed, then we can look at the $\F$-solutions of a semi-algebraic set, which we will call its \emph{$\F$-extension}.
Note that as a connected component of a semi-algebraic set,  $\Hit(S,n)$ is semi-algebraic.
With this we can define the following.

\begin{dfn}
\label{dfn_FHitchin}
A representation $\pi_1(S) \to \PSL(n,\F)$ is \emph{$\F$-Hitchin} if its $\PSL(n,\F)$-conjugacy class lies in the $\F$-extension of $\Hit(S,n)$,  called the \emph{$\F$-Hitchin component} and denoted $\Hit(S,n)_\F$.
\end{dfn}

Let us now describe the notion of positivity over real closed fields.
Fix an auxiliary hyperbolic structure on $S$.
Its universal cover $\tilde{S}$ is identified with $\H^2$, and let $\partial_\infty \tilde{S} \coloneqq \partial\H^2$ denote its circle at infinity.
The action of $\pi_1(S)$ on $\tilde{S}$ extends to a continuous action on $\partial_\infty \tilde{S}$.
The latter is identified with $\mathbb{S}^1$ and inherits a cyclic order.
Let $\Fix(S) \subseteq \partial_\infty \tilde{S}$ be the $\pi_1(S)$-invariant subset of $\partial_\infty\tilde{S}$ of $\pi_1(S)$-fixed points.
A full flag in $\F^n$ is a nested sequence of $n+1$ subspaces of $\F^n$ of strictly increasing dimension.
Denote by $\Flag(\F^n)$ the set of full flags in $\F^n$.
Given any ordered field $\F$, the concept of positive tuples of full flags in $\F^n$ can be defined in the same way as for $\R$, since it only involves positivity conditions on triple and double ratios (\Cref{dfn_TripleRatio} and \Cref{dfn_DoubleRatio}).
This leads to the following.

\begin{dfn}
\label{dfn_Fpositivity}
A map $\xi \from \Fix(S) \to \Flag(\F^n)$ is \emph{positive} if it maps any triple respectively quadruple of cyclically ordered points in $\Fix(S)$ to a positive triple respectively quadruple of flags (\Cref{dfn_PostivityTuples}).

A representation $\rho \from \pi_1(S) \to \PSL(n,\F)$ is \emph{positive} if there exists a (not necessarily continuous) positive map
\[ \xi_\rho \from \Fix(S) \to \Flag(\F^n), \]
that is $\rho$-equivariant,  i.e.\ for all $x \in \Fix(S)$ and $\gamma \in \pi_1(S)$ we have $\xi_\rho(\gamma x) = \rho(\gamma) \xi_\rho(x)$.
The map $\xi_\rho$ is called a \emph{limit map} of $\rho$.
\end{dfn}

For $\F=\R$, it is known that a representation is Hitchin if and only if it is $\PGL(n,\R)$-conjugate to an $\R$-positive representation in the sense of our definition,  see \cite[Theorem 1.14, Theorem 1.15]{FockGoncharov_ModuliSpacesLocalSystemsHigherTeichmuellerTheory}.
\Cref{thm_FHitchinEquivFPositive} is thus a direct generalization of this result to real closed field extensions of $\R$.

From the construction of $\rsp(\chi(S,\PSL(n,\R)))$ it follows that $\F$-Hitchin representations are exactly those representations representing points in the boundary of $\Hit(S,n)$ in the real spectrum compactification of the character variety,  see \Cref{subs_RSCCharVar}.
In fact,  the following corollary of the above considerations and \Cref{thm_FHitchinEquivFPositive} completes the proof of a result announced in \cite[Theorem 46]{BurgerIozziParreauPozzetti_RSCCharacterVarieties} and answers the question of geometric characterization for $\mathcal{C}=\Hit(S,n)$.

Denote by $\Hit(\bar{S},n)$ the connected component of $\chi(S,\PSL(n,\R))$ containing $\iota'_n \circ j$,  where $\iota'_n$ is any conjugate of $\iota_n$ under an element of $\PGL(n,\R) \setminus \PSL(n,\R)$.
If $n$ is odd,  $\Hit(\bar{S},n)$ is the same as $\Hit(S,n)$.

\begin{corol}
\label{corol_RSpCharHit}
A reductive representation $\pi_1(S) \to \PSL(n,\F)$ represents a point in the closure of $\Hit(S,n) \cup \Hit(\bar{S},n)$ in $\rsp(\chi(S,\PSL(n,\R)))$ if and only if it is $\F$-positive.
\end{corol}

For the backward direction in the proof of \Cref{thm_FHitchinEquivFPositive},  we introduce a multiplicative variant of the Bonahon--Dreyer coordinates for the Hitchin component \cite{BonahonDreyer_ParametrizingHitchinComponents}.
This variant can be extended to give coordinates for $\Hit(S,n)_\F$ for every real closed field $\F$ (\Cref{corol_CoordBDF}).
We then prove the following.

\begin{restatable}{theor}{thmB}
\label{thm_BDFParam}
The set of $\PGL(n,\F)$-equivalence classes of $\F$-positive representations is parametrized by the Bonahon--Dreyer coordinates over $\F$ and hence in bijection with a semi-algebraic subset of some $\F^N$.
\end{restatable}

We summarize general properties for $\F$-Hitchin representations.

\begin{propo}[\Cref{lem_hitchinReprAreDiagonalizable}, \Cref{propo_positiveReprHaveDistinctEigenvalues}, \Cref{propo_FHitchinDiscrete}]
\label{propo_PropertiesFHitchin}
Let $\rho \from \pi_1(S) \to \PSL(n,\F)$ be $\F$-Hitchin.
Then $\rho$ is injective,  discrete\footnote{Careful,  unless $\F=\R$ the topology is not locally compact.},  positively hyperbolic (\Cref{dfn_PosHypRepr}) and weakly dynamics-preserving (\Cref{dfn_DynPres}).
\end{propo}

For real closed field extensions $\F$ of $\R$,  this implies that the $\F$-Hitchin component---the semi-algebraically connected component of the $\F$-extension of $\chi(S,\PSL(n,\R))$ containing the $\PSL(n,\F)$-conjugacy class of $\iota_n \circ j$---can be thought of as an example of a ``non-Archimedean higher Teichm\"uller space''.
\medskip

Next we discuss how to associate a geodesic current to an $\F$-positive representation.
Geodesic currents were introduced by Bonahon in \cite{Bonahon_BoutsVarHypDim3}, providing a unifying framework for Thurston's compactification.
They are now an active field of research by themselves and have applications in various fields in the study of surfaces; see e.g.\ \cite{ErlandssonSouto_MirzakhaniCurveCountingGeodCurrents}.
To simplify we choose a hyperbolic structure on $S$, even though the following concept can be defined in a purely topological way,  see e.g.\ \cite[Fact 1]{Bonahon_GeometryTeichmuellerSpaceGeodesicCurrents}.

\begin{dfn}
\label{dfn_GeodesicCurrent}
A \emph{geodesic current} on $S$ is a locally finite, $\pi_1(S)$-invariant, regular Borel measure on the space of (unoriented and unparametrized) geodesics in the universal cover $\tilde{S}$ of $S$.
\end{dfn}

Denote by $\mathscr{C}(S)$ the space of geodesic currents on $S$ endowed with the weak$^\star$-topology, and by $\mathbb{P}\mathscr{C}(S):=\mathscr{C}(S)/\mathbb{R}_{>0}$ the space of projectivized geodesic currents, the latter being compact \cite[Proposition 4, Corollary 5]{Bonahon_GeometryTeichmuellerSpaceGeodesicCurrents}.
Many seemingly different objects are in fact geodesic currents, e.g.\ homotopy classes of closed curves on $S$ or isotopy classes of marked hyperbolic structures on $S$ \cite[Lemma 9]{Bonahon_GeometryTeichmuellerSpaceGeodesicCurrents}.
Martone-Zhang proved in \cite[Theorem 3.4]{MartoneZhang_PositivelyRatioedRepr},  that we can associate to a Hitchin representation a geodesic current such that $k$-length functions of the representation can be computed as intersections (\Cref{dfn_IntersectionGeodCurr}) with this current.
More precisely, let $\F\supseteq \R$ be a real closed field that admits an order-compatible valuation $\upsilon$,  see \Cref{dfn_OrderCompVal} (for $\F=\R$ we can take $\upsilon=-\log$).
For $g \in \PSL(n,\F)$ with distinct eigenvalues of the same sign, let $\lambda_1(g) > \ldots > \lambda_n(g)>0$ be the eigenvalues of a lift of $g$ to $\SL(n,\F)$.
For every $k=1,\ldots,n-1$ we define the \emph{$k$-length of $g$} as
\begin{equation}
\label{eqn:kLengths}
L_k(\rho(\gamma))\coloneqq -\sum_{j=1}^k \upsilon(\lambda_j(g)) + \sum_{j=n-k+1}^n \upsilon(\lambda_j(g)).
\end{equation}
These length functions can be interpreted as the translation length of an element $g \in \PSL(n, \F)$ acting on the metric space $\mathscr{B}_{\textnormal{PSL}(n,\F)}$ as defined in \cite[Section 3.4 and 4]{BurgerIozziParreauPozzetti_RSCCharacterVarieties}---the non-Archimedean analogue of the symmetric space of $\PSL(n,\R)$ that generalizes real trees to higher rank.

We establish the following result announced by Burger--Iozzi--Parreau--Pozzetti \cite[Theorem 47]{BurgerIozziParreauPozzetti_RSCCharacterVarieties}.
An equivalent version was proven for maximal representations by Burger--Iozzi--Parreau--Pozzetti \cite[Theorem 1.2]{BIPP_PositiveCR},  and recently extended to positively ratioed representations \cite[Theorem 9.6]{BurgerIozziParreauPozzetti_RSCCharacterVarieties2}.
Recall that to every $\rho \in \Hom_\Hit(\pi_1(S),\PSL(n,\R))$ are associated $n-1$ geodesic currents \cite{MartoneZhang_PositivelyRatioedRepr}.
This result can be extended in the following way.

\begin{theor}
\label{thm_FHitchinCR}
Let $\F\supseteq \R$ be a real closed field with an order-compatible valuation $\upsilon$ (assumed to be $-\log$ if $\F=\R$) and let $\rho \from\pi_1(S) \to \PSL(n,\F)$ be $\F$-positive.
Then for every $k=1,\ldots,n-1$,  there exists a geodesic current $\mu^k_\rho$ such that for any $e \neq \gamma \in \pi_1(S)$ we have
\[ i(\mu^k_\rho,\gamma)=L_k(\rho(\gamma)).\]
The current $\mu_\rho^k$ is non-zero if and only if there exists $\gamma \in \pi_1(S)$ with $\upsilon(|\tr(\rho(\gamma))|) < 0$.
\end{theor}

Considering the subset $\rsp^\cl(\Hit(S,n))$ of closed points in $\rsp(\Hit(S,n))$,  this implies the following.

\begin{corol}
\label{corol_RSpHitGeodCurr}
For all $k=1,\ldots,n-1$ the map
\begin{align*}
\rsp^\cl(\Hit(S,n)) &\to \PP  \mathscr{C}(S),\\
[(\rho,\F)] &\mapsto \left[\mu^k_\rho\right]
\end{align*}
is well-defined and continuous.
\end{corol}

\addtocontents{toc}{\SkipTocEntry}
\subsection*{Applications}
We end the introduction with two applications of our results: one concerning Weyl chamber length functions \cite{Parreau_CompactificationsReprSpacesFinGenGroups},  and the other dual spaces of geodesic currents \cite{DeRosaMartinezGranado_DualSpacesGeodesicCurrents}.
The objects will be introduced in detail in \Cref{section_ProofApplications}.

Let $\F$ be a real closed field with an order-compatible valuation $\upsilon \from \F \to \R \cup \{\infty\}$ (if $\F=\R$,  we take $\upsilon=-\log$). 
Given $\rho \from \pi_1(S) \to \PSL(n,\F)$ an $\F$-positive representation,  then $\rho(\gamma)$ is positively hyperbolic with eigenvalues $\lambda_1(\rho(\gamma))> \ldots > \lambda_n(\rho(\gamma))>0$.
This is classical for $\F=\R$ and follows from \Cref{propo_positiveReprHaveDistinctEigenvalues} for general $\F$.
Define the \emph{Weyl chamber length function}
\begin{align*}
L_\rho \from \pi_1(S)&\to \mathfrak{a}^+,\\
\gamma &\mapsto (-\upsilon(\lambda_1(\rho(\gamma))),\ldots,-\upsilon( \lambda_n(\rho(\gamma)))),
\end{align*}
where $\mathfrak{a}$ denotes the Cartan subalgebra of the Lie algebra of $\SL(n,\R)$ consisting of diagonal matrices of trace zero,  and $\mathfrak{a}^+$ the Weyl chamber consisting of those elements $\textnormal{diag}(a_1,\ldots,a_n)\in \mathfrak{a}$ with $a_1>\ldots>a_n$.

\begin{corol}
\label{corol_PiecewiseSemiAlgPath}
Let $\rho \from \pi_1(S) \to \PSL(n,\F)$ be an $\F$-positive representation,  where $\F$ is a non-Archimedean field with an order-compatible valuation.
Assume $L_\rho$ does not vanish identically.
Then there is a continuous piecewise semi-algebraic path 
\[
\R_{\geq 0} \to \Hom_\Hit(\pi_1(S),\PSL(n,\R)), \quad
t \mapsto \rho_t
\]
and a constant $c>0$ such that
\[ \lim_{t \to \infty}\frac{L_{\rho_t}(\gamma)}{t}=c L_\rho(\gamma).\]
\end{corol}

In particular,  the renormalized $k$-length functions (as defined in (\ref{eqn:kLengths})) converge (up to scaling) to the length function of $\rho$.

\medskip
This has the following application to dual spaces of geodesic currents.
Recall that to any current $\mu$ one associates the dual space as defined in \cite{DeRosaMartinezGranado_DualSpacesGeodesicCurrents}.
In the case when $\mu$ is the geodesic current associated to some $\rho \in \Hom_\Hit(\pi_1(S),\PSL(3,\R))$,  taking the dual space gives a distance on $\H^2$ that is isometric $\Omega_\rho$ endowed with its Hilbert metric,  where $\Omega_\rho$ is the properly convex subset of $\R\mathbb{P}^2$ associated to $\rho$ \cite{ChoiGoldman_ConvexRealProjStructuresClosedSurfaces}.
Using the above corollary we obtain the following.

\begin{corol}
\label{corol_DualSpace}
In the setting of \Cref{corol_PiecewiseSemiAlgPath} for $n=3$,  the sequence of convex projective structures with Hilbert metrics $(\Omega_{\rho_t},\tfrac{1}{t}d_{\Omega_{\rho_t}})$ converges in the equivariant Gromov--Hausdorff topology to the dual space of the current $\mu_\rho$ associated to $\rho$ (up to scaling).
\end{corol}

\addtocontents{toc}{\SkipTocEntry}
\subsection*{Methods}
The key tool to establish the forward direction of \Cref{thm_FHitchinEquivFPositive} is the \emph{Tarski--Seidenberg transfer principle} (\Cref{thm_TarskiSeidenberg}),  a real closed analogue of the Lefschetz principle for algebraically closed fields.
It enables to deduce results for real closed extensions of $\R$,  that are stated in first-order logic in the language of ordered fields,  from results that hold true over $\R$.
This allows to prove positive hyperbolicity for $\F$-Hitchin representations.
Equipped with this we construct a limit map for an $\F$-Hitchin representation $\rho$.
Namely to the fixed point of a hyperbolic element $\gamma$ we associate the stable flag of the positively hyperbolic element $\rho(\gamma)$.
To prove its positivity properties,  we use again the Tarski--Seidenberg transfer principle and the positivity properties of real Hitchin representations (\Cref{thm_FlagCurveHitchinRepr}).
The backward direction of \Cref{thm_FHitchinEquivFPositive} is more involved.
Here we cannot apply the transfer principle, as the notion of $\F$-positive representation involves infinitely many conditions.
We will show that, actually, a finite set of data points suffices to determine an $\F$-positive representation up to conjugation.
The equivariant limit map that comes together with an $\F$-positive representation allows us to associate to the representation the variant of the Bonahon--Dreyer coordinates over $\F$. 
We then verify that these coordinates satisfy the same polynomial equalities and inequalities as the ones for $\F$-Hitchin representations,  and that two $\F$-positive representations have the same coordinates if and only if they are conjugate.
This proves \Cref{thm_BDFParam}, and then also \Cref{thm_FHitchinEquivFPositive}.
The main difficulty compared to the case $\F=\R$ lies in the proof of the closed leaf inequality \ref{item_ClosedLeafInquality} in \Cref{subsection_BDRelations}.
For this we prove that $\F$-positive representations are positively hyperbolic and weakly dynamics-preserving (\Cref{propo_positiveReprHaveDistinctEigenvalues}) without any analytic tools, relying solely on their positivity properties.

To prove \Cref{thm_FHitchinCR} we show that $\F$-positive representations are positively ratioed using the Tarski--Seidenberg transfer principle,  to then apply a result of Burger--Iozzi--Parreau--Pozzetti about geodesic currents associated to positive cross-ratios (\Cref{thm_BIPPGeodCurrAssToPosCR}).

\addtocontents{toc}{\SkipTocEntry}
\subsection*{Organization}
The article is organized as follows.
\Cref{section_BackgroundRealAlgGeom} states the important notions and results from real algebraic geometry and sets up notation.
We introduce Hitchin representations over real closed fields in \Cref{section_SemiAlgHitchin}.
In \Cref{section_PreliminariesFlags} we give the necessary preliminaries on flags, including the definition of the triple and double ratios.
This leads to the notion of positive tuples of full flags in \Cref{section_Positivity},  and we hint at the connection between positivity of flags and positivity properties of matrices.
Even though most results from \Cref{section_BackgroundRealAlgGeom}-\ref{section_Positivity} are known to the experts for $\R$, we carefully study all objects in question focusing in particular on their semi-algebraic structures,  with the goal of generalizing to real closed fields different from $\R$.

\Cref{section_PosHypRepr}-\ref{section_ProofOfMainTheorem} are the heart of the proof of \Cref{thm_FHitchinEquivFPositive}:
\Cref{section_PosHypRepr} is devoted to the proof of positive hyperbolicity of $\F$-Hitchin and $\F$-positive representations.
In Section \ref{section_VariantOfBDCoordinates} we define a semi-algebraic variant of the Bonahon--Dreyer coordinates that generalizes to real closed fields different from $\R$.
We finish by giving a proof of Theorem \ref{thm_FHitchinEquivFPositive} in Section~\ref{section_ProofOfMainTheorem}.

\Cref{section_PropBoundaryRepr} collects properties of $\F$-Hitchin representations and finishes the proof of \Cref{propo_PropertiesFHitchin}.
In \Cref{section_GeodCurrentsFPosRepr},  we give some background on geodesic currents,  positive cross-ratios,  and valuations. 
We then prove \Cref{thm_FHitchinCR} and \Cref{corol_RSpHitGeodCurr}.
The proof of the two applications in \Cref{section_ProofApplications} finishes the paper.

\addtocontents{toc}{\SkipTocEntry}
\subsection*{Acknowledgments}
I thank Marc Burger for introducing me to this topic,  his helpful remarks and explaining the interesting applications of the main result to me.
A special thanks goes to Giuseppe Martone for pointing me to the correct way of proving \Cref{propo_positiveReprHaveDistinctEigenvalues}.
I am particularly grateful to Jacques Audibert, Karin Baur, Peter Feller, Gerhard Hi{\ss}, Alessandra Iozzi,  Fanny Kassel, Mareike Pfeil,  Valdo Tatitscheff and Tengren Zhang for constructive discussions or detailed feedback on this manuscript.

%% file: sections/02_realalggeom.tex
\section{Real algebraic geometry}
\label{section_BackgroundRealAlgGeom}
In this section we recall general definitions and results from real algebraic geometry and set up notation.
We refer the reader to \cite{BochnakCosteRoy_RealAlgebraicGeometry}, in particular Chapters 1, 2,5 and 7, for more details and proofs.

\subsection{Semi-algebraic sets}

Let $\F$ be an ordered field.
It is \emph{real closed} if every positive element is a square and every odd degree polynomial has a root.
Examples include the real numbers, the real algebraic numbers and the field of Puiseux series with coefficients in $\R$.
Every ordered field $\K$ has a \emph{real closure} $\overline{\K}^r$, i.e.\ an algebraic field extension that is real closed and that is extending the order.
The main objects of study in real algebraic geometry are semi-algebraic sets.

\begin{dfn}\label{dfn_SemiAlgSet}
Let $\F$ be a real closed field.
A subset $B \subseteq \F^d$ is a \emph{basic semi-algebraic set}, if there exists a polynomial $f \in \F[X_1,\ldots,X_d]$ such that
\[ B = \{ x \in \F^d \mid f(x)>0\}.\]
A subset $X \subseteq \F^d$ is \emph{semi-algebraic} if it is a Boolean combination of basic semi-algebraic sets,  i.e.\ $X$ is obtained by taking finite unions and intersections of basic semi-algebraic sets and their complements.

Let $X \subseteq \F^d$ and $Y \subseteq \F^m$ be two semi-algebraic sets.
A map $f \from X \to Y$ is called \emph{semi-algebraic} if its graph $\textnormal{Graph}(f) \subseteq X \times Y$ is semi-algebraic in $\F^{d+m}$.
\end{dfn}

Note that algebraic sets are semi-algebraic and any polynomial or rational map is semi-algebraic.

\begin{propo}[{\cite[Proposition 2.2.6 (i)]{BochnakCosteRoy_RealAlgebraicGeometry}}]
\label{propo_CompSemiAlgMap}
The composition $g \circ f$ of semi-algebraic maps $f \from X \to Y$ and $g \from Y \to Z$ is semi-algebraic.
\end{propo}

\begin{propo}[{\cite[Proposition 2.2.7]{BochnakCosteRoy_RealAlgebraicGeometry}}]
\label{propo_ImSemiAlgMap}
Let $f \from X \to Y$ be a semi-algebraic map. 
If $S \subseteq X$ is semi-algebraic, then so is its image $f(S)$. 
If $T \subseteq Y$ is semi-algebraic, then so is its preimage $f^{-1}(T)$.
\end{propo}

Using the order on $\F$ we can define a topology on $\F$, where a basis is given by the open intervals.
Note that if $\F\neq \R$ then $\F$ is totally-disconnected.
However we have the following notion of connectedness for semi-algebraic sets.

\begin{dfn} \label{dfn_SemiAlgConn}
Let $\F$ be a real closed field.
A semi-algebraic set $X \subseteq \F^d$ is \emph{semi-algebraically connected} if it cannot be written as the disjoint union of two non-empty semi-algebraic subsets of $\F^d$ both of which are closed in $X$.
\end{dfn}

\begin{theor}[{\cite[Theorem 2.4.5]{BochnakCosteRoy_RealAlgebraicGeometry}}]
\label{thm_RConnSemiAlgConn}
A semi-algebraic set of $\R^d$ is connected if and only if it is semi-algebraically connected.
Every semi-algebraic set of $\R^d$ has a finite number of connected components, which are semi-algebraic.
\end{theor}

From now let $\K$ be a real closed field extension of $\F$.

\begin{dfn}\label{dfn_ExtSemiAlgSets}
Let $X \subseteq \F^d$ be a semi-algebraic set given as
\[ X = \bigcup_{i=1}^s\bigcap_{j=1}^{r_i} \{x \in \F^d \mid f_{ij}(x) \ast_{ij} 0 \},\]
with $f_{ij} \in \F[X_1,\ldots,X_d]$ and $\ast_{ij}$ is either $<$ or $=$ for $i = 1, \ldots,s$ and $j=1,\ldots,r_i$.
The \emph{$\K$-extension $X_{\K}$} of $X$ is the set given by the same Boolean combination of sign conditions as $X$, more precisely
\[X_{\K} = \bigcup_{i=1}^s\bigcap_{j=1}^{r_i} \{x \in \K^d \mid f_{ij}(x) \ast_{ij} 0 \}.\]
\end{dfn}

Note that $X_{\K}$ is semi-algebraic and depends only on the set $X$, and not on the Boolean combination describing it, see \cite[Proposition 5.1.1]{BochnakCosteRoy_RealAlgebraicGeometry}.
The proof of this is based on the Tarski--Seidenberg principle.

\begin{theor}[Tarski--Seidenberg principle, {\cite[Theorem 5.2.1]{BochnakCosteRoy_RealAlgebraicGeometry}}]
\label{thm_TarskiSeidenberg}
Let $X \subseteq \F^{d+1}$ be a semi-algebraic set.
Denote by $\pr$ the projection $\F^{d+1} \to \F^n$ onto the first $d$ coordinates.
Then $\pr(X) \subseteq \F^d$ is semi-algebraic.
Furthermore, if $\K$ is a real closed extension of $\F$, and $\pr_\K \from \K^{d+1} \to \K$ is the projection on the first $d$ coordinates,  then
\[\pr_\K(X_\K) = (\pr(X))_\K.\]
\end{theor}

Using this one can prove an extension theorem for semi-algebraic maps.

\begin{theor}[{\cite[Propositions 5.3.1, 5.3.3]{BochnakCosteRoy_RealAlgebraicGeometry}}]
\label{thm_ExtSemiAlgMaps}
Let $X \subseteq \F^d$ and $Y \subseteq \F^m$ be two semi-algebraic sets, and $f \from X \to Y$ a semi-algebraic map.
Then $(\textnormal{Graph}(f))_\K$ is the graph of a semi-algebraic map $f_\K \from X_\K \to Y_\K$, that is called the \emph{$\K$-extension} of $f$.
Furthermore, $f$ is injective (respectively surjective, respectively bijective) if and only if $f_\K$ is injective (respectively surjective, respectively bijective),  and f is continuous if and only if $f_\K$ is continuous.
\end{theor}

Finally, we have the following relation between extension of semi-algebraic sets and semi-algebraically connected components.

\begin{theor}[{\cite[Proposition 5.3.6 (ii)]{BochnakCosteRoy_RealAlgebraicGeometry}}]
\label{thm_ExtConnComp}
Let $X \subseteq \F^d$ be semi-algebraic.
Then $X$ is semi-algebraically connected if and only if $X_\K$ is semi-algebraically connected. 
More generally, if $C_1, \ldots, C_m$ are the semi-algebraically connected components of $X$, then $(C_1)_\K, \ldots, (C_m)_\K$ are the semi-algebraically connected components of $X_\K$.
\end{theor}

\subsection{Real spectrum of a ring}

In this section we introduce the notion of the real spectrum of a ring.
We follow \cite[Chapters 7.1 and 7.2]{BochnakCosteRoy_RealAlgebraicGeometry}.
Let $A$ be a commutative ring with $1$.
In our examples,  $A$ will often be a polynomial ring,  compare \Cref{subs_RSCSemiAlgebraicSets}.

\begin{dfn}
\label{dfn_RealSpectrum}
The \emph{real spectrum} $\rsp(A)$ is the topological space
\begin{align*}
\rsp(A) = \{ (\mathfrak{p},\leq) \mid \mathfrak{p} \subseteq &A \textrm{ prime ideal }, \\
&\leq \textrm{ is an order on the fraction field } \textnormal{Frac}(A/\mathfrak{p}) \}
\end{align*}
together with the following subbasis of open sets: For $a \in A$ let
\[ \mathcal{U}(a) \coloneqq \{ (\mathfrak{p}, \leq) \in \rsp(A) \mid \overline{a}^\mathfrak{p} >0\}, \]
where $\overline{a}^\mathfrak{p}$ is the image of $a$ in $\textnormal{Frac}(A/\mathfrak{p})$ under the homomorphism
\[ A \to A/\mathfrak{p} \to \textnormal{Frac}(A/\mathfrak{p}).\]
This topology is called the \emph{spectral topology} on $\rsp(A)$.
\end{dfn}

Contrary to the spectrum of a ring (where we consider all prime ideals), we restrict our attention here to so called \emph{real ideals}, which motivates the name of the real spectrum.
An ideal is called real if, whenever $a_1^2 + \ldots +a_k^2 \in I$ for some $a_1,\ldots, a_k \in A$, we have $a_i \in I$ for all $i=1,\ldots k$, see \cite[Definition 4.1.3]{BochnakCosteRoy_RealAlgebraicGeometry}.
By \cite[Lemma 4.1.6]{BochnakCosteRoy_RealAlgebraicGeometry} we see that a prime ideal $I \subseteq A$ is real if and only if the fraction field of $A/I$ is orderable.

If $k$ is a field,  the real spectrum $\rsp(k)$ of $k$ is homeomorphic to the set of orders on $k$ together with the Harrison topology \cite[Example 7.1.4 a)]{BochnakCosteRoy_RealAlgebraicGeometry}.
It is non-empty if and only if $k$ is orderable.
The real spectrum of $\Z$ is one point corresponding to the zero prime ideal $(0)$ and the unique order on the fraction field $\Q$ of $\Z$.

\begin{examp}[{\cite[Example 7.1.4 (b), 7.5.2]{BochnakCosteRoy_RealAlgebraicGeometry}}]
\label{examp_RSpR}
Let $\R[X]$ be the polynomial ring in one variable.
Since $\R$ is real closed and $\R[X]$ is a principal ideal domain, all non-zero prime ideals are generated by an irreducible polynomial which is either of degree one or two.
Irreducible polynomials of degree one correspond to maximal ideals, hence to elements of $\R$, with residue field $\R$.
Since $\R$ is real closed it has a unique order.
Irreducible polynomials of degree two correspond to algebraic extensions of $\R$ of degree two, which are algebraically closed, and hence not orderable.
We also need to describe the orders on the field of rational function $\R(X)$, which is the fraction field of $\R[X]$ corresponding to the prime ideal $(0)$.
It suffices to order the variable $X$ with respect to $\R$.
More precisely, we define the following orders.
For $\lambda \in \R$ we set
\begin{itemize}
\item $\lambda^+$: We have $\lambda <_{\lambda^+} X$, but $X<_{\lambda^+} \mu$ for every $\mu \in \R$ with $\lambda <\mu$.
\item $\lambda^-$: We have $X <_{\lambda^-}  \lambda$, but $\mu<_{\lambda^-} X$ for every $\mu \in \R$ with $\mu<\lambda$.
\item $+\infty$: We have $\mu <_{+\infty} X$ for all $\mu \in \R$.
\item $-\infty$: We have $X <_{-\infty} \mu$ for all $\mu \in \R$.
\end{itemize}
It turns out that this list is the complete set of total orders on $\R[X]$.
Putting everything together we obtain
\[\rsp(\R[X]) =\{ (\langle X-\lambda \rangle, \leq_\R) \mid \lambda \in \R\} \cup \{((0),\leq_o) \mid o \in \{\lambda^{\pm}, \pm \infty\}\}.\]
\end{examp}

There is an equivalent way of viewing points in the real spectrum.

\begin{propo}[{\cite[Proposition 7.1.2]{BochnakCosteRoy_RealAlgebraicGeometry}}]
\label{propo_RealSpecHomom}
Points $(\mathfrak{p},\leq) \in \rsp(A)$ are in bijection with equivalence classes of ring homomorphisms $[\varphi \from A \to \F]$ for a real closed field $\F$, where we consider the equivalence relation generated by proclaiming two homomorphisms $\varphi \from A \to \F$ and $\varphi' \from A \to \F'$ to be \emph{equivalent} if there exists an order-preserving field homomorphism $\F \to \F'$ such that the diagram commutes:
\[
\begin{tikzcd}
A \arrow[r,"\varphi"] \arrow[rd,swap,"\varphi'"] & \F \arrow[d]\\
& \F'
\end{tikzcd}
\]
More precisely,  given $(\mathfrak{p},\leq)$ we get a homomorphism from $A$ into the real closure of the residue field $\textnormal{Frac}(A/\mathfrak{p})$ with the order $\leq$ by composing the following maps
\[\varphi \from A \to A/\mathfrak{p} \hookrightarrow \textnormal{Frac}(A/\mathfrak{p}) \hookrightarrow \overline{(\textnormal{Frac}(A/\mathfrak{p}),\leq)}^r.\]
On the other hand, given $[\varphi \from A \to \F]$ for a real closed field $\F$, take $(\ker(\varphi),\leq)$ with the restriction of the order of $\F$ to $\textnormal{Frac}(A/\ker(\varphi))$.
Under this identification we have that for $a \in A$
\[\mathcal{U}(a) = \{ [\varphi \from A \to \F] \mid \F \textrm{ real closed}, \, \varphi(a) > 0 \}.\]
\end{propo}

\begin{theor}[{\cite[Proposition 7.1.12]{BochnakCosteRoy_RealAlgebraicGeometry}}]
\label{thm_RSCTopolProp}
The topological space $\rsp(A)$ with its spectral topology is compact (but not necessarily Hausdorff).
\end{theor}

In order to define the real spectrum compactification of semi-algebraic sets,  we need the following.
\begin{dfn}
A subset of $\rsp(A)$ is called \emph{constructible} if it can be obtained as a Boolean combination, i.e.\ finite unions, finite intersections and complements, from the sets $\mathcal{U}(a)$ defined above.
\end{dfn}

\begin{propo}[{\cite[Proposition 7.1.25 (ii)]{BochnakCosteRoy_RealAlgebraicGeometry}}]
\label{propo_RSpClHausdorff}
Let $C \subseteq \rsp(A)$ be a constructible subset endowed with the subspace topology of the spectral topology.
The topological space $C^\cl$ of closed points of $C$ is a compact Hausdorff space.
In particular $\rsp^\cl(A)$, the set of closed points of $\rsp(A)$,  is a compact Hausdorff space.
\end{propo}

\begin{examp}[continuation of \Cref{examp_RSpR}]
\label{examp_RSpRCont}
The points $((0),\pm \infty) \in \rsp(\R[X])$ are closed.
The points $((0),\lambda^\pm) \in \rsp(\R[X])$ for $\lambda \in \R$ are not closed.
We have $\overline{\{\lambda^\pm\}}=\{\lambda^\pm,\lambda\}$.
The space $\rsp^\cl(\R[X])$ is homeomorphic to the two point compactification of $\R$, i.e.\ to the closed interval $[0,1]$. 
\end{examp}

We have the following characterization of closed points of $\rsp(A)$.

\begin{dfn}
Let $A \subseteq A' \subseteq \F$ be two subsets of an ordered field $\F$.
We say that $A'$ is \emph{Archimedean over $A$} if for all $a' \in A$ there exists $a \in A$ with $a' < a$. 
An ordered field is \emph{Archimedean} if it is Archimedean over $\N$.
\end{dfn}

\begin{propo}[{\cite[Proposition 2.2 (e)]{Brumfiel_RSCTeichmullerSpace}}]
\label{propo_ClosedPointsChar}
A point $(\mathfrak{p},\leq) \in \rsp(A)$ is closed if and only if 
$\textnormal{Frac}(A/\mathfrak{p})$ is Archimedean over $\varphi(A)$, where 
\[\varphi \from A \to A/\mathfrak{p} \hookrightarrow \textnormal{Frac}(A/\mathfrak{p}) \hookrightarrow \overline{(\textnormal{Frac}(A/\mathfrak{p}),\leq)}^r\]
is defined as in \Cref{propo_RealSpecHomom}.
\end{propo}

\subsection{Real spectrum compactification of closed semi-algebraic sets}
\label{subs_RSCSemiAlgebraicSets}

We now describe the real spectrum compactification of closed semi-algebraic sets.
Note that these include the algebraic subsets.
Let $A \coloneqq \R[X_1,\ldots,X_d]$ the coordinate ring with coefficients in $\R$.
Note that $A$ is naturally an $\R$-algebra,  hence contains $\R$.

\begin{lem}
\label{lem_VarphiRLin}
Let $\varphi \from A \to \F$,  for $\F$ a real closed field,  represent a point in $\rsp(A)$.
Then $\R \subseteq \F$ and $\varphi$ is $\R$-linear.
\end{lem}
\begin{proof}
The ring homomorphism $\varphi$ restricted to $\R$ is injective (since the only ideals of $\R$ are trivial and $\varphi$ sends one to one),  hence $\R \subseteq \F$ via $\varphi$.
Furthermore,  for $\lambda \in \R$ and $f \in A$, we have $\varphi(\lambda f) =\varphi(\lambda) \varphi(f)$, since $\varphi$ is a ring homomorphism.
\end{proof}

\begin{propo}[{\cite[Proposition 7.1.5]{BochnakCosteRoy_RealAlgebraicGeometry}}]
\label{propo_ImageVInRealSpec}
The map 
\[\Psi \from \R^d \to \rsp(A), \quad v \mapsto (\langle X_1-v_1,\ldots,X_d-v_d\rangle, \leq_\R),\]
where $\leq_\R$ is the unique order on $\R$,  is injective and induces a homeomorphism
from $\R^d$ with its Euclidean topology, onto its image in $\rsp(A)$ with its spectral topology.
\end{propo}

If $X \subseteq \R^d$ is a semi-algebraic set given by a Boolean combination of the basic semi-algebraic sets $\mathcal{B}(f_i)$ for some $f_i \in A$ (\Cref{dfn_SemiAlgSet}),  then we define $\widetilde{X}$ to be the constructible subset of $\rsp(A)$ given by the same Boolean combination of the open sets $\mathcal{U}(f_i)$ (\Cref{dfn_RealSpectrum}).
It turns out that $\widetilde{X}$ is intrinsically defined by the semi-algebraic set $X$ (up to homeomorphism) and does not depend on the ambient space $\R^d$ in which $X$ is embedded,  see \cite[Corollary 7.2.4, Remark 7.2.5]{BochnakCosteRoy_RealAlgebraicGeometry}.

\begin{theor}[{\cite[Theorem 7.2.3]{BochnakCosteRoy_RealAlgebraicGeometry}}]
\label{thm_SemiAlgSetsConstructibleSubsets}
\hfill

\begin{enumerate}
\item There is an isomorphism of Boolean algebras
\begin{align*}
\{\textrm{semi-algebraic subsets of } \R^d\} &\leftrightarrow \{\textrm{constructible subsets of } \rsp(A)\}\\
X &\mapsto \widetilde{X},\\
\Psi^{-1}(\widetilde{X}\cap \Psi(\R^d)) &\mapsfrom \widetilde{X}.
\end{align*}
\item 
\label{item_thm_SemiAlgSetsConstructibleSubsets}
$X$ is closed (resp.\ open) if and only if $\widetilde{X}$ is closed (resp.\ open).
\end{enumerate}
\end{theor}

\begin{propo}[{\cite[Proposition 7.2.7]{BochnakCosteRoy_RealAlgebraicGeometry}}]
\label{propo_CharTilde}
Let $X$ be a closed (respectively open) semi-algebraic subset of $\R^d$.
Then $\widetilde{X}$ is the smallest (respectively largest) closed (respectively open) subset of $\rsp(A)$ whose intersection with $\Psi(\R^d)$ is $\Psi(X)$.
\end{propo}

With this at hand we are now ready to define the real spectrum compactification of closed semi-algebraic sets.

\begin{dfn}
Let $X \subseteq \R^d$ be a closed semi-algebraic subset.
Its \emph{real spectrum compactification} $\rsp(X)$ is the closure of its image 
\[X \subseteq \R^d \xhookrightarrow{\Psi} \rsp(A).\]
\end{dfn}

The definition of the compactification might \`a priori depend on the embedding $X \subseteq \R^d$.
Using \Cref{propo_CharTilde},  one easily sees that if $X \subseteq \R^n$ be a closed semi-algebraic subset of $\R^d$, then $\rsp(X) = \widetilde{X}$. 
This implies that the real spectrum compactification $\rsp(X)$ is intrinsic to $X$,  and does not depend on the embedding $X \subseteq \R^d$.
Note that in \cite[\S 2.7]{BurgerIozziParreauPozzetti_RSCCharacterVarieties2},  they define the real spectrum compactification of a semi-algebraic set by its associated constructible set,  which in the case of closed semi-algebraic sets gives the same definition as here.

From now on let $X \subseteq \R^d$ be a closed semi-algebraic set.
We saw in \Cref{examp_RSpRCont} that in general $X$ is not open in $\rsp(X)$ (not even for algebraic sets).
Moreover,  $\rsp(X)$ with its spectral topology is in general not Hausdorff (\Cref{thm_RSCTopolProp}),  which is something one often desires when compactifying a Hausdorff topological space.
Both of these issues can be resolved when considering the subset of closed points.

\begin{dfn}
Let $\rsp^\cl(X) \coloneqq \widetilde{X}^\cl $ be the subset of closed points in $\rsp(X)$.
\end{dfn}

Note that $\rsp^\cl(X)  = \widetilde{X} \cap \rsp^\cl(A)$, since $\widetilde{X}$ is closed by \Cref{thm_SemiAlgSetsConstructibleSubsets}~(\ref{item_thm_SemiAlgSetsConstructibleSubsets}).

\begin{propo}[{\cite[Corollary 2.32,  Proposition 2.33]{BurgerIozziParreauPozzetti_RSCCharacterVarieties2}}]
\label{thm_CompClosedSemiAlg}
The closed semi-algebraic set $X$ is dense in $\rsp(X)$, and open and dense in $\rsp^\cl(X)$.
\end{propo}
\begin{equation*}
\begin{tikzcd}
& & &\rsp^\cl(X) \arrow[dd,hook]\\
X \arrow[hook]{urrr}[near end]{\textnormal{open and dense}} \arrow[hook]{drrr}[swap]{\textnormal{dense}} &&&\\
&&& \rsp(X)
\end{tikzcd}
\end{equation*}

Let us collect some important properties of the real spectrum compactification of closed semi-algebraic sets.
Recall that $\rsp(X)=\widetilde{X}$.

\begin{propo}[{\cite[Proposition 7.5.1]{BochnakCosteRoy_RealAlgebraicGeometry}}]
\label{propo_RSCConnComp}
Let $X$ be a semi-algebraic set.
Then $X$ is semi-algebraically connected if and only if $\widetilde{X}$ is connected.
Furthermore,  if $X_1,\ldots,X_m$ are the semi-algebraically connected components of $X$, then $\widetilde{X}_1,\ldots,\widetilde{X}_m$ are the connected components of $\widetilde{X}$.
\end{propo}

From this remark we immediately obtain that the real spectrum compactification of a closed connected semi-algebraic set is connected.
The following proposition tells us that semi-algebraic maps extend continuously to the compactification.

\begin{propo}[{\cite[Proposition 7.2.8]{BochnakCosteRoy_RealAlgebraicGeometry}}]
Let $X$ and $Y$ be two semi-algebraic sets and $f \from X \to Y$ a semi-algebraic map.
Then there exists a unique map $\tilde{f} \from \widetilde{X} \to \widetilde{Y}$ such that for all semi-algebraic subsets $Y' \subseteq Y$ we have
\[\tilde{f}^{-1}(\widetilde{Y'}) = \widetilde{f^{-1}(Y')}.\]
If additionally $f$ is a homeomorphism then so if $\tilde{f}$.
\end{propo}

%% file: sections/03_FHit.tex
\section{$\F$-Hitchin component}
\label{section_SemiAlgHitchin}
Let $\F$ be a real closed field.
In this section we describe the semi-algebraic structure of the projective special linear group,  the character variety and the Hitchin component.
This allows to define the $\F$-Hitchin component and $\F$-Hitchin representations.

\subsection{Semi-algebraic structure of $\PSL(n,\F)$}
\label{subsection_SemiAlgPSLn}
The group $\SL(n,\F)$ is naturally an algebraic subset of the set of $(n\times n)$-matrices, denoted $\Mat(n,\F)$, which we can identify with $\F^{n^2}$.

\begin{lem}
The group $\PGL(n,\F)$ is algebraic.
The group $\PSL(n,\F)$ is semi-algebraic.
\end{lem}
\begin{proof}
By the Skolem-Noether theorem, see e.g.\ \cite[Theorem 2.7.2]{GilleSzamuely_CentralSimpleAlgebrasGaloisCohomology}, $\PGL(n,\F)$ is isomorphic to the set of $\F$-algebra automorphisms \linebreak$\Aut(\Mat(n,\F))$ of the set of $(n\times n)$-matrices, and the isomorphism is given by the adjoint representation
\[ \Ad \from \PGL(n,\F) \to \Aut(\Mat(n,\F)), \, [A] \mapsto (M \mapsto AMA^{-1}).\]
Thus $\PGL(n,\F) \cong \Aut(\Mat(n,\F)) \subseteq \GL(n^2,\F)$ is naturally an algebraic subset of $\F^{n^4+1}$, since being an $\F$-algebra automorphism is an algebraic condition.
Hence $\PSL(n,\F)$---the semi-algebraically connected component of $\PGL(n,\F)$ containing the identity---is semi-algebraic; see \Cref{thm_RConnSemiAlgConn} and \Cref{thm_ExtConnComp}.
\end{proof}

As in the real case, we have $\PSL(n,\F)=\SL(n,\F)$ when $n$ is odd, and $\PSL(n,\F)$ is an index two subgroup of $\SL(n,\F)$ when $n$ is even.
Thus the following considerations only concern the case when $n$ is even.

\begin{lem}
\label{lem_ProjMapSemiAlg}
The natural projection map $\SL(n,\F) \to \PSL(n,\F)$ is semi-algebraic.
\end{lem}
\begin{proof}
The projection map is realized by the semi-algebraic map
\[ \Ad \from \SL(n,\F) \to \Aut(\Mat(n,\F)),  \, A \mapsto (M \mapsto AMA^{-1}).\qedhere
\]
\end{proof}

\begin{rem}
\label{rem_ExtensionSL}
If $\F$ is a real closed extension of $\R$, the $\F$-extension of the real semi-algebraic sets $\SL(n,\R)$,  $\PSL(n,\R)$ and $\PGL(n,\R)$ correspond to the groups $\SL(n,\F)$, $\PSL(n,\F)$ respectively $\PGL(n,\F)$.
\end{rem}

\subsection{Semi-algebraic models for character varieties}
\label{subsection_SemiAlgModels}

As in the introduction,  let $S$ be a closed,  connected and oriented surface of genus at least two, and $\pi_1(S)$ its fundamental group.
We denote by $G$ either $\SL(n,\R)$ or $\PSL(n,\R)$.
The real semi-algebraic structure on $G$ endows the space $\Hom(\pi_1(S),G)$ with a real semi-algebraic structure by choosing a finite set of generators for $\pi_1(S)$.
More precisely,  for $F=\{\gamma_1,\ldots,\gamma_{2g}\}$ a finite generating set for $\pi_1(S)$, the map
\[ \ev \from \Hom(\pi_1(S),G) \to G^F, \quad \rho \mapsto (\rho(\gamma_1),\ldots,\rho(\gamma_{2g}))\]
induces a homeomorphism between $\Hom(\pi_1(S),G)$ and its image \linebreak $X_F(\pi_1(S),G)$.
Note that $X_F(\pi_1(S),G)$ is a semi-algebraic subset of some $\R^d$.
If $F'$ is a different choice of a generating set for $\pi_1(S)$,  then $X_F(\pi_1(S),G)$ and $X_{F'}(\pi_1(S),G)$ are semi-algebraically homeomorphic.
From now on we drop the choice of a generating set from the notation.

A representation $\rho \from \pi_1(S) \to G$ is \emph{reductive} if,  seen as a linear representation on $\R^m$,  it is completely reducible, i.e.\ a direct sum of irreducible representations.
Similarly,  if $\F$ is a real closed field extension and $G_\F$ the $\F$-extension of $G$ (\Cref{dfn_ExtSemiAlgSets}),  we say that $\rho \from \Gamma \to G_\F$ is reductive if, seen as a linear representation on $\F^n$, it is completely reducible.
Denote the set of reductive homomorphisms from $\pi_1(S)$ to $G$ by $\Hom_\red(\pi_1(S),G)$.
Any irreducible representation is reductive.
It follows from \cite[\S 20,  p.\ 376 Corollaire a)]{Bourbaki_AlgChap8} that reductive representations are up to conjugation in $\GL(n,\R)$ determined by their trace function.
The same holds true for reductive representations into $\GL(n,\F)$,  where $\F$ is a real closed field.

Denote by $X^\red(\pi_1(S),G)$ the image under the map $\ev$ of the space of reductive homomorphisms $\Hom_\red(\pi_1(S),G)$.
By the theory of Richardson-Slodowy \cite[Section 4]{RichardsonSlodowy_MinimumVectorsForRealReductiveGroups}, $X^\red(\pi_1(S),G)$ is real semi-algebraic as a subset of $X(\pi_1(S),G)$. 
From now on we will identify $\Hom(\pi_1(S),G)$ and $\Hom_\red(\pi_1(S),G)$ with their respective images under the map $\ev$.
Let $\F$ be a real closed field containing $\R$.

\begin{rem}
The adjoint representation $\Ad$ in the proof of \Cref{lem_ProjMapSemiAlg} induces a continuous semi-algebraic map, which we also denote by $\Ad$, between the spaces of homomorphisms
\[\Ad \from \Hom(\pi_1(S),\SL(n,\F)) \to \Hom(\pi_1(S),\PSL(n,\F)).\]
The same is true if we restrict to the set of reductive homomorphisms.
\end{rem}

The group $G$ acts on $\Hom(\pi_1(S),G)$ by conjugation, i.e.\ for all $\rho \in \Hom(\pi_1(S),G)$ and $g \in G$ we have
\[(g.\rho)(\gamma) \coloneqq g \rho(\gamma) g^{-1} \, \textrm{ for all } \gamma \in \pi_1(S).\]
The subset of reductive homomorphisms is invariant under the action of $G$ on $\Hom(\pi_1(S),G)$ by conjugation, which allows us to define the following.

\begin{dfn}
The \emph{character variety} is the topological quotient
\[\chi(S,G)\coloneqq \Hom_\red(\pi_1(S),G)/G,\]
where $G$ acts on $\Hom_\red(\pi_1(S),G)$ by conjugation.
\end{dfn}

We would like to describe the semi-algebraic structure of $\chi(S,G)$.

\begin{dfn}
\label{dfn_SemiAlgModel}
A \emph{semi-algebraic model for $\chi(S,G)$} is a semi-algebraic, continuous map
\[ p \from \Hom_\red(\pi_1(S),G) \to \R^l\]
for some $l \in \N$, such that the fibers over the image of $p$ are exactly the $G$-orbits, and $p$ induces a homeomorphism $\Hom_\red(\pi_1(S),G)/G \xrightarrow{\sim} \textrm{Im}(p) \subseteq \R^l$.
Note that $ \textrm{Im}(p) $ is semi-algebraic, see \Cref{propo_ImSemiAlgMap}.
\end{dfn}

A semi-algebraic model for the character variety has been given in \cite[Theorem 7.6]{RichardsonSlodowy_MinimumVectorsForRealReductiveGroups}.
For more details on this construction,  especially in the context of character varieties,  we recommend \cite[\S 7.2]{BurgerIozziParreauPozzetti_RSCCharacterVarieties2}.

\begin{lem}
\label{lem_UniqueSemiAlgModel}
A semi-algebraic model for $\chi(S,G)$ is unique up to semi-algebraic homeomorphism,  that means that if $p \from \Hom_\red(\pi_1(S),G) \to \R^d$ and $p' \from \Hom_\red(\pi_1(S),G) \to \R^{d'}$ are two such models, then there exists a semi-algebraic homeomorphism $\bar{p} \from \textrm{Im}(p') \to \textrm{Im}(p)$ with $\bar{p}\circ p' = p$.
\end{lem}
\begin{proof}
Assume we have two semi-algebraic models $p$ and $p'$ for $\chi(S,G)$.
The map $p$ factors through $\textrm{Im}(p')$, since both $p$ and $p'$ have the same fibres, and induces a bijective map $\bar{p} \from \textrm{Im}(p') \to \textrm{Im}(p)$.
Note that $\bar{p}$ is continuous,  since $p'$ induces a homeomorphism between $\Hom_\red(\pi_1(S),G)/G$ and $\textrm{Im}(p')$.
We claim that $\bar{p}$ is also semi-algebraic.
We have\begin{align*}
\textrm{Graph}(\bar{p}) &= \{(x,y) \in \textrm{Im}(p') \times \textrm{Im}(p) \mid \bar{p}(x)=y\}\\
&=\{(x,y) \in \textrm{Im}(p') \times \textrm{Im}(p) \mid \exists\, c \in \Hom_\red(\pi_1(S),G) \\
&\quad\quad\quad\quad \textrm{ such that } p'(c)=x,\, p(c)=y\},
\end{align*}
which by Tarski--Seidenberg (\Cref{thm_TarskiSeidenberg}), is a semi-algebraic set since both $p$ and $p'$ are semi-algebraic.
Reversing the roles of $p$ and $p'$ in the above argument proves the claim.
\end{proof}

If $\mathcal{C} \subseteq \Hom(\pi_1(S),G)$ is a $G$-invariant connected component---hence semi-algebraic by \Cref{thm_RConnSemiAlgConn}---consisting of reductive homomorphisms, we can define in an analogous way a semi-algebraic model for the quotient $\mathcal{C}/G$.
Indeed, the restriction to $\mathcal{C}$ of the semi-algebraic model for the whole character variety coming from \cite[Theorem 7.6]{RichardsonSlodowy_MinimumVectorsForRealReductiveGroups} provides a semi-algebraic model for $\mathcal{C}/G$.
In general we can find different such models, which are related by semi-algebraic homeomorphisms, see \Cref{lem_UniqueSemiAlgModel}.
In fact, for connected components of geometric significance,  there are often several semi-algebraic models that exploit their geometric interpretation.
For example in \Cref{section_VariantOfBDCoordinates} we introduce a variant of the Bonahon--Dreyer coordinates for the Hitchin component, which generalize the shear coordinates for Teichm\"uller space (the case of $\PSL(2,\R)$) developed by Thurston \cite[Section 9]{Thurston_MinimalStretchMaps} and \cite[Theorem A]{Bonahon_ShearingHyperbolicSurfaces}.

\begin{lem}
\label{lem_ExtSemiAlgModel}
Let $p \from \mathcal{C} \to \R^d$ be a semi-algebraic model for $\mathcal{C}/G$, with image $p(\mathcal{C})$.
For a real closed field extension $\F$ of $\R$,  we can consider the $\F$-extension of this model,  i.e.\
\[ p_\F \from \mathcal{C}_\F \to \F^d\]
with image $p(\mathcal{C})_\F=p_\F(\mathcal{C}_\F)$.
Then the fibres over $p(\mathcal{C})_\F$ are exactly the $G_\F$-orbits.
\[
\begin{tikzcd}[row sep=large, column sep = large]
\mathcal{C} \arrow[d,swap,twoheadrightarrow, "p"] \arrow[r,swap, hook,  "\F\textrm{-extension}"]  &  \mathcal{C}_\F \arrow[d,twoheadrightarrow, "p_\F"] \\
p(\mathcal{C}) \arrow[r,swap, hook, "\F\textrm{-extension}"] & p(\mathcal{C}) _\F  
\end{tikzcd}
\]
\end{lem}
\begin{proof}
For $\rho \in \mathcal{C}$ the fibre over $p(\rho) \coloneqq [\rho]$ is exactly its $G$-orbit, since $p$ is a semi-algebraic model.
Hence
\begin{align*} 
p^{-1}([\rho]) &= \{ \rho' \in \mathcal{C} \mid p(\rho)=p(\rho')\}\\
&= \{ \rho' \in \mathcal{C} \mid \exists\, g \in G : \rho=g \rho' g^{-1}\}\\
&= \pr_{\mathcal{C}}\left(\{(\rho',g) \in \mathcal{C} \times G \mid \rho = g \rho' g^{-1} \} \right),
\end{align*}
which is semi-algebraic as a projection of a semi-algebraic set by the Tarski--Seidenberg principle (\Cref{thm_TarskiSeidenberg}).
Thus the $p_\F$-fibre over $p(\mathcal{C})_\F$ is a $G_\F$-orbit; see also \Cref{thm_ExtSemiAlgMaps}.
\end{proof}

\subsection{$\F$-Hitchin component and representations}
\label{subsection_FHit}
Let us now focus our attention on the Hitchin component.

\begin{dfn}
\label{dfn_HitchinHom}
We define $\Hom_\Hit(\pi_1(S),\PSL(n,\R))$ as the connected component of $\Hom(\pi_1(S),\PSL(n,\R))$ that contains $\iota_n\circ j$, where $\iota_n$ and $j$ are defined as in \Cref{dfn_HitchinComp}.
\end{dfn}

\begin{rem}
With this definition we have that
\[\Hit(S,n) = \Hom_\Hit(\pi_1(S),\PSL(n,\R))/\PSL(n,\R),\]
 and a Hitchin representation as defined in \Cref{dfn_HitchinComp} is the same as an element of $\Hom_\Hit(\pi_1(S),\PSL(n,\R))$.
The set $\Hom_\Hit(\pi_1(S),\PSL(n,\R))$ is real semi-algebraic as it is by definition a connected component of the semi-algebraic set $\Hom(\pi_1(S),\PSL(n,\R))$,  see \Cref{thm_RConnSemiAlgConn}.
\end{rem}

Let $\F$ be a real closed field containing $\R$.
Recall from \Cref{dfn_FHitchin} that a representation from $\pi_1(S)$ to $\PSL(n, \F)$ is $\F$-Hitchin if its $\PSL(n,\F)$-conjugacy class lies in the $\F$-extension of the semi-algebraic set $\Hit(S,n)$.

\begin{dfn}
Let $\Hom_\Hit(\pi_1(S),\PSL(n,\F))$ be the $\F$-extension of \linebreak $\Hom_\Hit(\pi_1(S),\PSL(n,\R))$.
\end{dfn}

\begin{lem}
\label{lem_FHitchinHomHit}
The set of $\F$-Hitchin representations is equal to \linebreak$\Hom_\Hit(\pi_1(S),\PSL(n,\F))$.
Furthermore,  $\Hom_\Hit(\pi_1(S),\PSL(n,\F))$ is the semi-algebraically connected component of $\Hom(\pi_1(S), \PSL(n,\F))$ that contains $\iota_n \circ j$ as defined in \Cref{dfn_HitchinComp}.
\end{lem}
\begin{proof}
The first statement follows from \Cref{lem_ExtSemiAlgModel}.
The second statement follows since the extension of semi-algebraic sets to $\F$ preserves semi-algebraically connected components, see \Cref{thm_RConnSemiAlgConn} and \Cref{thm_ExtConnComp}.
\end{proof}

All of these different points of view are helpful in the following and we might confound a representation with its conjugacy class.

%% file: sections/charactervarietieshigherTMspaces.tex
\subsection{Real spectrum compactification of the Hitchin component}
\label{subs_RSCCharVar}

We have seen in \Cref{subs_RSCSemiAlgebraicSets} how to define the real spectrum compactification of a closed semi-algebraic set,  which we now would like to apply in the context of character varieties and the Hitchin component.

Let $G$ be either $\SL(n,\R)$ or $\PSL(n,\R)$.
In \Cref{subsection_SemiAlgModels},  we alluded to the fact that there exists a semi-algebraic model $p$ for $\chi(S, G)$,  that identifies $\chi(S, G)$ with a closed semi-algebraic subset $Y(S,G)\coloneqq \Im(p)$ of some $\R^d$.
We can embed $\chi(S, G)$ in the compact space
\[\rsp(\chi(S, G)) \coloneqq \widetilde{Y(S, G)} \subseteq \rsp(\R[X_1,\ldots,X_d])\]
with dense image,  and hence the latter provides a compactification of $\chi(S, G)$.

To get a better description of the points in $\chi(S,G)$,  we need to introduce some notation.
Recall that if $\F$ is a real closed field,  we write $G_\F$ for the $\F$-extension of $G$,  which agrees with either  $\SL(n,\F)$ or $\PSL(n,\F)$,  see \Cref{rem_ExtensionSL}.

\begin{dfn}
\label{dfn_RhoMinField}
Let $\F$ be a real closed field and $\rho \from \pi_1(S) \to G$ a homomorphism.
Then $\F$ is \emph{$\rho$-minimal} if there is no proper real closed subfield $\K \subseteq \F$ such that $\rho$ is conjugate into $G_\K$ by an element of $G_\F$.
\end{dfn}

\begin{propo}[{\cite[Corollary 7.9]{BurgerIozziParreauPozzetti_RSCCharacterVarieties2}}]
If $\rho \from \pi_1(S) \to G_\F$ is reductive, then there exists a unique real closed subfield $\F_\rho \subseteq \F$ such that $\rho$ is $G_\F$-conjugate to a representation $\rho' \from \pi_1(S) \to G_{\F_\rho}$ and $\F_\rho$ is $\rho'$-minimal. 
\end{propo}

%


\begin{theor}[{\cite[Theorem 1.1]{BurgerIozziParreauPozzetti_RSCCharacterVarieties2}}]
\label{thm_RSCCharBoundary}
Let $G$ be either $\SL(n,\F)$ or $\PSL(n,\F)$.
Then
\[
\rsp(\chi(S, G))\cong \left\{ (\rho, \F_\rho) \ \middle\vert 
\begin{array}{l}
	\rho \from \pi_1(S) \to  G_{\F_\rho} \textrm{ reductive,}\\
\F_\rho \supseteq \R \textrm{ real closed and } \rho\textrm{-minimal}
\end{array}
\right\}_{\Big/ \sim} \,,
\]
where two representations $\rho_1 \from \Gamma \to G_{\F_{\rho_1}}$ and $\rho_2 \from \Gamma \to G_{\F_{\rho_2}}$ are \emph{equivalent} ($\sim$) if there exists an order-preserving isomorphism $i \from \F_{\rho_1} \to \F_{\rho_2}$ such that $\rho_2$ and $i \circ \rho_1$ are $G_{\F_{\rho_2}}$-conjugate.
Points in the boundary correspond to representations into $ \PSL(n,\F_\rho)$ for $\F_\rho$ a non-Archimedean field.
Moreover $\F_\rho$ is of finite transcendence degree over $\R$.
\end{theor}

The Hitchin component $\Hit(S,n) \subseteq \chi(S,\PSL(n,\R))$ is again closed semi-algebraic as it is a connected component of a closed semi-algebraic set.
We denote by $\rsp(\Hit(S,n))$ its real spectrum compactification,  i.e.\ the closure of $\Hit(S,n)$ in $\rsp(\chi(S,\PSL(n,\R)))$.
%
Recall from \Cref{dfn_FHitchin} that a representation $\rho \from \pi_1(S) \to \PSL(n,\F)$ is $\F$-Hitchin if its $\PSL(n,\F)$-equivalence class lies in the $\F$-extension of $\Hit(S,n)$.
\begin{theor}[{\cite[Theorem 46]{BurgerIozziParreauPozzetti_RSCCharacterVarieties}}]
\label{thm_RSpHitchinComponent}
\[
\rsp(\Hit(S,n))\cong \left\{ (\rho, \F_\rho) \ \middle\vert 
\begin{array}{l}
	\rho \from \pi_1(S) \to \PSL(n,\F_\rho) \textrm{ is } \F_\rho\textrm{-Hitchin,} \\
	\F_\rho \supseteq \R \textrm{ real closed and } \rho\textrm{-minimal}
\end{array}
\right\}_{\Big/ \sim} \, ,
\]
for the same equivalence relation as in the theorem before.
%
%
Furthermore, $\F_\rho = \overline{\R(\tr(\Ad(\rho)))}^r$,  where $\Ad$ is the adjoint representation of $\PSL(n,\F_\rho)$.
In addition, $(\rho,\F_\rho)$ represents a closed point if and only if $\F_\rho$ is Archimedean over the ring of traces $\R[\tr(\Ad(\rho))]$ of $\Ad \circ \rho$.
%
%
\end{theor}

Loosely speaking,  the goal of the following chapters is to replace the word ``$\F_\rho$-Hitchin'' by the word ``$\F_\rho$-positive'' in this theorem.

%% file: sections/04_flags.tex
\section{Flags} \label{section_PreliminariesFlags}
Let $\F$ be any field.
A \emph{(full) flag $E$ in $\F^n$} is an increasing sequence of subspaces in the finite-dimensional $\F$-vector space $\F^n$, i.e.\
\[E=\left( E^{(1)} \subset \ldots \subset E^{(n-1)} \right),\]
such that $\dim\left(E^{(a)}\right)=a$ for all $a=1, \ldots, n-1$.
Given a flag $E$ we will use the notation $E^{(a)}$ to denote the $a$-dimensional subspace of $\F^n$ defined by $E$.
In this article we are only concerned with full flags, and we will omit the word \emph{full} in the following when referring to full flags in $\F^n$.
The natural action of the general linear group $\GL(n,\F)$ on $\F^n$ induces an action on the set of flags $\Flag(\F^n)$, which descends to a transitive action of the projective linear group $\PGL(n,\F)$ on $\Flag(\F^n)$.

\begin{dfn}\label{dfn_TransverseFlags}
A $k$-tuple $(E_1,\ldots, E_k)$ of flags in $\F^n$ is called \emph{transverse} if for every $a_1, \ldots, a_k \in \{0, \ldots, n\}$ with $\sum_{i=1}^k a_i = n$
\[ E_1^{(a_1)} + \ldots + E_k^{(a_k)} = \F^n. \]
The space of $k$-tuples of transverse flags will be denoted by $\Flag(\F^n)^{(k)}$.
\end{dfn}

We observe that $\PGL(n,\F)$ acts transitively on pairs of transverse flags.
As soon as we consider $k>2$ there is a whole configuration space of $k$-tuples of transverse flags up to $\PGL(n,\F)$-action.
Let us begin by considering the case $k=3$.

\begin{propo} \label{propo_StabilizerTripleFlags}
Let $(E,F,G) \in \Flag(\F^n)^{(3)}$ be a transverse triple of flags.
Then $\Stab_{\PGL(\F^n)}(E,F,G) = \{ \Id_{\PGL(\F^n)} \}$.
\end{propo}
\begin{proof}
By the transversality of the triple, we can choose a basis $e'_1,\ldots, e'_n$ of $\F^n$ such that for all $a=0,\ldots,n$
\[ E^{(a)} = \langle e'_1, \ldots, e'_a \rangle, \textrm{ and } F^{(a)} = \langle e'_n, \ldots, e'_{n-a+1} \rangle.\]
Let $0 \neq g \in G^{(1)}$ be a generator, and write $g = \sum_{i=1}^n g_i e'_i$ for some $g_i \in \F$.
Again by transversality, $g_i \neq 0$ for all $i=1,\ldots,n$, and we set $e_i \coloneqq \tfrac{1}{g_i} e'_i$.
Then $e_1, \ldots, e_n$ is a basis of $\F^n$ such that $g= e_1 + \ldots + e_n$.
Let now $\varphi \in \GL(\F^n)$ be in the stabilizer of $(E,F,G)$.
The matrix $M$ representing $\varphi$ in the basis $e_1,\ldots,e_n$ is diagonal.
Since $M$ maps the vector $e_1+\ldots + e_n$ to a non-trivial multiple of itself, it follows that there exists $0 \neq \alpha \in \F$ such that $M = \textrm{diag}(\alpha,\ldots,\alpha)$, and thus $M$ lies in the center of $\GL(n,\F)$.
\end{proof}

To parametrize the configuration space of triples of transverse flags,  Fock--Goncharov \cite{FockGoncharov_ModuliSpacesLocalSystemsHigherTeichmuellerTheory} introduced so-called triple ratios.
The triple ratios are expressed in terms of the exterior algebra $\bigwedge^n \F^n $ of $\F^n$.
If $E$ is a full flag, then for every $a$ between $1$ and $n$ the space $\bigwedge^a E^{(a)}$ is isomorphic to $\F$.
Choose a non-zero element $e^{(a)} \in \bigwedge^a E^{(a)}$.
We use the same notation to denote its image in $\bigwedge^a \F^n $.
The following definition is independent of the choices of $e^{(a)} \in \bigwedge^a E^{(a)}$.

\begin{dfn}
\label{dfn_TripleRatio}
Let $(E,F,G)$ be a transverse triple of flags in $\F^n$.
For $a, b, c \in \{1, \ldots,  n-2\}$ with $a+b+c=n$, we define the \emph{$(abc)$-triple ratio $T_{abc}$} of $(E,F,G)$ by
\begin{align*}
T_{abc}(E,F,G)&= 
\frac{e^{(a+1)}\wedge f^{(b)} \wedge g^{(c-1)}}{e^{(a-1)}\wedge f^{(b)} \wedge g^{(c+1)}} \\
&\quad \quad \quad \cdot \frac{e^{(a)}\wedge f^{(b-1)} \wedge g^{(c+1)}}{e^{(a)}\wedge f^{(b+1)} \wedge g^{(c-1)}}\cdot \frac{e^{(a-1)}\wedge f^{(b+1)} \wedge g^{(c)}}{e^{(a+1)}\wedge f^{(b-1)} \wedge g^{(c)}} \in \F.
\end{align*}
\end{dfn}

Note that all of the involved expressions are non-zero by the transversality of the triple.
The triple ratios are invariant under the action of $\PGL(n,\F)$.

Let us now turn our attention to the case $k=4$.
We consider other $\PGL(n,\F)$-invariant rational functions, so-called double ratios, which are, similarly to the triple ratios, expressed in terms of the exterior algebra $\bigwedge^n \F^n $ of $\F^n$, and we keep the notions that were introduced above.

\begin{dfn}
\label{dfn_DoubleRatio}
Let $(E,F,G,H)$ be a transverse quadruple of flags in $\F^n$.
For $a=1,\ldots, n-1$ we define the \emph{$a$-th double ratio $D_a$} of $(E,F,G,H)$ by
\[D_{a}(E,F,G,H)= -
\frac{e^{(a)}\wedge f^{(n-a-1)} \wedge g^{(1)}}{e^{(a)}\wedge f^{(n-a-1)} \wedge h^{(1)}} \cdot \frac{e^{(a-1)}\wedge f^{(n-a)} \wedge h^{(1)}}{e^{(a-1)}\wedge f^{(n-a)} \wedge g^{(1)}} \in \F.
\]
\end{dfn}

We remark that the definition of the double ratios only involves the one-dimensional subspaces of the flags $G$ and $H$.
We summarize Lemmas $5$, $6$ and $7$ from \cite{BonahonDreyer_ParametrizingHitchinComponents}, which relate the different ratios defined above and explain how they behave under permutations of the involved flags.
The proof is a direct computation.

\begin{lem}[{\cite[Lemmas 5, 6, 7]{BonahonDreyer_ParametrizingHitchinComponents}}] \label{lem_PropertiesOfRatiosUnderPermutation}
Let $(E,F,G)$ be a transverse triple of flags in $\F^n$.
Then
\begin{enumerate}
\item \label{lem_PropertiesOfRatiosUnderPermutation_item_TripleRatio}
$T_{abc}(E,F,G) = T_{bca}(F,G,E) = T_{bac}(F,E,G)^{-1}$ for all $a+b+ c=n$, $a,b,c \in \N_{\geq 1}$.
\end{enumerate}
Let $H$ be a forth flag in $\F^n$ such that $(E,F,G,H)$ is a transverse quadruple.
Then
\begin{enumerate}[resume]
\item \label{lem_PropertiesOfRatiosUnderPermutation_item_DoubleRatio1}
$D_a(E,F,H,G)= D_a(E,F,G,H)^{-1}$ for all $a=1, \ldots, n-1$,  and
\item \label{lem_PropertiesOfRatiosUnderPermutation_item_DoubleRatio2}
$D_a(F,E,G,H)= D_{n-a}(E,F,G,H)^{-1}$ for all $a=1, \ldots, n-1$.
\end{enumerate}
\end{lem}

From now on let $\F$ be real closed.

\begin{lem}
\label{lem_FlagSpacesAlg}
The spaces $\Flag(\F^n)$, $\Flag(\F^n)^k$ and $\Flag(\F^n)^{(k)}$ are semi-algebraic.
\end{lem}
\begin{proof}
Denote by $\Gr(k,\F^n)$ the Grassmannian of $k$-dimensional subspaces of $\F^n$.
The product of Grassmannians $\Gr(1,\F^n) \times \ldots \times \Gr(n-1,\F^n)$ is an algebraic set, see for example \cite[Proposition 3.4.4]{BochnakCosteRoy_RealAlgebraicGeometry}.
In the proof the authors define a bijection
\begin{align*}
\Psi_k \from \Gr(k,\F^n) &\to \{ M \in \Mat(n,\F) \mid M^\top =M,\, M^2 = M, \, \tr(M)=k\} \eqqcolon H_k\\
V &\mapsto P_V,
\end{align*}
where $P_V$ is the matrix of the orthogonal projection on $V$ with respect to some scalar product on $\F^n$.
If $V, W \subseteq \F^n$ are two subspaces,   then $V  \subseteq W$ if and only if $P_{W} P_{V} = P_{V}$.
Thus the image of $\Flag(\F^n)$ under the map $\Psi_1 \times \cdots \times \Psi_{n-1}$ is the algebraic set 
\[ \{ (M_1,\ldots,M_{n-1}) \in H_1 \times \cdots \times H_{n-1} \mid M_{i+1}M_i = M_i \textrm{ for all } i=1,\ldots,n-2 \}.\]
Thus  $\Flag(\F^n)$ and  $\Flag(\F^n)^{k}$ are algebraic.

Given a $k$-tuple of flags $(E_1,\ldots,E_k)$ and $a_1,\ldots,a_k \in \{1,\ldots,n-1\}$ with $\sum_{i=1}^k a_k = n$,  then $E_1^{(a_1)} + \ldots + E_k^{(a_k)} = \F^n$ is equivalent to asking that
\[ \textnormal{rank} \left( M_{1,a_1} | \ldots | M_{k,a_k} \right) = n,\]
where each $E_i = (M_{i,1}, \ldots, M_{i,n-1}) \in H_1 \times \ldots \times H_{n-1}$ with $M_{i,j+1}M_{i,j}=M_{i,j}$ for all $i=1,\ldots,k$ and $j=1,\ldots,n-2$.
Asking for the rank of a matrix to be equal to $n$ can be expressed as the existence of a submatrix of size $n\times n$ that has non-zero determinant.
This condition is semi-algebraic by Tarski--Seidenberg (\Cref{thm_TarskiSeidenberg}).
It follows that $\Flag(\F^n)^{(k)}$ is semi-algebraic.
\end{proof}

\begin{propo}
\label{propo_TripleDoubleRatioSemiAlg}
The triple and double ratios
\[T_{abc} \from \Flag(\F^n)^{(3)} \to \F, \quad D_a \from \Flag(\F^n)^{(4)} \to \F\]
are semi-algebraic maps.
\end{propo}
\begin{proof}
We will only prove that $D_a$ is semi-algebraic. 
The proof for $T_{abc}$ is analogous.
Fix $a$ between $1,\ldots,n-1$.
For $(E,F,G) \in \Flag(\F^n)^{(3)}$,  we know by transversality that the $(n \times 3n)$-matrix $M_a \coloneqq (E_{a}|F_{n-a-1}|G_1)$ has rank $n$,  where $E= (E_{1}, \ldots, E_{n-1}) \in H_1 \times \ldots \times H_{n-1}$ with $E_{i+1}E_{i}=E_{i}$ for all $i=1,\ldots,n-2$ (and analogous for $F$ and $G$); compare to the notation in \Cref{lem_FlagSpacesAlg}.
Thus there exist indices $I\coloneqq \{1\leq i_1<\ldots<i_n\leq 3n\}$ such that $\det(M_{a,I}) \neq 0$,  where $M_{a,I}$ denotes the $n\times n$-submatrix of $M_a$ with rows $1,\ldots,n$ and columns $i_1,\ldots,i_n$.
Since $E_a$ has rank $a$,  it follows that the first $a$ indices are in between $1$ and $n$.
Similarly, the next $n-a-1$ indices are between $n+1$ and $2n$, and the last index is between $2n+1$ and $3n$.
Note that $\det(M_{a,I})$ depends on the choice of the indices $I$.
If $(E,F,H) \in \Flag(\F^n)^{(3)}$ is another triple of transverse flags and $N_a \coloneqq (E_{a}|F_{n-a-1}|H_1)$,  then for the same set of indices $I$ as before we have $\det(N_{a,I})\neq 0$.
Furthermore,  the quotient $\tfrac{\det(M_{a,I})}{\det(N_{a,I})}$ does not depend on the choice of $i_1,\ldots, i_{n-1}$.
Indeed,  assume for simplicity we replace $i_1$ by $i_1' \in \{1,\ldots,n\}$ to obtain the index set $I'$ and such that $\det(M_{a,I'})\neq 0$.
Then we can write the $i_1'$-th column as a linear combination of the columns of $E_a$ corresponding to the indices $i_1,\ldots,i_a$. 
Furthermore, the coefficient $\alpha$ of the column corresponding to the index $i_1$ will the non-zero.
Since the determinant is multilinear and alternating in the columns we get $\det(M_{a,I'})=\alpha \det(M_{a,I})$.
Similarly, we have $\det(N_{a,I'})=\alpha \det(N_{a,I})$,  and hence their quotient $\tfrac{\det(M_{a,I})}{\det(N_{a,I})}$ only depends on the choice of $i_n$.

By the same argument we can find indices $J\coloneqq \{1\leq j_1<\ldots<j_n\leq 3n\}$ such that $\det(M_{a-1,J}) \neq 0$ and $j_n=i_n$.
Then $\det(N_{a-1,J})\neq 0$,  and the product of the quotients $\tfrac{\det(M_{a,I})}{\det(N_{a,I})} \cdot \tfrac{\det(N_{a-1,J})}{\det(M_{a-1,J})}$ does not depend on the choices of $I$ and $J$.
Furthermore,  
\[D_a(E,F,G,H) = - \frac{\det(M_{a,I})}{\det(N_{a,I})} \cdot \frac{\det(N_{a-1,J})}{\det(M_{a-1,J})}.\]
Since we expressed the double ratio as a rational function of the quadruple of transverse flags, we conclude that $D_a$ is semi-algebraic.
\end{proof}

%% file: sections/05_positivity.tex
\section{Positivity}
\label{section_Positivity}

\subsection{Positivity of tuples of flags}
The definitions from the last section allow us to define positive triples and quadruples of flags in $\F^n$ for $\F$ any ordered field.
It is defined analogously as in \cite{FockGoncharov_ModuliSpacesLocalSystemsHigherTeichmuellerTheory, BonahonDreyer_ParametrizingHitchinComponents} for $\R$.
We recall it here, and refer to \cite{FockGoncharov_ModuliSpacesLocalSystemsHigherTeichmuellerTheory, BonahonDreyer_ParametrizingHitchinComponents} for more details.

\begin{dfn}\label{dfn_PostivityTuples}
A triple $(E,F,G)$ of flags in $\F^n$ is called \emph{positive}, if the triple is transverse and all triple ratios are positive.
A quadruple $(E,F,G, H)$ of flags in $\F^n$ is called \emph{positive}, if the quadruple is transverse, the triples $(E,F,G)$ and $(E,G,H)$ are positive, and all double ratios of $(E,G,F,H)$ are positive.
\end{dfn}

It is useful in the following to also define a notion of positivity for $k$-tuples of flags with $k\geq 4$ by fixing some additional data, which we describe in the following.
Let $x_1,\ldots, x_k$ be distinct points on the unit circle $\mathbb{S}^1$ that are cyclically ordered in clockwise direction, and $P$ the inscribed polygon that we obtain by connecting consecutive points by a straight line.
An \emph{ideal triangulation} of $P$ is a collection of \emph{oriented diagonals} $\mathcal{E}=\{e_1,\ldots, e_{k-3}\}$, i.e.\ straight lines in $P$ that do not intersect and that connect two non-consecutive vertices,  such that $P\setminus \mathcal{E}$ is a union of $k-2$ triangles, together with a choice of preferred vertex for each triangle $\mathcal{V}=(v_{t_1}, \ldots, v_{t_{k-2}})$ with $v_{t_i} \in \{x_1,\ldots,x_k\}$ for all $i=1,\ldots,k-2$.
A choice of ideal triangulation $(\mathcal{E},\mathcal{V})$ gives the following information.
\begin{itemize}
\item[--] Each connected component $t$ of $P\setminus \mathcal{E}$, together with a choice of preferred vertex $v_t \in \mathcal{V}$, singles out a triple of vertices $(v_t, x_t', x_t'')$ of $P$ that appear in this clockwise order around the circle as the vertices of $t$.
\item[--] Each oriented diagonal $e \in \mathcal{E}$ is contained in the closure of exactly two connected components of $P\setminus \mathcal{E}$ with vertices $x_{e^+}, x_{e^r},  x_{e^-}$ and $x_{e^+},  x_{e^-},  x_{e^l}$ respectively,  and therefore $e$ singles out four vertices $x_{e^+}, x_{e^r},  x_{e^-},  x_{e^l}$ which appear in this clockwise order around the circle.
\end{itemize}
For every choice of ideal  triangulation $(\mathcal{E},\mathcal{V})$ we can define a map
\[ \phi_{(\mathcal{E},\mathcal{V})} \from \Flag(\F^n)^{(k)} \to \F^{\frac{(n-1)(n-2)}{2}(k-2)} \times \F^{(k-3)(n-1)} \]
by assigning to a $k$-tuple $(F_1,\ldots,F_k)$ of transverse flags the following data.
\begin{itemize}
\item[--] For each triangle $t$ of $P\setminus \mathcal{E}$, together with a choice of preferred vertex $v_t \in \mathcal{V}$,  we compute the $\frac{(n-1)(n-2)}{2}$ triple ratios $T_{abc}(F_{v_t}, F_{x_t'}, F_{x_t''})$.
\item[--] For each oriented diagonal $e \in \mathcal{E}$, we compute the $n-1$ double ratios $D_a(F_{x_{e^+}},  F_{x_{e^-}}, F_{x_{e^r}},  F_{x_{e^l}})$,
\end{itemize}
where $F_{x_i} \coloneqq F_i$ for all $i=1,\ldots,k$.
Note that the triple ratios depend on the choice of preferred vertex. 
Any two such choices permute the flags and are thus related by the equalities given in \Cref{lem_PropertiesOfRatiosUnderPermutation} (\ref{lem_PropertiesOfRatiosUnderPermutation_item_TripleRatio}).

\begin{dfn} \label{dfn_kTuplePosFlags}
Let $\mathcal{E}$ be an ideal triangulation of a polygon with $k$ vertices. 
A $k$-tuple of flags $(F_1, \ldots, F_k) \in \Flag(\F^n)^{(k)}$ is \emph{positive} if $ \phi_{(\mathcal{E},\mathcal{V})}(F_1, \ldots,F_k)$ has only positive coordinates.
We denote the space of positive $k$-tuples of flags by $\Flag(\F^n)^{(k,+)}$.
\end{dfn}

At first sight it seems that this definition depends on the choice of triangulation $(\mathcal{E},\mathcal{V})$.
Fock--Goncharov proved in \cite[Theorem 9.1]{FockGoncharov_ModuliSpacesLocalSystemsHigherTeichmuellerTheory} that if a $k$-tuple is positive with respect to one triangulation it is positive with respect to any other.
Their proof gives an explicit rational expression, which preserves positivity, for the change of coordinates under a \emph{flip of a diagonal}, that is, for a pair of adjacent triangles we remove the edge that forms the diagonal and add an edge that forms the other diagonal of that quadrilateral.
Since we can get from any ideal triangulation of a polygon to any other by a sequence of flips of diagonals \cite{Hatcher_TriangulationsSurfaces}, we obtain the desired result.
We also recommend \cite[Section 2.2]{Martone_SequencesHitchinReprTreeType} for a detailed account of how the flip of a diagonal changes the coordinates $\phi_{(\mathcal{E},\mathcal{V})}$.
Since the triple and double ratios are $\PGL(n,\F)$-invariant,  a $k$-tuple is positive if and only if a $k$-tuple in the same $\PGL(n,\F)$-orbit is.
We refer the reader to \cite{FockGoncharov_ModuliSpacesLocalSystemsHigherTeichmuellerTheory} for the more general definition of positivity for $k$-tuples of flags which clarifies the minus sign in the definition of the double ratios.

\begin{rem} \label{rem_ImageOfkCyclicallyOrderedPointsIsPositivekTuple}
In light of the above definition of positive $k$-tuples of flags a positive map $\Fix(S) \to \Flag(\F^n)$ (\Cref{dfn_Fpositivity}) can equivalently be defined as a map that
sends $k$-tuples of distinct points in $\Fix(S)$, occurring in this clockwise order, to positive
$k$-tuples of flags for any $k \geq 2$.
\end{rem}

\subsection{Total positivity}

\begin{dfn} \label{dfn_totpos}
Let $\K$ be an ordered field and $\mathcal{B}$ a basis of $\K^n$.
An element in $\GL(n,\K)$ is \emph{totally positive} with respect to $\mathcal{B}$, if all the minors of its matrix in the basis $\mathcal{B}$ are (strictly) positive.
A unipotent element in $\GL(n,\K)$ is \emph{totally positive upper triangular} with respect to $\mathcal{B}$ if its matrix in the basis $\mathcal{B}$ is upper triangular and all its minors,  which are not constant equal to zero because of the shape of the matrix,  are positive.
\end{dfn}

We denote the set of unipotent,  totally positive upper triangular matrices with respect to $\mathcal{B}$ by $\mathcal{U}^\urtriangle_{>0}(\mathcal{B})$.
Similarly,  we denote the set of unipotent,  totally positive lower triangular matrices with respect to $\mathcal{B}$ by $\mathcal{U}^\lltriangle_{>0}(\mathcal{B})$.

The following lemmas due to Fock--Goncharov shed light on the connection between positivity of tuples of flags and total positivity of matrices.
It can be formulated for $k$-tuples with $k>5$, but for our purposes these versions suffice.
For a more conceptual approach to total positivity in split real Lie groups we refer the reader to \cite{Lusztig_TotalPositivityReductiveGroups}.
Let $\F$ be a real closed field.

\begin{lem}[{\cite[Definition 5.2,  Theorem 5.2, Theorem 5.3,  Lemma 5.5,  Lemma 5.9]{FockGoncharov_ModuliSpacesLocalSystemsHigherTeichmuellerTheory}}]
\label{lem_PoskTupleNormalForm}
Let $k$ be either $0$,  $1$ or $2$.
A $(3+k)$-tuple $(F_1,\ldots,F_{3+k})$ of flags in $\F^n$ is positive if and only if there exists a basis $\mathcal{B}=(e_1,\ldots,e_n)$ of $\F^n$ such that $e_a \in F_1^{(a)} \cap F_3^{(n-a+1)}$ for all $a=1,\ldots,n$,  and there exist $u \in\mathcal{U}^\lltriangle_{>0}(\mathcal{B})$ and $v_1,\ldots, v_{k} \in \mathcal{U}^\urtriangle_{>0}(\mathcal{B})$ such that
\[F_2 = u F_1,  \quad F_i = v_1^{-1}\cdot \ldots \cdot v_{i-3}^{-1} F_3 \textnormal{ for all } i = 4, \ldots, 3+k.\]
\end{lem}

We say that $F_1$ is the \emph{ascending flag} and $F_3$ is the \emph{descending flag} of $\mathcal{B}$.
The following proposition shows that the double ratios mimic the monotonicity behavior of the cross-ratios for $n=2$ in higher dimensional flag varieties.

\begin{propo} \label{propo_5TuplePosInequDR}
Let $(F_1,F_2,F_3,F_4,F_5)$ be a positive 5-tuple of flags in $\F^n$.
Then for all $a=1,\ldots, n-1$ we have the following inequality between the $a$-th double ratios
\[D_a(F_1, F_3, F_2, F_4) < D_a(F_1, F_3, F_2,F_5).\]
\end{propo}

A direct corollary is the monotonicity of the double ratios on the image of a positive map.

\begin{corol}
Let $D \subseteq \mathbb{S}^1$ be a subset and $\xi \from D \to \Flag(\F^n)$ a positive map.
For $x,z \in D$ distinct,  denote by $(x,z)$ the set of points $y \in D$ such $x,y,z$ are cyclically ordered in clockwise direction.
For all $x,z \in D$ distinct and $y \in (x,z)$, the map
\[ (z,x) \to \F, \quad D_a(t) \coloneqq D_a(\xi(x),\xi(z),\xi(y),\xi(t)) \]
is strictly monotone increasing,  i.e.\ for all $t \in (z,x)$ and $t' \in (t,x)$ we have $D_a(t)<D_a(t')$.
\begin{figure}[H]
\centering
\includegraphics[width=0.35\textwidth]{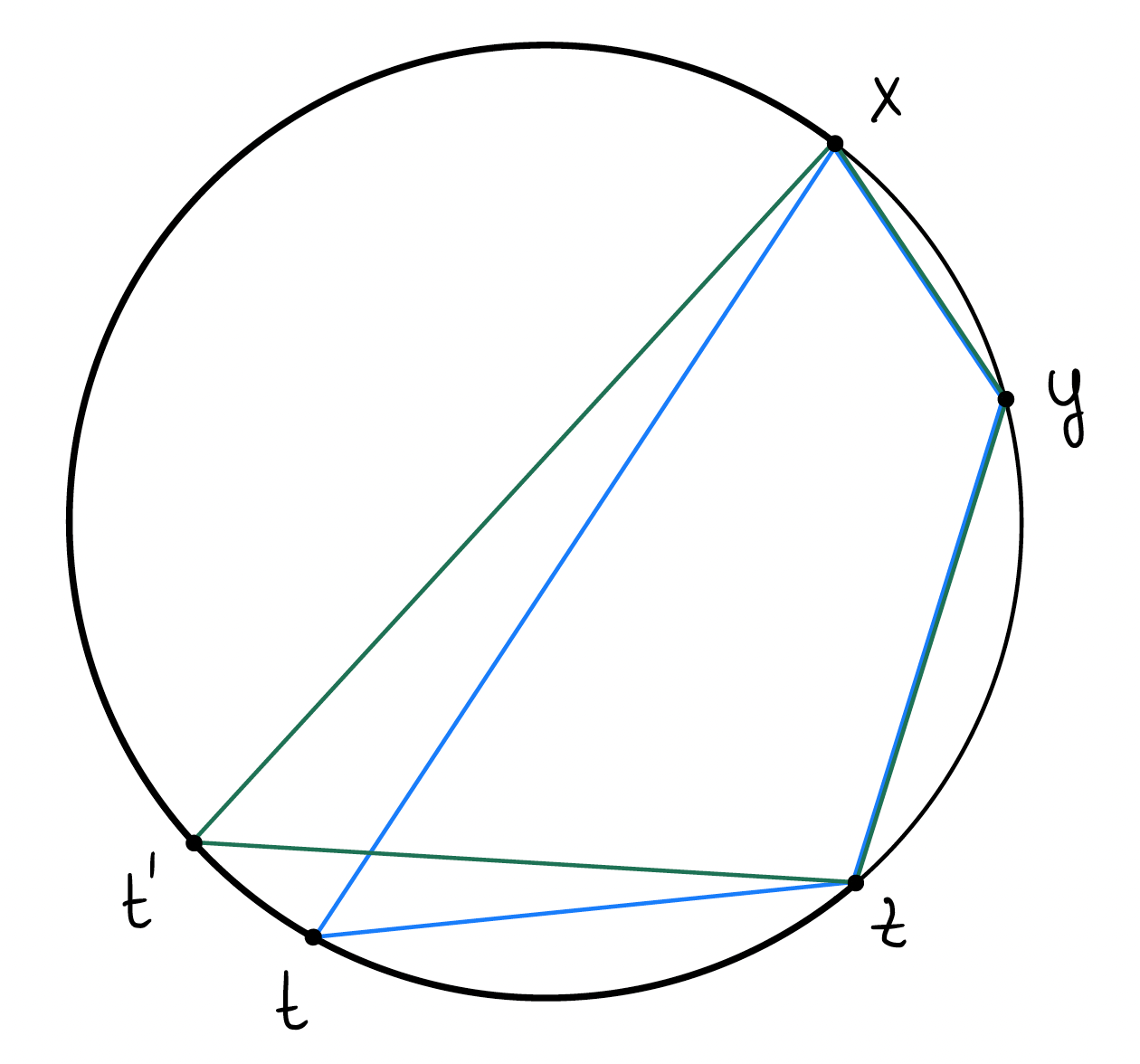}
\captionof{figure}{Monotonicity of the double ratios.} \label{fig_MonotonicityDoubleRatios}
\end{figure}
\end{corol}

Before proving the above proposition we need some preliminary results on totally positive matrices.
We begin by recalling \emph{Cauchy--Binet's formula}, see for example \cite[Chapter 1]{Pinkus_TotallyPositiveMatrices}.
Let $A$ be an $(n\times n)$-matrix over any field.
For $I=\{1 \leq i_1 < \ldots < i_p \leq n\}$ and $J=\{1 \leq j_1 < \ldots < j_p \leq n\}$ denote by $A_{I,J}$ the $(p \times p)$-minor of $A$ obtained by taking the determinant of the submatrix of $A$ with rows $i_1, \ldots, i_p$ and columns $j_1,\ldots, j_p$.
If $A,B,C$ are $(n\times n)$-matrices over any field and $A=BC$, then Cauchy--Binet's formula states that
\[A_{I,J}=\sum_{1\leq k_1 < \ldots < k_p \leq n} B_{I,\{k_1,\ldots,k_p\}} C_{\{k_1,\ldots,k_p\},J}.\]
If $A$ is invertible, we have
\[(A^{-1})_{I,J}=\frac{(-1)^{\sum_{k=1}^p i_k+j_k}}{\det A} \, A_{J^c,I^c},\]
where $I^c=\{i_1'<\ldots<i_{n-p}'\}$ and $J^c=\{j_1'<\ldots<j_{n-p}'\}$ are so that $I \cup I^c = \{1,\ldots,n\}$ and $J\cup J^c = \{1,\ldots,n\}$.
If $I=\{i\}$ and $J=\{j\}$,  we write $A_{i,j}$ instead of $A_{\{i\},\{j\}}$ and $i^c$ instead of $\{i\}^c$.
A special case of the above formula is thus 
\[(A^{-1})_{i,j}= \frac{(-1)^{i+j}}{\det A} A_{j^c,i^c}.\]
The following corollary about products of totally positive (upper/lower triangular) matrices is a direct application of this formula and is classical.

\begin{corol} \label{corol_TotPosProduct}
Let $\K$ be an ordered field.
\begin{enumerate}
\item \label{corol_NonNegTotPosProduct_item1}
The product of two totally positive matrices is totally positive.
The same holds true if we restrict to the set of upper triangular matrices (or lower triangular matrices).
\item \label{corol_NonNegTotPosProduct_item2}
The product of a totally positive lower respectively upper triangular matrix with a totally positive upper respectively lower triangular matrix is totally positive.
\end{enumerate}
\end{corol}

To prove \Cref{propo_5TuplePosInequDR} we need the following lemma, which is a direction application of Cauchy--Binet's formula.

\begin{lem} \label{lem_TotPosMatricesInverses}
Let $\K$ be an ordered field and let $A$ and $B$ be two totally positive upper triangular $(n \times n)$-matrices over $\K$ of determinant one.
For every $k=1,\ldots,n-1$ we have
\[ \frac{(A^{-1})_{k,n}}{(A^{-1})_{k+1,n}} > \frac{((BA)^{-1})_{k,n}}{((BA)^{-1})_{k+1,n}}.\]
\end{lem}
\begin{proof}
By Cauchy--Binet's formula we have for all $j=1,\ldots,n$
\[(A^{-1})_{j,n}=(-1)^{j+n}A_{n^c,j^c}.\]
Since $A$ is totally positive upper triangular,  we have that $A_{n^c,j^c}>0$.
Because of \Cref{corol_TotPosProduct} the same holds true if we replace $A$ by $BA$.
Thus all of the numbers in the above inequality are non-zero.

Fix $k$ between $1$ and $n-1$.
By the same argument we remark that $(A^{-1})_{k+1,n}$ and $((BA)^{-1})_{k+1,n}$ have the same sign.
Thus the inequality amounts to showing that
\[ (A^{-1})_{k,n} ((BA)^{-1})_{k+1,n} - ((BA)^{-1})_{k,n} (A^{-1})_{k+1,n} >0.\]
By Cauchy--Binet's formula we have for all $j=1,\ldots,n$
\begin{align*}
((BA)^{-1})_{j,n}&=(-1)^{j+n}(BA)_{n^c,j^c}\\
&= (-1)^{j+n}\sum_{1\leq k_1<\ldots<k_{n-1}\leq n} B_{n^c,\{k_1,\ldots,k_{n-1}\}} A_{\{k_1,\ldots,k_{n-1}\},j^c}\\
&=(-1)^{j+n}\sum_{1\leq l \leq n} B_{n^c,l^c} A_{l^c,j^c}.
\end{align*}
Putting everything together we obtain
\begin{align}\label{calc_EntriesInvTotPos}
(A^{-1})_{k,n} &((BA)^{-1})_{k+1,n} - ((BA)^{-1})_{k,n} (A^{-1})_{k+1,n}\nonumber\\
&=(-1)^{2k+2n+1} A_{n^c,k^c} \bigg(\sum_{1\leq l \leq n} B_{n^c,l^c} A_{l^c,(k+1)^c}\bigg)\nonumber\\
&\quad\quad - \bigg[ (-1)^{2k+2n+1} A_{n^c,(k+1)^c} \bigg(\sum_{1\leq l \leq n} B_{n^c,l^c} A_{l^c,k^c}\bigg) \bigg]\nonumber\\
&= \sum_{1\leq l \leq n} B_{n^c,l^c} \bigg( A_{n^c,(k+1)^c} A_{l^c,k^c} - A_{n^c,k^c} A_{l^c,(k+1)^c} \bigg).
\end{align}
Since $B$ is totally positive upper triangular,  we know that $ B_{n^c,l^c}>0$ for all $l=1,\ldots,n$.
We now investigate the second factors in the summands of (\ref{calc_EntriesInvTotPos}).
For $l=n$ this factor is equal to zero.
We observe that for all $l=1,\ldots, n-1$ using again Cauchy--Binet's formula that
\begin{align*}
&A_{n^c,(k+1)^c} A_{l^c,k^c} - A_{n^c,k^c} A_{l^c,(k+1)^c}\\ 
&= (-1)^{n+2k+l+1} \left((A^{-1})_{k+1,n} (A^{-1})_{k,l} -(A^{-1})_{k,n}(A^{-1})_{k+1,l} \right)\\
&=  (-1)^{n+l+1}(A^{-1})_{\{k,k+1\},\{l,n\}}\\
&=  (-1)^{n+l+1}(-1)^{n+2k+l+1} A_{\{l,n\}^c,\{k,k+1\}^c}\\
&= A_{\{l,n\}^c,\{k,k+1\}^c}. 
\end{align*}
Whenever $k \leq l \leq n-1$ we have that $A_{\{l,n\}^c,\{k,k+1\}^c}>0$,  since $A$ is totally positive upper triangular (see for example \cite[Section 6.2]{Pinkus_TotallyPositiveMatrices}).
By the same argument,  if $l<k$, then $A_{\{l,n\}^c,\{k,k+1\}^c}=0$.
Thus all summands in (\ref{calc_EntriesInvTotPos}) are either zero or positive.
Since $1\leq k \leq n-1$,  there exists at least one $l$ for which one of the summands is positive.
This proves the lemma.
\end{proof}

\begin{proof}[Proof of \Cref{propo_5TuplePosInequDR}]
Since $(F_1,\ldots,F_5)$ is a positive 5-tuple of flags,  \Cref{lem_PoskTupleNormalForm} implies that there exists a basis $\mathcal{B}=(e_1,\ldots,e_n)$ of $\F^n$ in which $F_1$ is the ascending flag $F^+$ and $F_3$ is the descending flag $F^-$.
Furthermore, there exist $u \in \mathcal{U}^\lltriangle_{>0}(\mathcal{B})$ and $v,w \in \mathcal{U}^\urtriangle_{>0}(\mathcal{B})$ such that
\[F_2 = u F^+,  \quad F_4 = v^{-1}F^-, \quad F_5 = v^{-1} w^{-1} F^-.\]
We thus need to prove that for all $a=1,\ldots, n-1$ we have
\[D_a(F^+, F^-, u F^+, v^{-1} F^-) < D_a(F^+, F^-, u F^+, v^{-1} w^{-1} F^-).\]
Let $x_i, y_i, z_i \in \F$ such that $(u F^+)^{(1)} = \langle \sum_{i=1}^n x_i e_i \rangle$,  $(v^{-1}F^-)^{(1)} = \langle \sum_{i=1}^n y_i e_i \rangle$ and $(v^{-1}w^{-1}F^-)^{(1)} = \langle \sum_{i=1}^n z e_i \rangle$.
By transversality of all of the involved flags, we have $x_i, y_i, z_i \neq 0$ for all $i=1,\ldots,n$.
Note that since $u$ is totally positive lower triangular,  we can choose $x_i >0$ for all $i=1,\ldots,n$.
We compute
\begin{align*}
D_a(F^+, F^-, &u F^+, v^{-1} F^-)= - \frac{e_1 \wedge \ldots \wedge e_a \wedge e_n \wedge \ldots \wedge e_{a+2} \wedge \sum x_i e_i}{e_1 \wedge \ldots \wedge e_a \wedge e_n \wedge \ldots \wedge e_{a+2} \wedge \sum y_i e_i}\\
&\cdot \frac{e_1 \wedge \ldots \wedge e_{a-1} \wedge e_n \wedge \ldots \wedge e_{a+1} \wedge \sum y_i e_i}{e_1 \wedge \ldots \wedge e_{a-1} \wedge e_n \wedge \ldots \wedge e_{a+1} \wedge \sum x_i e_i}=- \frac{x_{a+1} y_a}{x_a y_{a+1}}.
\end{align*}
Similarly we obtain
\[D_a(F^+, F^-, u F^+, v^{-1} w^{-1} F^-)=- \frac{x_{a+1} z_a}{x_a z_{a+1}}.\]
Thus
\begin{align*}
&D_a(F^+, F^-, u F^+, v^{-1} w^{-1} F^-)-D_a(F^+, F^-, u F^+, v^{-1} F^-)\\
&\quad \quad\quad\quad= \frac{x_{a+1}}{x_a} \cdot \frac{y_{a} z_{a+1}-y_{a+1} z_a}{y_{a+1} z_{a+1}},
\end{align*}
which is positive, if and only if $\frac{y_{a} z_{a+1}-y_{a+1} z_a}{y_{a+1} z_{a+1}}$ is positive,  since all $x_i$ are positive.

By definition we can choose $y_a = (v^{-1})_{a,n}$ and $z_a = ((wv)^{-1})_{a,n}$ for all $a=1,\ldots,n$.
Since both $v$ and $wv$ are totally positive upper triangular and of determinant one, see \Cref{corol_TotPosProduct},  Cauchy--Binet's formula implies that $y_a$ and $z_a$ have the same sign,  and hence their product $y_a z_a$ is positive for all $a=1,\ldots,n$.
Applying \Cref{lem_TotPosMatricesInverses} to $A=v$ and $B=w$ shows that $y_{a} z_{a+1}-y_{a+1} z_a$ is positive for all $a=1,\ldots,n-1$.
This finishes the proof.
\end{proof}

\subsection{Positive hyperbolicity}

Let $\F$ be a real closed field.
\begin{dfn}
\label{dfn_PosHypPSL}
An element of $\PSL(n,\F)$ is \emph{positively hyperbolic} if one of its lifts to $\SL(n,\F)$ has distinct and only positive eigenvalues.
\end{dfn}

There is no ambiguity for odd $n$, since $\PSL(n,\F)=\SL(n,\F)$ for every real closed field $\F$.
If $n$ is even, an element in $\PSL(n,\F)$ admits two lifts that differ by sign, and its eigenvalues are only well-defined up to multiplication with $\pm 1$.

\begin{dfn}
\label{dfn:PosHypStableFlag}
For a positively hyperbolic element $M \in \PSL(n,\F)$ with eigenvalues $|\lambda_1| >  \ldots > |\lambda_n| >0$,  and corresponding eigenspaces $V_1, \ldots, V_n$, we define its \emph{stable flag $F^+_M$} and its \emph{unstable flag $F^-_M$} by
\begin{align*}
&F^+_M =\left( V_1,V_1 \oplus V_2,\ldots, V_1 \oplus \ldots \oplus V_{n-1}\right),\\
&F^-_M =\left( V_n, V_n \oplus V_{n-1}, \ldots, V_n \oplus \ldots \oplus V_{2}\right).
\end{align*}
\end{dfn}

Note that if $\F \neq \R$,  then a stable flag is not necessarily attracting in the dynamical sense.

\begin{lem}
\label{lem_PosHypMatricesSemiAlg}
Denote by $\Pos(n,\F)$ the set of positively hyperbolic elements in $\PSL(n,\F)$.
Then $\Pos(n,\F)$ is semi-algebraic.
\end{lem}
\begin{proof}
The set of positively hyperbolic elements in $\PSL(n,\F)$ is described by
\begin{align*}
\Pos(n,\F) &= \{ M \in \PSL(n,\F) \mid \exists\,  M' \in \SL(n,\F), \, \exists\,  v_1,\ldots,v_n \in \F^n, \\
& \quad\quad\quad \exists\, \lambda_1 > \ldots > \lambda_n \in \F_{>0} : \Ad_\F(M') = M, \\
& \quad\quad\quad  \det(v_1 | \ldots |v_n) \neq 0, \,M'v_i = \lambda_i v_i \textrm{ for all } i=1,\ldots,n \},
\end{align*}
which is the projection of a semi-algebraic set,  hence semi-algebraic (\Cref{thm_TarskiSeidenberg}).
\end{proof}

\begin{lem}
\label{lem_StableFlagSemiAlg}
Let $\F$ be a real closed field.
The map
\[ f \from \Pos(n,\F) \to \Flag(\F^n), \quad M \mapsto F_M^+\]
is semi-algebraic.
The same is true if we replace the stable flag by the unstable flag.
\end{lem}
\begin{proof}
We need to prove that $\Graph(f)$ is semi-algebraic (\Cref{dfn_SemiAlgSet}). 
We have
\begin{align*}
\Graph(f)  &= \{ (M,F)\in \Pos(n,\F) \times \Flag(\F^n) \mid F=F_M^+ \}\\
&=\{ (M,F)\in \Pos(n,\F) \times \Flag(\F^n) \mid \exists\,  M' \in \SL(n,\F), \\
& \quad\quad\quad \exists\,  v_1,\ldots,v_n \in \F^n, \exists\, \lambda_1 > \ldots > \lambda_n \in \F_{>0} : \\
& \quad\quad\quad  \Ad_\F(M') = M, \, \det(v_1 | \ldots |v_n) \neq 0, \\
& \quad\quad\quad M'v_i = \lambda_i v_i \textrm{ and }v_1,\ldots,  v_i \in F^{(i)} \textrm{ for all } i=1,\ldots,n \}.
\end{align*}
We can conclude that $\Graph(f)$ is semi-algebraic as the projection of a semi-algebraic set to its first coordinates (\Cref{thm_TarskiSeidenberg}).
Note that the condition $v_1,\ldots,  v_i \in F^{(i)}$ can be expressed as $P_{F^{(i)}}(v_j) = v_j$ for all $j=1,\ldots,i$; see the proof of \Cref{lem_FlagSpacesAlg}.
\end{proof}

%% file: sections/06_poshypFHitpos.tex
\section{Positive hyperbolicity of $\F$-Hitchin and $\F$-positive representations}
\label{section_PosHypRepr}

For this section let $\F$ be a real closed field extension of $\R$.
Recall from \Cref{dfn_PosHypPSL} that an element of $\PSL(n,\F)$ is positively hyperbolic if it has a positively hyperbolic lift to $\SL(n,\F)$, i.e.\ a matrix which has distinct and only positive eigenvalues.

\begin{dfn}
\label{dfn_PosHypRepr}
A representation $\rho \from \pi_1(S) \to \PSL(n,\F)$ is \emph{positively hyperbolic} if $\rho(\gamma)$ is positively hyperbolic for all $e \neq \gamma \in \pi_1(S)$.
\end{dfn}

\subsection{$\F$-Hitchin representations are positively hyperbolic}
\label{section_FHitchinRepr}

Recall from \Cref{lem_FHitchinHomHit} that an $\F$-Hitchin representation is a representation that lies in the $\F$-extension $\Hom_\Hit(\pi_1(S),\PSL(n,\F))$ of the set of real Hitchin representations $\Hom_\Hit(\pi_1(S),\PSL(n,\R))$.

\begin{propo} \label{lem_hitchinReprAreDiagonalizable}
Let $\rho \from \pi_1(S) \to \PSL(n,\F)$ be an $\F$-Hitchin representation.
Then $\rho$ is positively hyperbolic.
\end{propo}
\begin{proof}
Fix a non-trivial $\gamma \in \pi_1(S)$.
Since $\Hom_\Hit(\pi_1(S),\PSL(n,\R))$ and $\Pos(n,\R)$ are real semi-algebraic we can consider the real semi-algebraic set
\[
X_\gamma \coloneqq \{ \rho' \in \Hom_\Hit(\pi_1(S),\PSL(n,\R)) \mid \rho'(\gamma) \in \Pos(n,\R) \}.
\]
Since real Hitchin representations are positively hyperbolic,  refer for example to \cite[Proposition 8 and Lemma 9]{BonahonDreyer_ParametrizingHitchinComponents}, we obtain $\Hom_\Hit(\pi_1(S),\PSL(n,\R))$ $=X_\gamma $.
By the Tarski--Seidenberg principle (\Cref{thm_TarskiSeidenberg}),  we obtain that for any real closed extension $\F$ of $\R$, 
\[
(X_{\gamma})_\F = \Hom_\Hit(\pi_1(S),\PSL(n,\F)).\]
This means that $\F$-Hitchin representations are positively hyperbolic.
\end{proof}

\subsection{$\F$-positive representations are positively hyperbolic}
\label{section_FpositiveRepr}

In the following we prove that $\F$-positive representations (Definition \ref{dfn_Fpositivity}) are positively hyperbolic and weakly dynamics-preserving in the following sense.

\begin{dfn}\label{dfn_DynPres}
An $\F$-positive representation $\rho \from \pi_1(S) \to \PSL(n,\F)$ with $\rho$-equivariant limit map $\xi_\rho$ is \emph{weakly dynamics-preserving} if for all $e \neq \gamma \in \pi_1(S)$ we have $\xi_\rho(\gamma^+) = F_{\rho(\gamma)}^+$,  where $\gamma^+$ denotes the attracting fixed point of $\gamma$ and $ F_{\rho(\gamma)}^+$ the stable flag of $\rho(\gamma)$; compare \Cref{dfn:PosHypStableFlag}.
\end{dfn}

This result will be crucial in proving that $\F$-positive representations are $\F$-Hitchin.
In the case $\F=\R$,  this result is proven implicitly in \cite[Theorem 7.2]{FockGoncharov_ModuliSpacesLocalSystemsHigherTeichmuellerTheory}.
They first extend the limit map continuously to a positive map on the whole $\partial_\infty \tilde{S}$, and then use continuity and dynamical arguments to prove the statement---a strategy that cannot be adapted to general real closed fields.
We provide a proof that uses positivity of the limit map directly.

\begin{propo} \label{propo_positiveReprHaveDistinctEigenvalues}
Let $\F$ be a real closed field and $\rho \from \pi_1(S) \to \PSL(n,\F)$ an $\F$-positive representation.
Then $\rho$ is positively hyperbolic and weakly dynamics-preserving.
In particular,  the limit map of an $\F$-positive representation is unique.
\end{propo}
\begin{proof}
Let $\gamma \in \pi_1(S)$ be non-trivial.
Denote by $\gamma^+$ respectively $\gamma^-$ the attracting respectively repelling fixed point of $\gamma$ in $\Fix(S) \subseteq \partial_\infty\tilde{S}$.
Let $x, y \in \Fix(S) \setminus \{\gamma^+,\gamma^-\}$ be two distinct points such that $x$ lies on the right hand side of the axis of $\gamma$ and $y$ lies on the left hand side of the axis of $\gamma$,  where the axis is oriented towards $\gamma^+$.
Then $(\gamma^+,x,\gamma^-,y)$ and $(\gamma^+,x,\gamma^-,y,\gamma y)$ are cyclically ordered in clockwise direction.
\begin{figure}[H]
\centering
\includegraphics[width=0.4\textwidth]{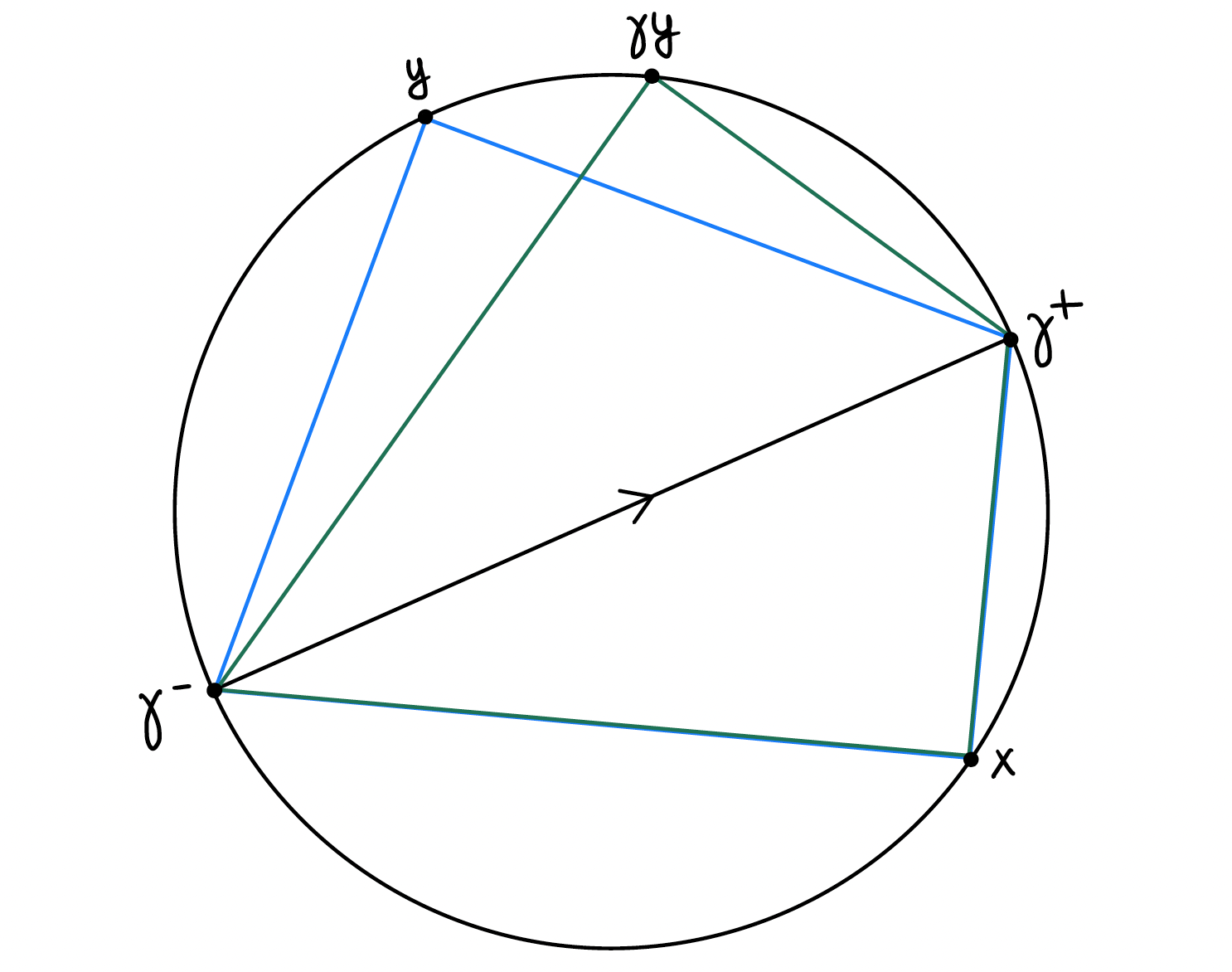}
\captionof{figure}{The cyclically ordered tuple $(\gamma^+,x,\gamma^-,y,\gamma y)$.} \label{fig_5TupleAxis}
\end{figure}
\noindent Denote by $\xi_\rho$ the $\rho$-equivariant positive map $\Fix(S) \to \Flag(\F^n)$ associated to the $\F$-positive representation $\rho$.
By positivity of $\xi_\rho$ and \Cref{rem_ImageOfkCyclicallyOrderedPointsIsPositivekTuple} the 5-tuple of flags 
\[(F_1,\ldots,F_5)\coloneqq (\xi_\rho(\gamma^+),\xi_\rho(x),\xi_\rho(\gamma^-),\xi_\rho(y),\xi_\rho(\gamma y))\]
is positive.
By \Cref{propo_5TuplePosInequDR} and $\rho$-equivariance of $\xi_\rho$ we have
\begin{equation} \label{calc_InequalityDR}
D_a(F_1, F_3, F_2, F_4)< D_a(F_1, F_3, F_2, \rho(\gamma)F_4).
\end{equation}
As in the proof of \Cref{propo_5TuplePosInequDR},  let $\mathcal{B}=(e_1,\ldots,e_n)$ be a basis of $\F^n$ in which $F_1$ is the ascending flag and $F_3$ is the descending flag.
Since both $F_1$ and $F_3$ are $\rho(\gamma)$-invariant,  in this basis there exists a lift $\widetilde{\rho(\gamma)} \in \SL(n,\F)$ of $\rho(\gamma)$ with $\widetilde{\rho(\gamma)} = \textrm{diag}(\lambda_1,\ldots,\lambda_n)$ for $\lambda_1,\ldots,\lambda_n \in \F\setminus \{0\}$.
Let $x_i, y_i  \in \F\setminus\{0\}$ such that $F_2^{(1)}= \langle \sum_{i=1}^n x_i e_i \rangle$ and $F_4^{(1)}= \langle \sum_{i=1}^n y_i e_i \rangle$.
Then $\rho(\gamma)F_4^{(1)} = \langle \sum_{i=1}^n \lambda_i y_i e_i \rangle$.
With these notations we compute as in the proof of \Cref{propo_5TuplePosInequDR}
\[D_a(F_1,F_3,F_2,F_4)= -\frac{x_{a+1} y_a}{x_a y_{a+1}}\eqqcolon d >0,\]
and 
\[D_a(F_1,F_3,F_2,\rho(\gamma)F_4)= -\frac{\lambda_a}{\lambda_{a+1}} \frac{x_{a+1} y_a}{x_a y_{a+1}} = \frac{\lambda_a}{\lambda_{a+1}} d.\]
By inequality (\ref{calc_InequalityDR}) we obtain that
\[d \bigg(\frac{\lambda_a}{\lambda_{a+1}}-1\bigg)>0.\]
Since $d$ is positive,  it follows that $\lambda_a/\lambda_{a+1}>1$, which is what we had to prove.
\end{proof}

%% file: sections/07_BDcoordinates.tex
\section{A variant of the Bonahon--Dreyer coordinates}
\label{section_VariantOfBDCoordinates}
In this section we describe an explicit semi-algebraic model for the Hitchin component in the sense of \Cref{dfn_SemiAlgModel}.
It is based on the seminal  \linebreak work of Fock--Goncharov \cite{FockGoncharov_ModuliSpacesLocalSystemsHigherTeichmuellerTheory}.
Originally defined for non-closed surfaces,  Bonahon--Dreyer \cite{BonahonDreyer_ParametrizingHitchinComponents} adapt the coordinates to closed surfaces.
In the case $n=2$ they agree with the shear coordinates of Teichm\"uller space, see \cite[Section 9]{Thurston_MinimalStretchMaps} or \cite[Theorem A]{Bonahon_ShearingHyperbolicSurfaces}.
We first focus on the case $\F=\R$, and later extend to other real closed fields in \Cref{subsection_BDExt}.
The coordinates build on the existence of positive limit maps into the flag variety associated to Hitchin representations.

\begin{theor}[{\cite[Theorem 1.15]{FockGoncharov_ModuliSpacesLocalSystemsHigherTeichmuellerTheory}}, {\cite[Theorem 1.4]{Labourie_AnosovFlowsSurfaceGroupsAndCurvesInProjectiveSpace}}]
\label{thm_FlagCurveHitchinRepr}
Let \linebreak$\rho \from \pi_1(S) \to \PSL(n,\R)$ be a Hitchin representation.
Then there exists a $\rho$-equivariant continuous map
\[ \xi_\rho \from \partial_\infty\tilde{S} \to \Flag(\R^n),\]
such that
\begin{enumerate}
\item \label{thm_FlagCurveHitchinRepr_stableFlag}
if $x \in \partial_\infty\tilde{S}$ is the attracting fixed point of $\gamma \in \pi_1(S)$ then $\xi_\rho(x)$ is the stable flag of $\rho(\gamma)$,
\item  \label{thm_FlagCurveHitchinRepr_transverse}
for $x, y \in  \partial_\infty\tilde{S}$ distinct, the flag tuple $\left(\xi_\rho(x), \xi_\rho(y)\right)$ is transverse,
\item  \label{thm_FlagCurveHitchinRepr_triplespos}
for $x, y, z \in  \partial_\infty\tilde{S}$ distinct, the flag triple $\left(\xi_\rho(x), \xi_\rho(y), \xi_\rho(z)\right)$ is positive,
\item  \label{thm_FlagCurveHitchinRepr_quadruplepos}
for $x, z,y, z' \in  \partial_\infty\tilde{S}$ distinct occurring in this clockwise order around the circle $\partial_{\infty} \tilde{S}$, the flag quadruple $\left(\xi_\rho(x), \xi_\rho(z), \xi_\rho(y), \xi_\rho(z')\right)$ is positive.
\end{enumerate}
\end{theor}

In the following, we give a slight variant of Bonahon--Dreyer's parametrization of the real Hitchin component presented in \cite{BonahonDreyer_ParametrizingHitchinComponents}.
We emphasize that this variant differs only from the original Bonahon--Dreyer coordinates by not taking logarithms and requiring an additional positivity condition. 
We recall the definition of the coordinates here. 
For a more detailed description and more information we recommend \cite{BonahonDreyer_ParametrizingHitchinComponents}.

\subsection{Coordinates}
\label{subsection_BDCoord}
In order to define the coordinates we need to fix some topological data on $S$.
Choose an auxiliary hyperbolic metric on $S$ and a maximal geodesic lamination $\lambda$ with finitely many leaves, and denote by $\tilde{\lambda} \subset \tilde{S}$ its lift to $\tilde{S}$.
A \emph{geodesic lamination} $\lambda$ on $S$ is a lamination of $S$ whose leaves are geodesics.
We say $\lambda$ is \emph{finite} if $\lambda$ has finitely many leaves.
It is called \emph{maximal} if it is maximal with respect to inclusion, or equivalently if its complement $S\setminus \lambda$ is a union of ideal triangles.

\begin{figure}[H]
\centering
\includegraphics[width=0.54\textwidth]{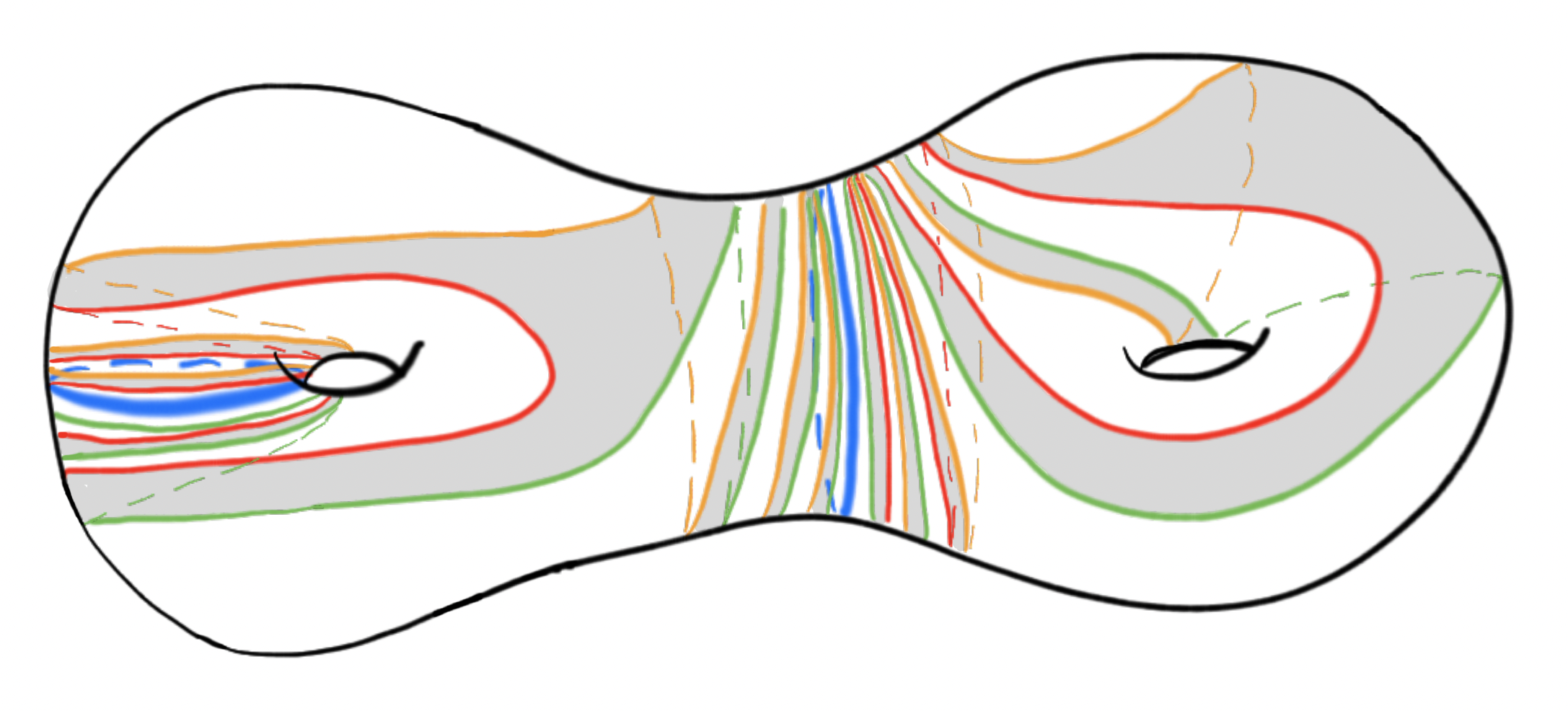}
\includegraphics[width=0.34\textwidth]{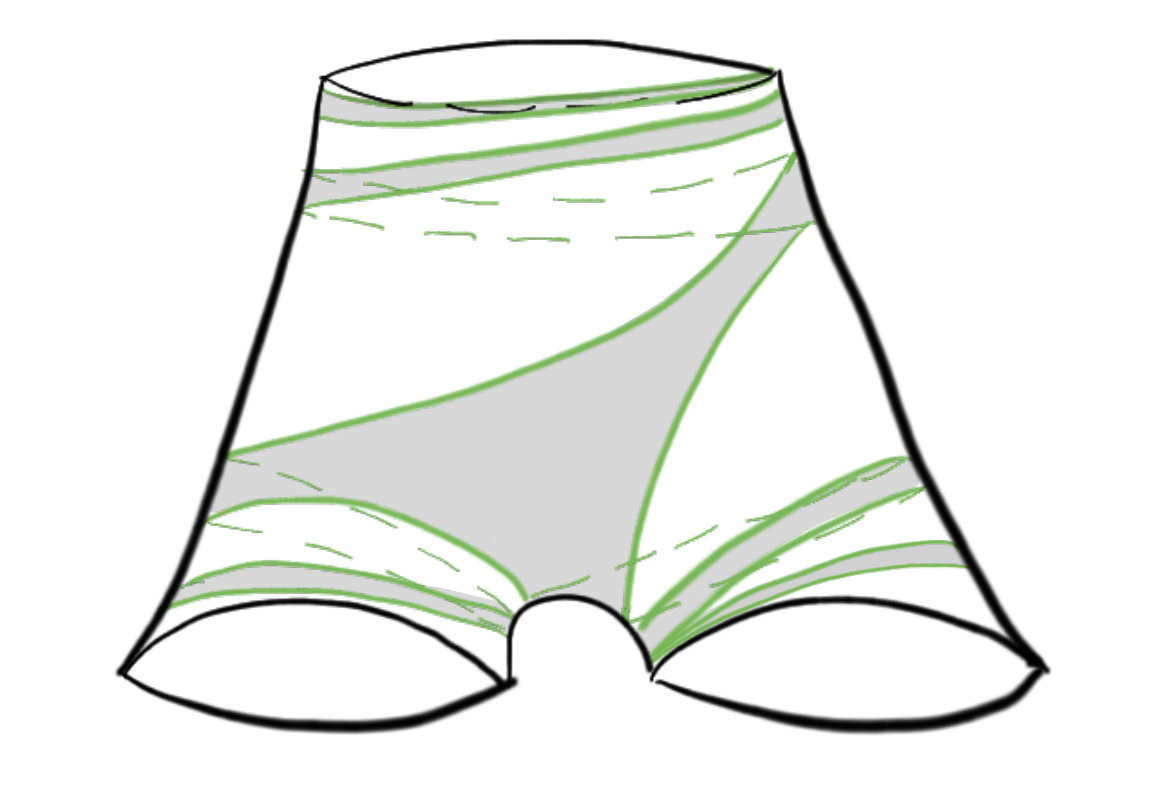}
\captionof{figure}{A maximal geodesic lamination of $S$ (left) and of a pair of pants (right).  The white and shaded regions are ideal triangles.} \label{fig_MaximalGeodLam}
\end{figure}
Although geodesic laminations can be defined in a metric-independent way and, in particular, are purely topological objects, see e.g.\ \cite[Proposition 8.9.4]{Thurston_TheGeometryAndTopologyOf3Mnfds} or \cite[Lemma 18]{Bonahon_GeodesicLaminationsOnSurfaces}, it will be convenient here to fix a hyperbolic metric on $S$.
Furthermore, we (arbitrarily) choose an orientation on each leaf of $\lambda$.
Note that the leaves of $\lambda$ come in two flavours: there are closed leaves and infinite leaves.
Finally, for every closed geodesic leaf $\gamma$ in $\lambda$ we choose an arc $k$ that is transverse to $\lambda$, cuts $\gamma$ in exactly one point and has endpoints in $S \setminus \lambda$.
We write $\partial_\infty\tilde{\lambda} \subseteq \partial_\infty\tilde{S}$ for the endpoints of the leaves of $\tilde{\lambda}$.
Since $\lambda$ is a closed subset of $S$, we observe that $\partial_\infty\tilde{\lambda} \subseteq \Fix(S)$.

We are now ready to describe the Bonahon--Dreyer coordinates that parametrize the Hitchin component.
Let $\rho \from \pi_1(S) \to \PSL(n,\R)$ be a Hitchin representation and $\xi_\rho \from \partial_{\infty} \tilde{S} \to \Flag(\R^n)$ the limit map of $\rho$ by \Cref{thm_FlagCurveHitchinRepr}.
We associate to $\rho$ the following invariants.

\begin{enumerate}
\item \label{item_triangleInvariants}
\emph{Triangle invariants}:
For every ideal triangle $t \subset S\setminus \lambda$ and every vertex $v$ of $t$, we associate for $a, b, c \in \N_{\geq 1}$ with $a+b+c = n$ the triangle invariant
\[ T_{abc}^\rho(t,v) \coloneqq  T_{abc}(\xi_\rho(\tilde{v}), \xi_\rho(\tilde{v}'), \xi_\rho(\tilde{v}'')), \]
where $\tilde{v}, \tilde{v}', \tilde{v}''$ are the vertices of a lift $\widetilde{t}$ of $t$ to $\tilde{S}$, and $\tilde{v}$ is the vertex corresponding to the vertex $v$ of $t$.
Note that the vertices $\tilde{v}, \tilde{v}', \tilde{v}'' \in \partial_\infty \tilde{S}$ are labelled in clockwise order around $\widetilde{t}$.

\item \label{item_shearInvariantsInfiniteLeaves}
\emph{Shear invariants for infinite leaves}:
For every infinite leaf $h \subset \lambda$ we associate for $a = 1, \ldots, n-1$ the shear invariant
\[ D^\rho_a(h) \coloneqq  D_a(\xi_\rho(h^+), \xi_\rho(h^-), \xi_\rho(z), \xi_\rho(z')), \]
where $\tilde{h}$ is a lift of $h$ to $\tilde{S}$, $h^\pm \in \partial_\infty \tilde{S}$ its positive respectively negative endpoint, and $z, z' \in \partial_\infty \tilde{S}$ the third vertices of $\widetilde{t}$ and $\widetilde{t}'$, respectively, where $\widetilde{t}$ and $\widetilde{t}'$ are the ideal triangles that lie on the left respectively right side of $\tilde{h}$ (where the orientation of $\tilde{h}$ comes from the orientation of $h$ which was part of the topological data), see \Cref{fig_ShearInvariantsInfiniteLeaves}.

\begin{minipage}{\linewidth}
\centering
\includegraphics[width=0.38\textwidth]{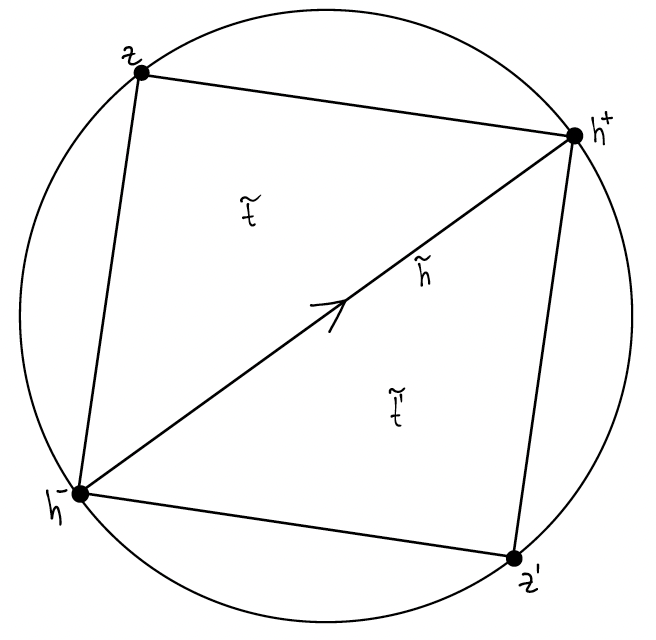}
\captionof{figure}{The construction of the shear invariants for infinite leaves.} \label{fig_ShearInvariantsInfiniteLeaves}
\end{minipage}

\item \label{item_shearInvariantsClosedLeaves}
\emph{Shear invariants for closed leaves}:
For every closed leaf $\gamma \subset \lambda$ we associate for $a = 1, \ldots, n-1$ the shear invariant
\[ D^\rho_a(\gamma) \coloneqq D_a(\xi_\rho(\gamma^+), \xi_\rho(\gamma^-), \xi_\rho(z), \xi_\rho(z')), \]
where $\tilde{\gamma}$ is a lift of $\gamma$ to $\tilde{S}$, $\gamma^\pm \in \partial_\infty \tilde{S}$ its positive respectively negative endpoint, and $z, z' \in \partial_\infty \tilde{S}$ are vertices of the triangles $\widetilde{t}$ and $\widetilde{t}'$ defined by the arc $k$ intersecting $\gamma$ in the following way:
Lift $k$ to an arc $\tilde{k}$ that intersects $\tilde{\gamma}$ in exactly one point.
The two endpoints of $\tilde{k}$ lie in two ideal triangles $\widetilde{t}$ and $\widetilde{t}'$, that each share one vertex with the endpoints of $\tilde{\gamma}$ (this is by definition of the arc $k$), and such that $\widetilde{t}$ lies to the left of $\tilde{\gamma}$ and $\widetilde{t}'$ to the right of $\tilde{\gamma}$.
Now $z, z'$ are the endpoints of $\widetilde{t}$ respectively $\widetilde{t}'$ that are that are not adjacent to a component of $\tilde{S} \setminus (\widetilde{t} \cup \widetilde{t}')$ that contains $\tilde{\gamma}$; see \Cref{fig_ShearInvariantsClosedLeaves} for one of the four possible configurations.

\begin{minipage}{\linewidth}
\centering
\includegraphics[width=0.4\textwidth]{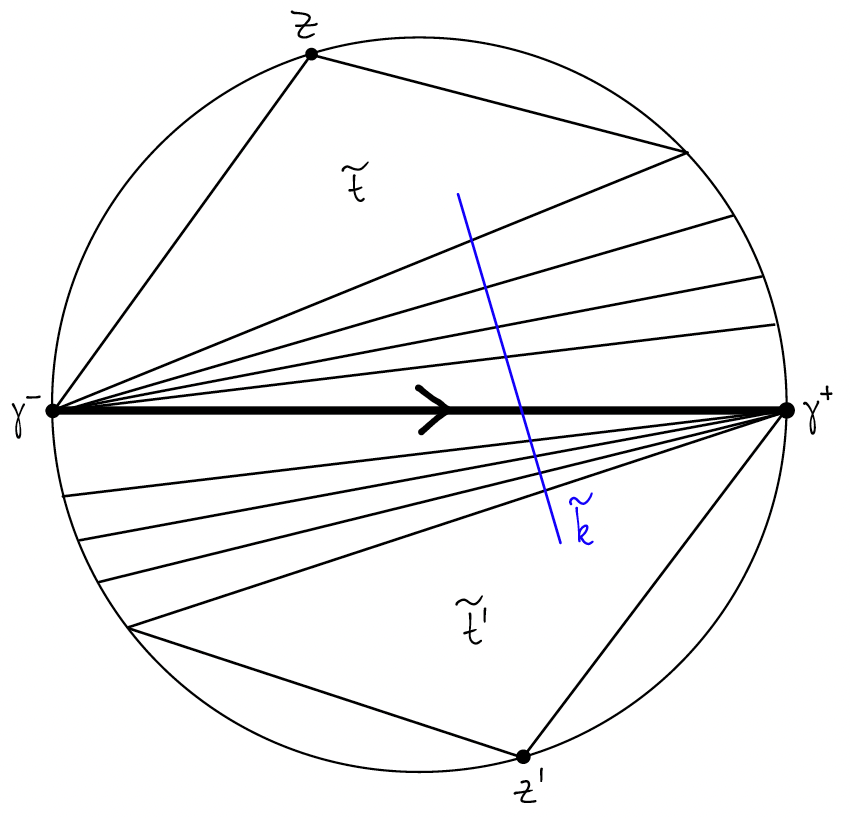}
\captionof{figure}{The construction of the shear invariants for closed leaves.}\label{fig_ShearInvariantsClosedLeaves}
\end{minipage}
\end{enumerate}

\begin{rem} \label{rem_invariantswelldefined}
If $\rho'$ and $\rho$ are conjugate by an element $g \in \PGL(n,\R)$ then $\xi_\rho = g \xi_{\rho'}$.
Thus conjugate representations the same invariants, since the triple and double ratios are $\PGL(n,\R)$-invariant.
\end{rem}

\begin{rem} \label{rem_flagdecoration}
Note that the definition of the above invariants only uses the \emph{flag decoration}, i.e.\ the restriction of $\xi_\rho$ to $\partial_\infty\tilde{\lambda}\subseteq \Fix(S)$.
\end{rem}

\subsection{Relations}
\label{subsection_BDRelations}
The Bonahon--Dreyer coordinates are not independent but satisfy the following relations.
\noindent\begin{enumerate}[label=(\roman*)]
\item\label{item_PositivityCondition}
\emph{Positivity condition}:
For every ideal triangle $t \subset S\setminus \lambda$ and every vertex $v$ of $t$, the triangle invariants $T_{abc}^\rho(t,v)$ are positive for all $a, b, c \in \N_{\geq 1}$ with $a+b+c = n$.
Similarly for every (infinite or closed) leaf $l \in \lambda$ the shear invariants $D_a^\rho(l)$ are positive for all $a=1, \ldots, n-1$.

\item \label{item_RotationCondition}
\emph{Rotation condition}: 
For every ideal triangle $t \subset S\setminus \lambda$ with vertices $v$ and $v'$ such that $v'$ immediately follows $v$ when going in clockwise direction around the boundary of $t$, and for every $a, b, c \in \N_{\geq 1}$ with $a+b+c = n$, we have
\[T_{abc}^\rho(t,v)=T_{bca}^\rho(t,v').\]

\item \label{item_ClosedLeafEquality}
\emph{Closed leaf equality}:
For every closed leaf $\gamma \subset \lambda$ and every $a = 1, \ldots, n-1$, we have
\[ L_a^{\textnormal{right}}(\gamma)=L_a^{\textnormal{left}}(\gamma), \]
where $L_a^{\textnormal{right}}(\gamma)$ and $L_a^{\textnormal{left}}(\gamma)$ will be defined in the following paragraph.

\item \label{item_ClosedLeafInquality}
\emph{Closed leaf inequality}:
For every closed leaf $\gamma \subset \lambda$ and every $a = 1, \ldots, n-1$, we have
\[ L_a^{\textnormal{right}}(\gamma) > 1 \]
\end{enumerate}
Following \cite{BonahonDreyer_ParametrizingHitchinComponents}, for $\gamma$ a closed leaf of the finite maximal lamination $\lambda$ choose a side of $\gamma$ for the chosen orientation.
Denote by $h_l$ and $t_l$ for $l=1, \ldots, k$ the infinite leaves and the ideal triangles that spiral on the chosen side of $\gamma$.
An infinite leaf will appear twice in this list if its two ends spiral on the selected side of $\gamma$.
The spiralling of the triangle $t_l$ on this side of $\gamma$ occurs in the direction of a vertex which we call $v_l$.
Define
\[ \overline{D}_a(h_l) \coloneqq
\begin{cases}  
D_a^\rho(h_l),\quad \; \textnormal{ if } h_l \textnormal{ is oriented toward } \gamma, \\
D_{n-a}^\rho(h_l), \; \textnormal{ if } h_l \textnormal{ is oriented away from } \gamma.
\end{cases}
\]
We define 
\[ L_a^{\textnormal{right}}(\gamma) \coloneqq  \prod_{l=1}^{k} \overline{D}_a(h_l) \prod_{l=1}^k \prod_{b+c=n-a} T_{abc}^\rho(t_{l}, v_l) ,\]
if the spiralling of the triangles occurs in the direction of the orientation of $\gamma$, and otherwise
\[ L_a^{\textnormal{right}}(\gamma) \coloneqq  \left( \prod_{l=1}^{k} \overline{D}_{n-a}(h_l) \prod_{l=1}^k \prod_{b+c=a} T_{(n-a)bc}^\rho(t_{l}, v_l) \right)^{-1}. \]
Similarly, we define
\[ L_a^{\textnormal{left}}(\gamma) \coloneqq \left( \prod_{l=1}^{k} \overline{D}_a(h_l) \prod_{l=1}^k \prod_{b+c=n-a} T_{abc}^\rho(t_{l}, v_l) \right)^{-1}, \]
if the spiralling of the triangles occurs in the direction of the orientation of $\gamma$, and otherwise
\[ L_a^{\textnormal{left}}(\gamma) \coloneqq \prod_{l=1}^{k} \overline{D}_{n-a}(h_l) \prod_{l=1}^k \prod_{b+c=a} T_{(n-a)bc}^\rho(t_{l}, v_l). \]
Note that the triangles and infinite leaves that spiral towards a side of $\gamma$ differ depending on which side of $\gamma$ we choose.

\subsection{Parametrization}
\label{subsection_BDParam}
Suppose the geodesic lamination $\lambda$ has $p$ \linebreak closed leaves and $q$ infinite leaves, and its complement $S \setminus \lambda$ consists of $r$ ideal triangles.
There are $(n-1)(n-2)/2$ triples of integers $a, b, c \geq 1$ with $a + b + c = n$. 
For
\[N = 3r \tfrac{(n-1)(n-2)}{2} + (p + q)(n - 1)\]
the functions $\{T_{abc}\}$ and $\{D_a\}$ define a map from the Hitchin component to $\R^N$ satisfying
the following properties.
\begin{enumerate}[label=(\Alph*)]
\item \label{item_CondPosTriangleCoord}
For all $a, b, c \in \N_{\geq 1}$ with $a+b+c=n$, the function $T_{abc}$ associates a positive real number $T_{abc}(t,v)$ to every triangle $t \subset S\setminus \lambda$ and to every vertex $v$ of $t$;
\item \label{item_CondPosShearCoord}
For all $a=1, \ldots, n-1$, the function $D_a$ associates a positive real number $D_a(l)$ to each leaf $l \subset \lambda$;
\item \label{item_CondRotation}
For every triangle $t \subset S\setminus \lambda$ and all indices $a, b, c \in \N_{\geq 1}$ with $a+b+c=n$, the functions $T_{abc}$ satisfy the rotation condition as in \Cref{subsection_BDRelations} \ref{item_RotationCondition};
\item  \label{item_CondClosedLeafEquInequ}
For every closed leaf $\gamma \subset \lambda$ and every index $a=1,\ldots, n-1$, the functions $T_{abc}$ and $D_a$ satisfy the closed leaf equality and the closed leaf inequality as in \Cref{subsection_BDRelations} \ref{item_ClosedLeafEquality} respectively \ref{item_ClosedLeafInquality}.
\end{enumerate}
Consider the semi-algebraic subset $\mathcal{X} \subseteq \R^N$ defined by the properties \ref{item_CondPosTriangleCoord}-\ref{item_CondClosedLeafEquInequ}.
More precisely,  let
\[ \left( x_{abc,t,v},y_{m,l} \right) \in \R^N\]
where $a+b+c=n$,  $m=1,\ldots,n-1$, $v$ is a vertex of a component $t \subset S \setminus \lambda$ and $l$ is a leaf of $\lambda$.
Then $\mathcal{X}$ is the set of all $\left( x_{abc,t,v},y_{m,l} \right) \in \R^N$ satisfying:
\begin{itemize}
\item[--] 
$x_{abc,t,v},\, y_{m,l}>0$ for all $a,b,c,t,v,m,l$ (Positivity condition \ref{item_PositivityCondition});
\item[--] 
$x_{abc,t,v}=x_{bca,t,v'}$ for all $a,b,c,t$, where $v'$ is the vertex of $t$ that follows $v$ in clockwise direction (Rotation condition \ref{item_RotationCondition});
\item[--]
For every closed leaf of $\lambda$ and every index $1,\ldots,n-1$ the coordinates $x_{abc,t,v}$,  $y_{m,l}$ satisfy the closed leaf equality and inequality as in \Cref{subsection_BDRelations} \ref{item_ClosedLeafEquality} respectively \ref{item_ClosedLeafInquality}.
\end{itemize}
We have seen in \Cref{subsection_BDCoord} and \Cref{subsection_BDRelations} that the map that assigns to a Hitchin representation its triangle and shear invariants has image in $\mathcal{X}$.
In fact, the image of the map is exactly $\mathcal{X}$.

\begin{theor}[{\cite[Theorem 17]{BonahonDreyer_ParametrizingHitchinComponents}}]
\label{thm_BDCoord}
The semi-algebraic set $\mathcal{X} \subseteq \R^N$ is homeomorphic to $\Hit(S,n)$.
\end{theor}

From a point in $\mathcal{X}$ the authors construct a flag decoration $\partial_\infty \tilde{\lambda} \to \Flag(\R^n)$ and show that there is a unique representation $\pi_1(S) \to \PGL(n,\R)$ for which the flag decoration is equivariant, compare \Cref{propo_FlagDecFromCoord}.
We will use the same argument for $\F$-positive representations in the proof of \Cref{thm_BDFParam} in \Cref{section_ProofOfMainTheorem}.

\subsection{Bonahon--Dreyer coordinates as a semi-algebraic model}
This parametrization defines a semi-algebraic model for the Hitchin component in the sense of \Cref{dfn_SemiAlgModel}.

\begin{propo}
\label{lem_BDSemiAlgModel}
The map 
\[\pr^{\BD} \from \Hom_\Hit(\pi_1(S),\PSL(n,\R))\to \mathcal{X}\]
defined by Bonahon--Dreyer in \Cref{subsection_BDCoord} (\ref{item_triangleInvariants})-(\ref{item_shearInvariantsClosedLeaves}) is a semi-algebraic model for $\Hit(S,n)$.
\end{propo}
\begin{proof}
The map $\pr^{\BD}$ is continuous and its image $\mathcal{X}$ is by \Cref{thm_BDCoord} homeomorphic to $\Hit(S,n)$, thus the fibres over $\mathcal{X}$ are exactly the $\PSL(n,\R)$-orbits.
We need to prove that $\pr^{\BD}$ is a semi-algebraic map.
Let $y=(y_1,\ldots,y_k)$ be the finite set of points in $\Fix(S)$ that are endpoints of the lifts of the leaves of the lamination $\lambda$ needed to define the Bonahon--Dreyer coordinates.
The points $y_1,\ldots,y_k$ define a polygon $P$.
Given a $k$-tuple of transverse flags we can assign it to the polygon $P$.
In the same way as in the definition of the Bonahon--Dreyer coordinates we then obtain a map 
\[\phi_{\lambda} \from \Flag(\R^n)^{(k)} \to \R^N,\]
for $N$ as in \Cref{subsection_BDParam} by assigning the triangle and shear invariants as in \Cref{subsection_BDCoord} (\ref{item_triangleInvariants})-(\ref{item_shearInvariantsClosedLeaves}) according to $\lambda$.
The map $\pr^\BD$ is then the composition of the maps $\phi_\lambda \circ \xi_y$, where the latter is defined as
\begin{align*}
\xi_y \from \Hom_\Hit(\pi_1(S),\PSL(n,\R)) &\to \Flag(\R^n)^{(k)} ,  \\
\rho &\mapsto (\xi_\rho(y_1),\ldots,\xi_\rho(y_k)).
\end{align*}
We show that $\phi_\lambda$ and $\xi_y$ are semi-algebraic.
Refer to \Cref{lem_FlagSpacesAlg} for the real semi-algebraic structure on $\Flag(\R^n)^{(k)}$.
The map $\phi_\lambda$ is given by the triple and double ratios, and is hence semi-algebraic by \Cref{propo_TripleDoubleRatioSemiAlg}.

Let us discuss $\xi_y$.
Let $\gamma_1, \ldots, \gamma_k \in \pi_1(S)$ with $\gamma_i^+=y_i$.
Recall from \Cref{thm_FlagCurveHitchinRepr} (\ref{thm_FlagCurveHitchinRepr_stableFlag}) that $\xi_\rho(y_i) = F_{\rho(\gamma_i)}^+$ for all $i=1,\ldots,k$.
By \Cref{lem_StableFlagSemiAlg} it follows that $\xi_y$ is semi-algebraic.
Putting the two maps together we obtain that $\pr^{\BD}$ is semi-algebraic by \Cref{propo_CompSemiAlgMap}.
\end{proof}

\subsection{Extension to real closed fields}
\label{subsection_BDExt}

Let $\F$ be a real closed field extension of $\R$.
We would like to generalize the coordinates to the $\F$-Hitchin component.

\begin{corol}
\label{corol_CoordBDF}
The $\F$-extension of the Bonahon--Dreyer coordinates induces a bijection between the semi-algebraic set $\mathcal{X}_\F \subseteq \F^N$ and $\Hit(S,n)_\F$.
\end{corol}
\begin{proof}
We consider the $\F$-extension of the semi-algebraic model for $\Hit(S,n)$,  i.e.\
\[\pr^{\BD}_\F \from \Hom_\Hit(\pi_1(S),\PSL(n,\F)) \to \mathcal{X}_\F.\]
By \Cref{lem_ExtSemiAlgModel},  $\pr^\BD_\F$ induces a bijection between $\Hit(S,n)_\F$ and $\mathcal{X}_\F$.
\end{proof}

%% file: sections/08_proofthm.tex
\section{%
\texorpdfstring{%
Proof of \Cref{thm_FHitchinEquivFPositive}}%
{Proof of Theorem 1.1}}
\label{section_ProofOfMainTheorem}

We recall \Cref{thm_FHitchinEquivFPositive}.

\thmA*

\begin{proof}[($\implies$)]
We first prove the ``only if" direction, namely we have to show that an $\F$-Hitchin representation is $\F$-positive in the sense of \Cref{dfn_Fpositivity}.
Let thus $\rho \from \pi_1(S) \to \PSL(n,\F)$ be an $\F$-Hitchin representation.
The main idea is to use the Tarski--Seidenberg transfer principle to deduce results for $\F$-Hitchin representations from properties of $\R$-Hitchin representations.
For all $x \in \Fix(S)$, choose a non-trivial $\gamma \in \pi_1(S)$ with $x=\gamma^+$, where we write $\gamma^{+}$, $\gamma^-$ for the attracting respectively repelling fixed point of $\gamma$ acting on $\Fix(S) \subseteq \partial_\infty\tilde{S}$.
By \Cref{lem_hitchinReprAreDiagonalizable} we know that $\rho(\gamma)$ is positively hyperbolic for all $e \neq \gamma \in \pi_1(S)$.
Thus we can define $\xi_\rho$ by
\[ \xi_\rho \from  \Fix(S) \to \Flag(\F^n), \quad x \mapsto F_{\rho(\gamma)}^+, \]
where $F_{\rho(\gamma)}^+$ denotes the stable flag of $\rho(\gamma)$ (see \Cref{dfn:PosHypStableFlag}).
We need to check that this map is well-defined since we chose $\gamma$ with $x=\gamma^+$.
Indeed, for any non-trivial $\gamma_1, \gamma_2 \in \pi_1(S)$ with $x=\gamma_1^+ = \gamma_2^+$, we have $\gamma_1^- = \gamma_2^-$, and thus $\gamma_1^{k_1} = \gamma_2^{k_2}$ for some $k_1, k_2 \in \Z$.
But then, since we know that both $\rho(\gamma_1)$ and $\rho(\gamma_2)$ are diagonalizable and they agree up to some power (since $\gamma_1$ and $\gamma_2$ do), their lifts to $\SL(n,\F)$ must have the same eigenspaces and hence $F_{\rho(\gamma_1)}^+=F_{\rho(\gamma_2)}^+$, which shows that the map $\xi_\rho$ is well-defined.

The map $\xi_\rho$ is $\rho$-equivariant: Let $\gamma \in \pi_1(S)$ be a non-trivial element and $\eta^+ \in \Fix(S)$ for some non-trivial $\eta \in \pi_1(S)$. 
Then
\[ \xi_\rho(\gamma \cdot \eta^+) = \xi_\rho ( (\gamma \eta \gamma^{-1})^+ ) = \rho(\gamma) \xi_\rho ( \eta^+), \]
by observing that $\gamma \cdot \eta^+$ is the attracting fixed point of $\gamma \eta \gamma^{-1}$, and that the stable flag of $\rho(\gamma \eta \gamma^{-1})= \rho(\gamma) \rho(\eta) \rho(\gamma)^{-1}$ is obtained by applying $\rho(\gamma)$ to the stable flag of $\rho(\eta)$.

To verify that $\xi_\rho$ is a positive map (\Cref{dfn_Fpositivity}), we use the Tarski--Seidenberg transfer principle, in the same way as we did in \Cref{lem_hitchinReprAreDiagonalizable}.
Let $x=\gamma^+$, $y=\eta^+$ be distinct points in $\Fix(S)$ for some non-trivial $\gamma \neq \eta \in \pi_1(S)$.
We need to show that the flag tuple $(\xi_\rho(\gamma^+), \xi_\rho(\eta^+))$ is transverse.
The set
\[
X_{\gamma,\eta} \coloneqq \big\{ \rho' \in \Hom_\Hit(\pi_1(S),\PSL(n,\R)) \bigm| \big(F^+_{\rho'(\gamma)},F^+_{\rho'(\eta)}\big) \in \Flag(\R^n)^{(2)}\big\}
\]
is semi-algebraic and agrees with $\Hom_\Hit(\pi_1(S),\PSL(n,\R))$ by \Cref{lem_FlagSpacesAlg},  \Cref{lem_StableFlagSemiAlg} and \Cref{thm_FlagCurveHitchinRepr} (\ref{thm_FlagCurveHitchinRepr_stableFlag}) and (\ref{thm_FlagCurveHitchinRepr_transverse}).
Thus the same is true for their $\F$-extensions.
Since $\xi_\rho(\gamma^+)$ is defined as the stable flag of $\rho(\gamma)$ for all $\gamma^+ \in \Fix(S)$,  we conclude that $(\xi_\rho(\gamma^+), \xi_\rho(\eta^+))$ is transverse.

The argument for positivity is the same.
Let $x_1=\gamma_1^+$, $x_2=\gamma_2^+$, $x_3=\gamma_3^+$  be distinct points in $\Fix(S)$ for some non-trivial $\gamma_1, \gamma_2, \gamma_3  \in \pi_1(S)$, positively oriented around the circle.
We need to show that the flag triple $(\xi_\rho(\gamma_1^+), \xi_\rho(\gamma_2^+), \xi_\rho(\gamma_3^+))$ is positive, i.e.\ it is transverse and all triple ratios are positive.
Since assigning triple ratios is a semi-algebraic map (\Cref{propo_TripleDoubleRatioSemiAlg}), the set
\begin{align*}
X_{\gamma_1,\gamma_2,\gamma_3} &\coloneqq \big\{ \rho' \in \Hom_\Hit(\pi_1(S),\PSL(n,\R)) \bigm| \big(F^+_{\rho'(\gamma_1)},F^+_{\rho'(\gamma_2)},F^+_{\rho'(\gamma_3)}\big) \\
& \quad\quad\quad\quad \in \Flag(\R^n)^{(3)}, \, T_{abc}\big(F^+_{\rho'(\gamma_1)},F^+_{\rho'(\gamma_2)},F^+_{\rho'(\gamma_3)}\big)>0\\
& \quad\quad\quad\quad \textnormal{ for all } a,b,c = 1,\ldots, n-2 \textnormal{ with } a+b+c=n \big\}
\end{align*}
is semi-algebraic.
Since all Hitchin representations satisfy this property (\Cref{thm_FlagCurveHitchinRepr} (\ref{thm_FlagCurveHitchinRepr_triplespos})) the set $X_{\gamma_1,\gamma_2,\gamma_3}$ agrees with $\Hom_\Hit(\pi_1(S),\PSL(n,\R))$,  and thus their $\F$-extensions agree (\Cref{thm_TarskiSeidenberg}).
For cyclically ordered quadruples we just add the condition about the positivity of double ratios, which is a semi-algebraic condition again by \Cref{propo_TripleDoubleRatioSemiAlg}, and conclude by  \Cref{thm_FlagCurveHitchinRepr} (\ref{thm_FlagCurveHitchinRepr_quadruplepos}).
This proves one direction.
\end{proof}
For the ``if" direction we establish \Cref{thm_BDFParam}, which we recall here.

\thmB*

\begin{proof}
For the proof we use \Cref{propo_positiveReprHaveDistinctEigenvalues} and the multiplicative variant of the Bonahon--Dreyer coordinates for parametrizing the Hitchin component as in \Cref{section_VariantOfBDCoordinates}.
Fix therefore the necessary topological data $\lambda$ on the surface $S$, see \Cref{subsection_BDCoord}.
We use the same notation as in \Cref{section_VariantOfBDCoordinates}.
Suppose the geodesic lamination $\lambda$ has $p$ closed leaves and $q$ infinite leaves, and its complement $S\setminus \lambda$ consists of $r$ ideal triangles.
Set
\[ N = 3r\tfrac{(n-1)(n-2)}{2} + (p+q)(n-1).\]
The idea of the proof is as follows.
We define a map $\Psi$ from the set of $\F$-positive representations in $\Hom(\pi_1(S),\PSL(n,\F))$ to $\F^N$ by associating to a representation the $\F$-valued triangle and shear invariants using the positive limit map as in \Cref{subsection_BDCoord}.
If $\rho \from \pi_1(S) \to \PSL(n,\F)$ is an $\F$-positive representation, we show that $\Psi(\rho)$ satisfies the relations \ref{item_PositivityCondition}-\ref{item_ClosedLeafInquality} of \Cref{subsection_BDRelations}.
In other words, $\Psi(\rho) \in \mathcal{X}_\F$---the $\F$-extension of the semi-algebraic subset $\mathcal{X} \subseteq \R^N$ that is homeomorphic to the Hitchin component $\Hit(S,n)$, compare \Cref{subsection_BDParam} and \Cref{subsection_BDExt}.
Lastly,  we prove that if $\Psi(\rho)=\Psi(\rho')$ for two $\F$-positive representations,  then $\rho$ and $\rho'$ are conjugate by an element of $\PGL(n,\F)$.

The triangle and shear invariants can be defined using a flag decoration $\partial_\infty \tilde{\lambda} \to \Flag(\F^n)$ (\Cref{rem_flagdecoration}).
Thus, in the same way as in (\ref{item_triangleInvariants})-(\ref{item_shearInvariantsClosedLeaves}) in \Cref{subsection_BDCoord} we can associate triangle and shear invariants to $\F$-positive representations using only the associated limit map restricted to $\partial_\infty \tilde{\lambda} \subseteq \Fix(S)$.
This defines the map $\Psi$.

Let us show that $\Psi$ satisfies \ref{item_CondPosTriangleCoord}-\ref{item_CondClosedLeafEquInequ} in \Cref{subsection_BDParam} over $\F$.
The positivity conditions \ref{item_CondPosTriangleCoord} and \ref{item_CondPosShearCoord} follow from the positivity of the limit map.
Note that the rotation condition \ref{item_CondRotation} follows directly since it is a property of triple ratios and how they behave under permutation of the flags, see \Cref{lem_PropertiesOfRatiosUnderPermutation} (\ref{lem_PropertiesOfRatiosUnderPermutation_item_TripleRatio}).
Left to show are therefore the closed leaf equality and closed leaf inequality \ref{item_CondClosedLeafEquInequ}.
For Hitchin representations, the proof of the closed leaf equalities relies on \cite[Proposition 13]{BonahonDreyer_ParametrizingHitchinComponents}.
The multiplicative version of this proposition for the coordinates associated to $\F$-positive representations can be proven in an analogue fashion; see also \cite[Remark 15]{BonahonDreyer_ParametrizingHitchinComponents} for a multiplicative version of the statement of \cite[Proposition 13]{BonahonDreyer_ParametrizingHitchinComponents}.

\begin{propo}
\label{propo_ClosedLeafInequality}
Let $\rho \from \pi_1(S) \to \PSL(n,\F)$ be an $\F$-positive representation with limit map $\xi_\rho$.
For every closed leaf $\gamma \subset \lambda$ and every $a=1, \ldots, n-1$, we have
\[L_a^\textnormal{right}(\gamma) = \frac{\lambda_a(\widetilde{\rho(\gamma)})}{\lambda_{a+1}(\widetilde{\rho(\gamma)})} = L_a^\textnormal{left}(\gamma) ,\]
where $\lambda_a(\widetilde{\rho(\gamma)})$ is the eigenvalue of a lift $\widetilde{\rho(\gamma)}$ of $\rho(\gamma)$ to $\SL(n,\F)$ corresponding to the eigenspace $\xi_\rho(\gamma^+)^{(a)} \cap \xi_\rho(\gamma^-)^{(n-a+1)}$.
\end{propo}
\begin{proof}
By \Cref{propo_positiveReprHaveDistinctEigenvalues}, we can lift $\rho(\gamma) \in \PSL(n,\F)$ to $\widetilde{\rho(\gamma)} \in \SL(n,\F)$ such that the matrix $\widetilde{\rho(\gamma)}$ is diagonalizable with eigenvalues $\lambda_a(\widetilde{\rho(\gamma)})\neq 0$,  and $\widetilde{\rho(\gamma)}$ acts by multiplication with $\lambda_a(\widetilde{\rho(\gamma)})$ on $\xi_\rho(\gamma^+)^{(a)} / \xi_\rho(\gamma^+)^{(a-1)}$.
Note that $\frac{\lambda_a(\widetilde{\rho(\gamma)})}{\lambda_{a+1}(\widetilde{\rho(\gamma)})}$ is independent of the choice of lift of $\rho(\gamma)$ since two lifts differ by multiplication with $\pm \Id_n$.

The rest of the proof is a computation.
It works in the same way as the one of \cite[Proposition 13]{BonahonDreyer_ParametrizingHitchinComponents} by replacing $\tau_{abc}$ by $T_{abc}$ and $\sigma_a$ by $D_a$ (where $\tau_{abc}=\log(T_{abc})$ and $\sigma_a= \log(D_a)$ following the notation from \cite{BonahonDreyer_ParametrizingHitchinComponents}) and by writing everything multiplicatively.
We refer to this reference for more details on the computations.
\end{proof}

The above proposition immediately implies the closed leaf equality in \ref{item_CondClosedLeafEquInequ}.
Furthermore it implies that the closed leaf inequality amounts to showing that 
\[ \frac{\lambda_a(\widetilde{\rho(\gamma)})}{\lambda_{a+1}(\widetilde{\rho(\gamma)})}  > 1 \]
for all $a =1, \ldots, n-1$.
This follows from the second statement in \Cref{propo_positiveReprHaveDistinctEigenvalues}, that says that $\rho$ is weakly dynamics-preserving, i.e.\ $\xi_\rho(\gamma^+)=F_{\rho(\gamma)}^+$.
Thus $\Psi$ satisfies conditions \ref{item_CondPosTriangleCoord}-\ref{item_CondClosedLeafEquInequ}, in other words $\Im(\Psi) \subseteq \mathcal{X}_\F$.
The forward direction of \Cref{thm_FHitchinEquivFPositive} implies that $\Im(\Psi)=\mathcal{X}_\F$.
Since the triple and double ratios are $\PGL(n,\F)$-invariant,  conjugate $\F$-positive representations will have the same value under $\Psi$, see \Cref{rem_invariantswelldefined}.

We are left to show that the coordinates parametrize the set of conjugacy classes of $\F$-positive representations.
Let thus $\rho, \rho' \from \pi_1(S) \to \PSL(n,\F)$ be two $\F$-positive representations with limit map $\xi_\rho$ respectively $\xi_{\rho'}$,  and assume $\Psi(\rho)=\Psi(\rho')$.
We prove that $\rho$ and $\rho'$ are conjugate in $\PGL(n,\F)$.
Let $\Psi(\rho)=\Psi(\rho')= (x_{abc,t,v},y_{m,l}) \in \F^N$, where $a,b, c \geq 1 $ are integers with $a+b+c=n$,  $v$ is a vertex of a component $t \subset S \setminus \lambda$,  $m=1,\ldots,n-1$ and $l$ is a leaf of $\lambda$.
The idea of the proof is to reconstruct a flag decoration $\mathcal{F} \from \partial_\infty\tilde{\lambda}\to \Flag(\F^n)$ from the triple and double ratios,  which will agree up to an element of $\PGL(n,\F)$ with $\xi_\rho$ and $\xi_{\rho'}$,  when the latter are restricted to $\partial_\infty \tilde{\lambda}$.
Recall the notations from (\ref{item_triangleInvariants})-(\ref{item_shearInvariantsClosedLeaves})
in \Cref{subsection_BDCoord} in the definition of the Bonahon--Dreyer coordinates.
The statement of the following lemma in the real case is \cite[Lemma 24]{BonahonDreyer_ParametrizingHitchinComponents}.
It can be proven in a complete analogue fashion for general real closed fields by writing everything multiplicatively.
For the uniqueness result use \Cref{propo_StabilizerTripleFlags}.

\begin{propo}
\label{propo_FlagDecFromCoord}
Let $(x_{abc,t,v},y_{m,l}) \in \mathcal{X}_\F$.
Then there exists a flag decoration $\mathcal{F} \from \partial_\infty \tilde{\lambda} \to \Flag(\F^n)$ such that
\begin{enumerate}
\item 
$x_{abc,t,v} = T^{\mathcal{F}}_{abc}(t,v)$ for every component $t \subset S \setminus \lambda$, for every vertex $v$ of $t$, and integers $a,b,c \geq 1$ with $a+b+c=n$;
\item
$y_{m,l} = D^{\mathcal{F}}_m(l)$ for every leaf $l \subset \lambda$ and integer $m=1,\ldots,n-1$.
\end{enumerate}
Furthermore, $\mathcal{F}$ is unique up to postcomposition by an element of $\PGL(n,\F)$.
\end{propo}

Since $\Psi(\rho)=\Psi(\rho')$,  the uniqueness statement in the above lemma implies that there exists $g \in \PGL(n,\F)$ with $g \xi_\rho|_{\partial_\infty \tilde{\lambda}} = \xi_{\rho'}|_{\partial_\infty \tilde{\lambda}}$.
Thus for $x \in \partial_\infty\tilde{\lambda}$ and for all $\gamma \in \pi_1(S)$ we have
\[(g \rho(\gamma) g^{-1}) ( \xi_{\rho'}(x)) = g \rho(\gamma) \xi_\rho(x) = g\xi_{\rho}(\gamma x)=\xi_{\rho'}(\gamma x).\]
Similarly, for a triple of distinct points $x,y,z \in \partial_\infty \tilde{\lambda}$ we have for all $\gamma \in \pi_1(S)$
\[(g \rho(\gamma) g^{-1})( \xi_{\rho'}(x), \xi_{\rho'}(y),\xi_{\rho'}(z))= (\xi_{\rho'}(\gamma x),\xi_{\rho'}(\gamma y),\xi_{\rho'}(\gamma z)),\]
which implies that $g \rho(\gamma) g^{-1} = \rho'(\gamma)$ by $\rho'$-equivariance of $\xi_{\rho'}$, transversality of $\xi_{\rho'}$ and \Cref{propo_StabilizerTripleFlags}.
Thus $\rho$ and $\rho'$ are conjugate,  which finishes the proof.
\end{proof}

\begin{corol}
\label{propo_SameCoordImplyConj}
Let $\rho, \rho' \from \pi_1(S) \to \PSL(n,\F)$ be two representations, where $\rho$ is $\F$-positive and $\rho'$ is $\F$-Hitchin.
Assume that $\Psi(\rho)=\pr^\BD_\F(\rho')$.
Then $\rho$ and $\rho'$ are conjugate in $\PGL(n,\F)$.
\end{corol}
\begin{proof}
Denote the limit map of $\rho'$ constructed in the proof of the forward direction of \Cref{thm_FHitchinEquivFPositive} by $\xi_{\rho'}$, and let $\xi_\rho$ denote the limit map of $\rho$.
Viewing $\rho'$ as an $\F$-positive representation, we see immediately from the definitions of $\Psi$ and $\pr^\BD_\F$ that
\[ \Psi(\rho')=\pr^\BD_\F(\rho')=\Psi(\rho).\]
We conclude using \Cref{thm_BDFParam}.
\end{proof}

We can now finish the proof of \Cref{thm_FHitchinEquivFPositive}.

\begin{proof}[Proof of \Cref{thm_FHitchinEquivFPositive} ($\impliedby$)]
Let $\rho$ be an $\F$-positive representation.
Then its $\PGL(n,\F)$-equivalence class corresponds under the identification from \Cref{thm_BDFParam} to a point in $\mathcal{X}_\F$.
By \Cref{corol_CoordBDF} there exists \linebreak$\rho' \in \Hom_\Hit(S,\PSL(n,\F))$ with $\Psi(\rho) = \pr^\BD_\F(\rho')$.
\Cref{propo_SameCoordImplyConj} implies that $\rho$ and $\rho'$ are conjugate in $\PGL(n,\F)$, which was to prove.
\end{proof}

%% file: sections/09_propertiesFHit.tex
\section{Properties of $\F$-Hitchin representations}
\label{section_PropBoundaryRepr}
In this section we collect additional properties of $\F$-Hitchin representations.
For properties that are invariant under $\PGL(n,\F)$-conjugation we can use the notions $\F$-positive and $\F$-Hitchin representations interchangeably by \Cref{thm_FHitchinEquivFPositive} .
Most properties follow from the equivalent statements for real Hitchin representations and the Tarski--Seidenberg principle (\Cref{thm_TarskiSeidenberg}).

\subsection{Lifting to $\SL(n,\R)$}
We recall the notions of \Cref{subsection_FHit}.
The goal is to show that $\F$-Hitchin representations lift to $\SL(n,\F)$.

\begin{dfn}
We define
\[\Hom_\Hit(\pi_1(S),\SL(n,\F))\coloneqq \Ad_\F^{-1}\left(\Hom_\Hit(\pi_1(S),\PSL(n,\F))\right)\]
as the set of \emph{$\F$-Hitchin representations to $\SL(n,\F)$} (as opposed to $\PSL(n,\F)$).
\end{dfn}

\begin{propo}
\label{propo_FHitchinReprLift}
The set $\Hom_\Hit(\pi_1(S),\SL(n,\F))$ is non-empty and semi-algebraic.
Furthermore, every $\F$-Hitchin representation lifts to an $\F$-Hitchin representation into $\SL(n,\F)$.
\end{propo}
\begin{proof}
Goldman proved that representations $j \from \pi_1(S)\to \PSL(2,\R)$ that are discrete and faithful lift to representations into $\SL(2,\R)$ \cite[Theorem A]{Goldman_Thesis},  and hence $\iota_n \circ j$ lifts to $\SL(n,\R) \subseteq \SL(n,\F)$ (\Cref{dfn_HitchinComp}).
Thus $\Hom_\Hit(\pi_1(S),\SL(n,\F))$ is non-empty.
The set $\Hom_\Hit(\pi_1(S),\SL(n,\F))$ is semi-algebraic as it is the preimage of a semi-algebraic set under a semi-algebraic map, see \Cref{propo_ImSemiAlgMap}. 

Since the set of Hitchin representations is connected, it follows that all Hitchin representations lift.
In other words,  the map $\Ad$, when restricted to the set of Hitchin representations into $\SL(n,\R)$,
\[\Ad \from \Hom_\Hit(\pi_1(S),\SL(n,\R)) \to \Hom_\Hit(\pi_1(S),\PSL(n,\R))\]
is surjective.
If now $\F$ is a real closed field extension of $\R$, the extension of the map $\Ad$ to $\F$, when restricted to $\F$-Hitchin representations into $\SL(n,\F)$,
\[\Ad_\F \from \Hom_\Hit(\pi_1(S),\SL(n,\F)) \to \Hom_\Hit(\pi_1(S),\PSL(n,\F))\]
is still surjective,  see \Cref{thm_ExtSemiAlgMaps}.
Hence every $\F$-Hitchin representation is the image of an $\F$-Hitchin representation into $\SL(n,\F)$ under the map $\Ad_\F$, and thus every $\F$-Hitchin representation lifts to an $\F$-Hitchin representation into $\SL(n,\F)$.
\end{proof}

\begin{rem}
Note that $\Hom_\Hit(\pi_1(S),\SL(n,\F))$ is only semi-algebraically connected if $n$ is odd.
Otherwise,  it is a union of $2^{2g}$ semi-algebraically connected components,  where $g$ is the genus of $S$.
This follows from the equivalent statement over $\R$ and \Cref{thm_ExtConnComp}.
\end{rem}

\subsection{Irreducibility}
\label{section_FHitchinIrr}
Let $\F$ be any field and $\Gamma$ a finitely generated group.
\begin{dfn}
A representation $\rho \from \Gamma \to \GL(n,\F)$ is \emph{irreducible} if the only $\rho(\Gamma)$-invariant subspaces of  $\F^n$ are $\{0\}$ and $\F^n$.
Denote by $\Gr(\F^n)$ the set of all non-trivial proper subspaces of $\F^n$.
A representation $\rho \from \Gamma \to \PGL(n,\F)$ is \emph{irreducible} if its action on $\Gr(\F^n)$ has no fixed point.
\end{dfn}

\begin{rem}
Let $\gamma_1, \ldots, \gamma_k$ be a finite generating set for $\Gamma$ and $W \subseteq \F^n$ a vector subspace.
Then $\rho(\Gamma) W \subseteq W$ if and only if $\rho(\gamma_j) W \subseteq W$ for all $j=1,\ldots,k$.
A similar statement holds true for projective representations.
\end{rem}

Let now $\F$ be real closed.
We would like to show that $\F$-Hitchin representations when restricted to every finite index subgroup are irreducible.
We already know by \cite[Lemma 10.1]{Labourie_AnosovFlowsSurfaceGroupsAndCurvesInProjectiveSpace} that real Hitchin representations have this same property.
Thus the following proposition is an easy consequence of the Tarski--Seidenberg transfer principle.

\begin{propo}
\label{propo_FHitchinIrr}
Let $\rho \from\pi_1(S) \to\PSL(n,\F)$ be $\F$-Hitchin.
Then the restriction of $\rho$ to every finite index subgroup is irreducible.
\end{propo}
\begin{proof}
Let $\gamma_1,\ldots,\gamma_{2g}$ be a finite set of generators for $\pi_1(S)$.
We saw in \Cref{lem_FlagSpacesAlg} that $\Gr(k,\R^n)$ is algebraic for all $k=1,\ldots,n-1$.
Since $\Gr(\R^n)$ is the finite union of $\Gr(1,\R^n), \ldots, \Gr(n-1,\R^n)$ it is also algebraic.
Set $\Gr \coloneqq \Gr(\R^n)$.
Consider the real semi-algebraic set
\[ \{ (\rho', V)\in \Hom_\Hit(\pi_1(S),\PSL(n,\R)) \times \Gr \mid \rho'(\gamma_1)V=V, \ldots, \rho'(\gamma_{2g})V=V\}.\]
Then \cite[Theorem 10.1]{Labourie_AnosovFlowsSurfaceGroupsAndCurvesInProjectiveSpace} implies that this set is empty.
By the Tarski--Seidenberg transfer principle (\Cref{thm_TarskiSeidenberg}), so is its $\F$-extension.
Since the $\F$-extension $\Gr_\F$ of $\Gr$ is exactly $\Gr(\F^n)$, this implies that $\F$-Hitchin representations are irreducible.
The restriction of $\rho$ to every finite index subgroup of $\pi_1(S)$ corresponds to an $\F$-Hitchin representation of the corresponding cover of $S$,  thus the restricted representation is also irreducible.
\end{proof}

\subsection{Discreteness}
We finish by proving that $\F$-Hitchin representations are discrete.
The order on a real closed field induces a topology, which is Hausdorff.
However real closed fields different from $\R$ are not locally compact,  hence discrete representations have a more complicated behaviour.
It might be desirable to replace this notion in the context of real closed fields with a better suited candidate.

The proof follows \cite[Theorem 1.10, Theorem 1.13 (i)]{FockGoncharov_ModuliSpacesLocalSystemsHigherTeichmuellerTheory}.
In the case of Teichm\"uller space this was proven in \cite[Proposition 6.3]{Brumfiel_RSCTeichmullerSpace}.

\begin{propo}
\label{propo_FHitchinDiscrete}
Let $\F$ be a real closed field and $\rho \from \pi_1(S) \to \PSL(n,\F)$ an $\F$-positive representation.
Then $\rho$ has discrete image.
\end{propo}
\begin{proof}
Let $\lambda$ be a maximal geodesic lamination on $S$ and $\tilde{\lambda}$ its lift to the universal cover $\tilde{S}$ of $S$.
Choose a component $\widetilde{t}$ of $\tilde{S} \setminus \tilde{\lambda}$.
Denote its vertices in $\partial_\infty\tilde{S}$ by $a,b,c$ ordered in clockwise direction.
Let $\xi_\rho$ be the positive limit map of $\rho$.
Since $\xi_\rho$ is positive, the tuple $(\xi_\rho(a),\xi_\rho(b),\xi_\rho(c))$ is positive,  thus there exists a basis $\mathcal{B}^{ac}$ of $\F^n$ in which $\xi_\rho(a)$ is the ascending flag $F^+$ and $\xi_\rho(c)$ is the descending flag $F^-$, and $u \in\mathcal{U}^\lltriangle_{>0}(\mathcal{B}^{ac})$ such that $\xi_\rho(b) = u F^+$,  see \Cref{lem_PoskTupleNormalForm}.
By a similar argument as in the proof of \Cref{propo_StabilizerTripleFlags},  this basis is unique up to scaling each basis vector by a positive number.
The subset $\mathcal{F}_1^{ac} \coloneqq \{u F^+ \mid u \in\mathcal{U}^\lltriangle_{>0}(\mathcal{B}^{ac})\}\subset \Flag(\F^n)$ is open and contains $\xi_\rho(b)$,  thus there exists an open neighbourhood $U_{\xi_\rho(b)} \subset \mathcal{F}_1^{ac}$ of $\xi_\rho(b)$ and therefore an open neighbourhood $O^{ac}_b \subset \PSL(n,\F)$ of the identity such that $g\xi_\rho(b) \in U_{\xi_\rho(b)}$ for all $g \in O^{ac}_b$.
Exchanging the roles of $a$, $b$ and $c$ we can find neighbourhoods $O^{ca}_b, O^{ab}_c, O^{ba}_c, O^{bc}_a,O^{cb}_a$ in $\PSL(n,\F)$ of the identity with the same property.
Let $O \coloneqq O^{ac}_b \cap O^{ca}_b \cap O^{ab}_c \cap O^{ba}_c \cap O^{bc}_a \cap O^{cb}_a$.
Then $O \subset \PSL(n,\F)$ is an open neighbourhood of the identity.
We show that $\rho(\gamma) \neq O$ for all $e \neq \gamma \in \pi_1(S)$,  which proves that $\rho$ is discrete.

Let $e\neq \gamma \in \pi_1(S)$.
Let $a'=\gamma a$, $b' =\gamma b$ and $c'=\gamma c$ denote the vertices of $\gamma \widetilde{t}$.
Then $\gamma \widetilde{t}$ is a component of $\tilde{S} \setminus \tilde{\lambda}$, thus $\widetilde{t}$ and $\gamma\widetilde{t}$ do not have common interior points.
The points $a',b',c'$ hence lie in the closure of exactly one component of $\partial_\infty \tilde{S} \setminus \{a,b,c\}$.
Assume that $a',b',c'$ lie on the (closed) arc between $a$ and $c$ that does not contain $b$ as in \Cref{fig_discreteness}.
\begin{figure}[H]
\centering
\includegraphics[width=0.35\textwidth]{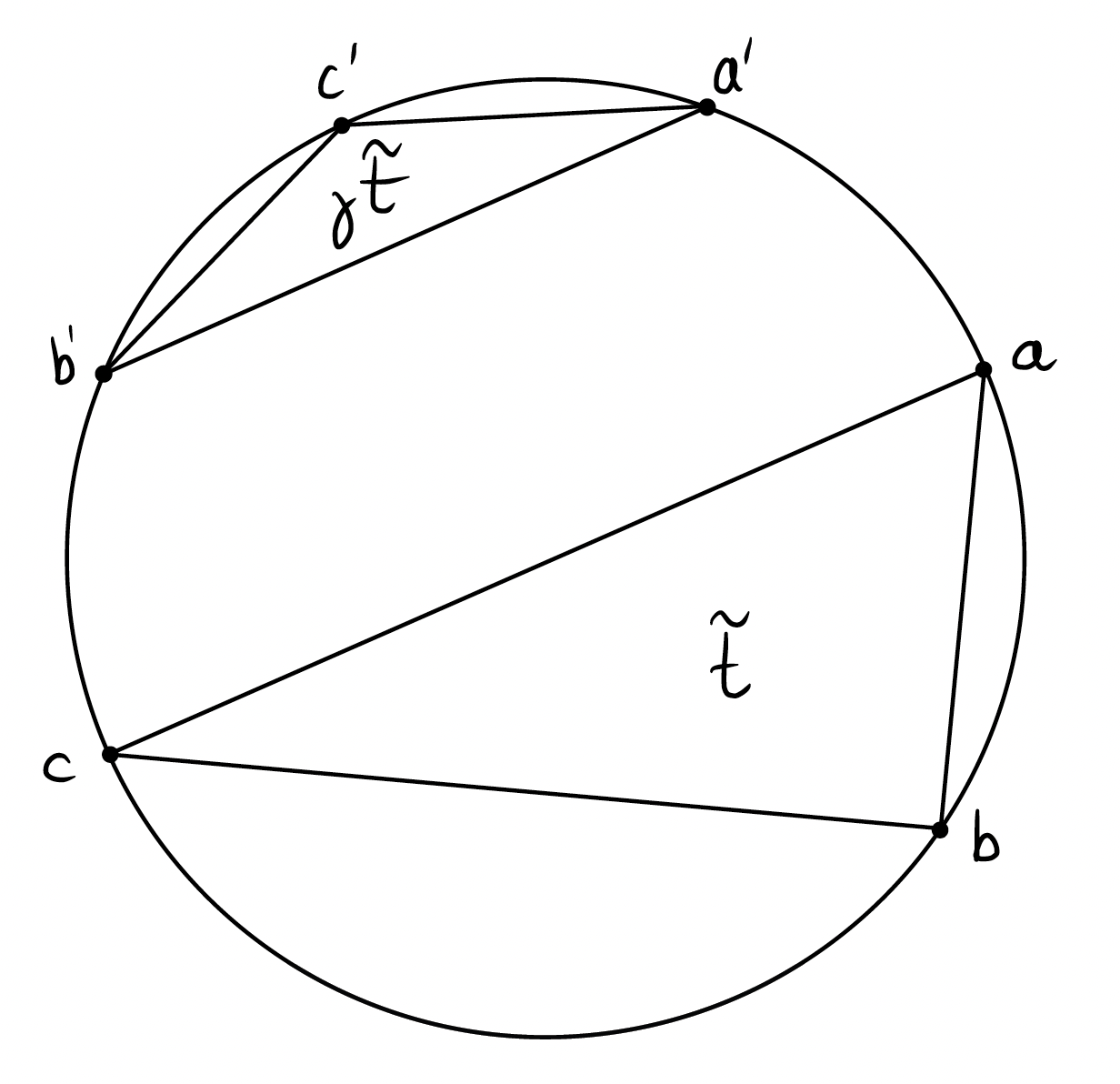}
\captionof{figure}{The component $\widetilde{t}$ and its $\gamma$-translate.} \label{fig_discreteness}
\end{figure}
\noindent It is possible that $b'$ agrees with $a$ or $c$.
Assume first that $b'$ does not agree with either point.
Since $\xi_\rho$ is positive, the tuple $(\xi_\rho(a),\xi_\rho(b),\xi_\rho(c),\xi_\rho(b'))$ is positive.
By the uniqueness (up to positive scaling of each basis vector) of $\mathcal{B}^{ac}$ and  \Cref{lem_PoskTupleNormalForm} there exists $v \in \mathcal{U}^\urtriangle_{>0}(\mathcal{B}^{ac})$ such that $\xi_\rho(\gamma b) = v^{-1}F^-$.
Consider the set $\mathcal{F}_1 \coloneqq \mathcal{F}_1^{ac} = \{u F^+ \mid u \in\mathcal{U}^\lltriangle_{>0}(\mathcal{B}^{ac})\}$ as before,  and define the open subset $\mathcal{F}_2 \coloneqq \{v^{-1} F^- \mid v \in\mathcal{U}^\urtriangle_{>0}(\mathcal{B}^{ac})\}$ of $\Flag(\F^n)$.
Then $\mathcal{F}_1 \cap \mathcal{F}_2= \emptyset$: Indeed if $uF^+ = v^{-1} F^-$, then $vuF^+=F^-$. 
Since $vu$ is totally positive (see \Cref{corol_TotPosProduct} (\ref{corol_NonNegTotPosProduct_item2})), this is impossible.
We note that $\xi_\rho(b) \in \mathcal{F}_1$ and $\xi_\rho(b') \in \mathcal{F}_2$.
In particular $\xi_\rho(b') = \rho(\gamma) \xi_\rho(b) \notin U_{\xi_\rho(b)}$, which is equivalent to saying that $\rho(\gamma) \notin O^{ac}_b$.
Thus $\rho(\gamma) \notin O$.

If $b'=c$, then $\xi_\rho(b') = \xi_\rho(c) =F^-$,  but $F^- \notin \mathcal{F}_1$, thus $\xi_\rho(b')\notin U_{\xi_\rho(b)}$,  hence $\rho(\gamma) \notin O^{ac}_b$.
If $b'=a$, then we consider the basis $\mathcal{B}^{ca}$  instead of $\mathcal{B}^{ac}$ and we conclude $\rho(\gamma) \notin O^{ca}_b$ by the same argument, and hence $\rho(\gamma) \notin O$.

If $a',b',c'$ lie in the closure of a different component of $\partial_\infty \tilde{S} \setminus \{a,b,c\}$ then we do the same argument with the appropriate bases and open sets.
\end{proof}

%% file: sections/geodesiccurrents.tex
\section{Geodesic currents for $\F$-positive representations}
\label{section_GeodCurrentsFPosRepr}

The goal of this section is to show how to associate geodesic currents to $\F$-positive representations.
This result was announced in \cite[Theorem 47]{BurgerIozziParreauPozzetti_RSCCharacterVarieties}.
The analogue version was proven for maximal representations by Burger–Iozzi–Parreau–Pozzetti \cite[Theorem 1.2]{BIPP_PositiveCR},  and recently generalized for positively ratioed representations \cite[Corollary 9.3]{BurgerIozziParreauPozzetti_RSCCharacterVarieties2}.

\subsection{Positive cross-ratios and  intersection currents}
\color{black}
More information on geodesic currents can be found in \cite{ErlandssonSouto_MirzakhaniCurveCountingGeodCurrents} or \cite{Martelli_IntroGeomTop}.

As was hinted at in the introduction, the space of geodesic currents $\mathscr{C}(S)$, see \Cref{dfn_GeodesicCurrent},  comes equipped with an intersection form generalizing the geometric intersection number.
Recall that we endow $S$ with an auxiliary hyperbolic structure.
Denote by $\mathcal{G}$ the space of (unoriented and unparametrized) geodesics of $\tilde{S}$.



\begin{dfn}
\label{dfn_IntersectionGeodCurr}
Consider $\mathcal{G}^{(2)} \subseteq \mathcal{G}\times \mathcal{G}$ the set of pairs of transversely intersecting geodesics.
We define the \emph{intersection form} 
\[i: \mathscr{C}(S)\times\mathscr{C}(S) \to \mathbb{R}_{\geq 0}, \quad i(\mu,\eta):=(\mu \times \eta)(\mathcal{G}^{(2)}/\pi_1(S)),\]
where $\mu \times \eta$ is the product measure.
\end{dfn}

Bonahon proved that $i$ is finite, continuous, symmetric and bilinear, and generalizes the geometric intersection number of homotopy classes of closed curves on $S$ \cite[Proposition 4.5]{Bonahon_BoutsVarHypDim3}.


To associate to a representation in a higher rank Teichm\"uller space a geodesic current, we make use of positive cross-ratios.
There are many (non-equivalent definitions) of cross-ratios.
The convention of how to arrange the arguments in the cross-ratio follows \cite[Definition 2.4]{MartoneZhang_PositivelyRatioedRepr}.
However the cross-ratio is defined on a smaller set and is not assumed to be continuous as in \cite[Definition 3.1]{BIPP_PositiveCR}.
Compare also to \cite[Remark 3.2]{BIPP_PositiveCR} for a comparison between the various definitions in the literature.

\begin{dfn}
\label{dfn_PositiveCR}
Let $X\subseteq \partial_\infty\tilde{S}$ be a $\pi_1(S)$-invariant non-empty subset,  e.g.\ $\Fix(S)$,  and let $X^{[4]}$ denote the set of positively oriented, i.e.\ cyclically ordered (in clockwise direction),  quadruples in $X$.
A \emph{cross-ratio} is a $\pi_1(S)$-invariant function 
\[B \from X^{[4]} \to \R,\]
that is 
\begin{enumerate}
\item symmetric, i.e.\ for all $x=(x,y,z,w) \in X$ we have $B(x,y,z,w)=B(z,w,x,y)$, and
\item additive, i.e.\ $B(x,y,z,w) + B(x,y,w,t) = B(x,y,z,t)$ for all  \linebreak $x,y,z,w,t \in X$ that are positively oriented.
\end{enumerate}
A cross-ratio $B$ is \emph{positive} if $B(x,y,z,w) \geq 0$ for all $(x,y,z,w) \in X^{[4]}$.
Denote the set of positive cross-ratios on $X$ by $\CR^+(X)$.
The \emph{$B$-period} of a non-trivial $\gamma \in \pi_1(S)$ with $\gamma^\pm \in X$ is 
\[\ell_B(\gamma) \coloneqq B(\gamma^+, \gamma^-,x,\gamma x)\]
for some (any) $x \in X \setminus \{\gamma^{\pm}\}$ such that $(\gamma^+, \gamma^-,x,\gamma x) \in X^{[4]}$, where $\gamma^+$ respectively $\gamma^-$ is the attracting respectively repelling fixed point of $\gamma$.
A geodesic current $\mu \in \mathscr{C}(S)$ is an \emph{intersection current for $B$} if for every non-trivial $\gamma \in \pi_1(S)$ with $\gamma^\pm\in X$ we have $\ell_B(\gamma) = i(\mu,\gamma)$, where we view $\gamma$ as a geodesic current; see e.g.\ \cite[\S 1]{Bonahon_GeometryTeichmuellerSpaceGeodesicCurrents}.
\end{dfn}

The Liouville current is an intersection current for the pull-back of the standard cross-ratio on $\partial \H^2$ via the developing map \cite[Proposition 14]{Bonahon_GeometryTeichmuellerSpaceGeodesicCurrents}.

For a general cross-ratio, intersection currents do not need to exist.
If we assume $X = \partial_\infty\tilde{S}$ and $B$ continuous, then an observation of Hamenst\"adt in \cite{Hamenstaedt_CocyclesSymplStrInters}, refer to \cite[Appendix A]{MartoneZhang_PositivelyRatioedRepr} for a detailed proof,  implies that if $B$ is positive, then there exists a unique intersection current for $B$. 
More generally, we have the following.

\begin{theor}[{\cite[Theorem 1.6]{BIPP_PositiveCR}}]
\label{thm_BIPPGeodCurrAssToPosCR}
Let $X \subseteq \partial_\infty\tilde{S}$ be a $\pi_1(S)$-invariant non-empty subset and $B$ a positive cross-ratio on $X$. 
Then there is a geodesic current $\mu$ on $S$ such that for all $e \neq \gamma \in \pi_1(S)$
\[\ell_B(\gamma) = i(\mu,\gamma).\]
The geodesic current $\mu$ depends continuously on the cross-ratio $B$, where $\CR^+(X)$ is endowed with the subspace topology of the topological vector space of cross-ratios on $X$ with the topology of pointwise convergence.
\end{theor}

%% file: sections/associatinggeodcurrentsHitchin.tex
\subsection{Valuations and big elements}

For an introduction to valuations we refer the reader to \cite{EnglerPrestel_ValuedFields}.

\begin{dfn}
\label{dfn_OrderCompVal}
Let $\F$ be a field.
A \emph{valuation} on $\F$ is a map 
\[\upsilon \from \F \to \R \cup \{\infty\}\]
such that $e^{-\upsilon} \from \F \to \R$ is a norm that satisfies the triangle inequality.
The valuation is \emph{trivial} if $\upsilon(a)=0$ for all $a \in \F^\times$, otherwise \emph{non-trivial}.

If $\F$ is ordered, we say that a valuation $\upsilon$ is \emph{order-compatible} if for all $0 < x \leq y \in \F$ we have $\upsilon(x) \geq \upsilon(y)$.
We say that $\upsilon$ is \emph{non-Archimedean} if $\upsilon(x+y) \geq \min\{\upsilon(x),\upsilon(y)\}$.
Two valuations $\upsilon$ and $\upsilon'$ are \emph{equivalent} if there exists $r \in \R$ positive with $\upsilon=r \upsilon'$.
\end{dfn}

\begin{lem}
\label{lem_ValuationsProperties}
Let $\upsilon$ be a valuation on $\F$.
Then $\upsilon(1)=0$,  $\upsilon(\zeta)=0$ for $\zeta$ a root of unity in $\F$,  $\upsilon(x^{-1})=-\upsilon(x)$ and $\upsilon(-x)=\upsilon(x)$ for all $x \in \F$.
If $\upsilon$ is non-Archimedean and $x,y \in \F$ with $\upsilon(x) \neq \upsilon(y)$, then $\upsilon(x+y)=\min\{\upsilon(x),\upsilon(y)\}$.
If $\F$ is ordered and $\upsilon$ is non-Archimedean order-compatible, then for $x,y \in \F_{>0}$ we have $\upsilon(x+y)=\min\{\upsilon(x),\upsilon(y)\}$.
\end{lem}
\begin{proof}
For the first statements see {\cite[Section 1.3 and (1.3.4)]{EnglerPrestel_ValuedFields}}.
For the last statement we can without loss of generality assume that $\upsilon(x)<\upsilon(y)$.
We already know $\upsilon(x+y) \geq \upsilon(x)$.
Since $x,y >0$ we have $x+y \geq x$,  and thus by order-compatibility of $\upsilon$ we have $\upsilon(x+y) \leq \upsilon(x)$.
\end{proof}

\begin{examp}
\label{examp_OrderCompVal}
The function $-\log$ is an order-compatible valuation on $\R$.
The map $\R(X) \to \R\cup \{\infty\}$, $\tfrac{p}{q} \mapsto \deg(q)-\deg(p)$,  where $\deg$ denotes the degree of the polynomial, is an order-compatible non-Archimedean valuation on $\R(X)$ with the order $+\infty$, compare \Cref{examp_RSpR}.
\end{examp}

\begin{dfn}
Let $\F$ be an ordered field.
A positive element $b \in \F$ is \emph{big} if for all $x \in\F$ there exists $n \in \N$ such that $x < b^n$.
For $b$ a big element, the \emph{logarithm with basis $b$} is defined by
\[\log_{b} \from \F_{>0} \to \R, \quad x \mapsto \inf \{ q \in \Q \mid x \leq b^q\}.\]
\end{dfn}

Note that a big element is always bigger than $1$.

\begin{lem}[{\cite[Proposition 5.2.(f)-(g)]{Brumfiel_RSCTeichmullerSpace}}]
\label{lem_ValBigElem}
Let $b$ be a big element in an ordered field $\F$.
To $b$ we can associate the following non-trivial order-compatible valuation
\[ \upsilon_b \from \F \to \R \cup \{\infty\}, \quad x \mapsto
\begin{cases}
-\log_b(|x|) &\textrm{ if } x \neq 0,\\
\infty &\textrm{ if } x =0.
\end{cases}
\]
If $\F$ is non-Archimedean, so is $\upsilon_b$.
\end{lem}

\begin{lem}[{\cite[Proposition 5.2.(d)]{Brumfiel_RSCTeichmullerSpace}}]
\label{lem_LogEquiv}
Let $b$ and $b'\in \F$ be two big elements.
Then $\log_{b'} = \log_{b'}(b) \cdot \log_b$, i.e.\ the logarithms differ by a positive scalar multiple.
\end{lem}

It follows that the associated valuations $\upsilon_b$ and $\upsilon_b'$ are equivalent (with scaling factor $-\upsilon_{b'}(b)=\log_{b'}(b)>0$).

\begin{examp}
In $\R$ every element larger than $1$ is a big element in the above sense.
If $\F$ is an Archimedean ordered field, i.e.\ a subfield of $\R$,  then $\log_b$ is the usual logarithm, for which the last lemma is known.

In $\R(X)$ with the order $+\infty$,  $X$ is a big element.
The valuation $\upsilon_X$ associated to $X$ is the one from \Cref{examp_OrderCompVal}.
If $\F$ is any non-Archimedean field then no rational number is big.
\end{examp}

\begin{lem}
\label{lem_ValUnique}
Let $\F$ be an ordered field with a big element $b \in \F$ and $\upsilon$ any order-compatible valuation on $\F$.
Then 
\[\upsilon = - \upsilon(b) \cdot \upsilon_b.\]
In other words, up to scaling, all order-compatible valuations come from the construction in \Cref{lem_ValBigElem}, hence are equivalent by \Cref{lem_LogEquiv}.
\end{lem}
\begin{proof}
Let $x \in \F$ be positive.
Then $\upsilon_b(x) = -\inf \{q \in \Q \mid x \leq b^q\}$.
Assume $x \leq b^q$,  then $\upsilon(x) \geq q \upsilon(b)$ by order-compatibility of $\upsilon$.
Since $b$ is big, $\upsilon(b) <0$ and hence $\upsilon(x)/\upsilon(b) \leq q$.
This implies that $\upsilon(x)/\upsilon(b) \geq -\upsilon_b(x)$, as the latter is defined as the infimum over all such $q$.

Take now $q' \in \Q$ with $q'<\upsilon(x)/\upsilon(b)$. 
Assume $x \leq b^{q'}$.
By order-compatibility $\upsilon(x) \geq q' \upsilon(b)$ and hence $\upsilon(x)/\upsilon(b) \leq q'$, a contradiction.
Thus $\upsilon(x) = -\upsilon(b) \cdot \upsilon_b(x)$.
By \Cref{lem_ValuationsProperties} the same conclusion holds for $x<0$, since $\upsilon(-x)=\upsilon(x)$.
\end{proof}

\begin{lem}[{\cite[\S 5]{Brumfiel_RSCTeichmullerSpace},  \cite[Chapter VI, §10]{Bourbaki_CommAlg}}]
\label{lem_FinTransDegreeBigElements}
Let $\K \subseteq \F$ be an ordered field of finite transcendence over a subfield $\K$ which contains a big element.
Then $\F$ contains a big element.
\end{lem}

\subsection{Positive cross-ratios for $\F$-positive representations}

\begin{theor}[{\cite[Proposition 2.24, Theorem 3.4]{MartoneZhang_PositivelyRatioedRepr}}]
\label{thm_MZHitchinPosRatioed}
Let $\rho \from \pi_1(S) \to \PSL(n,\R)$ be a Hitchin representation.
Then for all $k=1,\ldots, n-1$,  there exists a unique cross-ratio $B_k^\rho$ such that for all $\gamma \in \pi_1(S)$ non-trivial
\[\ell_{B_k^\rho}(\gamma) = \log \frac{\lambda_1(\rho(\gamma))\cdot \ldots\cdot\lambda_k(\rho(\gamma))}{\lambda_{n-k+1}(\rho(\gamma))\cdot \ldots\cdot\lambda_n(\rho(\gamma))},\]
where $\lambda_1(\rho(\gamma)) > \ldots> \lambda_n(\rho(\gamma))>0$ are the absolute values of the eigenvalues of $\rho(\gamma)$.
Furthermore,  the cross-ratio $B_k^\rho$ is positive.
\end{theor}

In particular,  it follows from this together with \Cref{thm_BIPPGeodCurrAssToPosCR}, that there exists an intersection current $\mu_\rho^k$ for $\rho$.
Similar results hold for maximal representations in $\textrm{Sp}(2n,\R)$ \cite[Section 4.2]{Labourie_CRAnosovReprEnergyFunc} and $\Theta$-positive representations in $\textrm{PO}(p,q)$ \cite[Theorem 4.9]{BeyrerPozzetti_PosSurfaceGroupReprPOpq}.

The following lemma is a generalization of the above theorem to $\F$-Hitchin representations.
Recall that a representation $\rho \from \pi_1(S) \to \PSL(n,\F)$ is called $\F$-Hitchin if it lies in the $\F$-extension of the real semi-algebraic Hitchin component $\Hit(S,n)$, see \Cref{dfn_FHitchin}.
Let now $\F$ be a real closed field together with a non-trivial order-compatible valuation $\upsilon$.
Recall from the introduction that for $g \in \SL(n,\F)$ totally hyperbolic and for $k=1,\ldots,n-1$ we define the \emph{$k$-length of $g$} as
\[L_k(g) \coloneqq -\sum_{j=1}^k \upsilon(\lambda_j(g)) + \sum_{j=n-k+1}^n \upsilon(\lambda_j(g)).\]

\begin{lem}
\label{lem_PosCRFHitchin}
Let $\F\supseteq \R$ be a real closed field with non-trivial order-compatible valuation $\upsilon$ (if $\F=\R$ we assume $\upsilon=-\log$).
Let $\rho \from \pi_1(S) \to \PSL(n,\F)$ be $\F$-positive with associated $\F$-positive limit map $\xi_\rho \from \Fix(S) \to \Flag(\F^n)$.
Then for all $k=1,\ldots,n-1$
\[B_k^\rho \coloneqq \tfrac{1}{2}\big(\tilde{B}_k^\rho +\tilde{B}_{n-k}^\rho\big),\]
is a positive cross-ratio on $\Fix(S)$, where for $x=(x_1,\ldots,x_4) \in \Fix(S)^{[4]}$ we define
\begin{align*}
&\widetilde{M}_k^\rho(x) \coloneqq \frac{\xi_\rho(x_1)^{(n-k)} \wedge \xi_\rho(x_3)^{(k)}}{\xi_\rho(x_1)^{(n-k)} \wedge \xi_\rho(x_4)^{(k)}} \cdot \frac{\xi_\rho(x_2)^{(n-k)} \wedge \xi_\rho(x_4)^{(k)}}{\xi_\rho(x_2)^{(n-k)} \wedge \xi_\rho(x_3)^{(k)}},\textrm{ and } \\
&\tilde{B}_k^\rho(x) \coloneqq -\upsilon(\widetilde{M}_k^\rho(x)).
\end{align*}
Here we adapt the same notation as in \Cref{section_PreliminariesFlags} to define $\widetilde{M}_k^\rho(x)$.
Furthermore,  for all $e \neq \gamma \in \pi_1(S)$ we have
\[\ell_{B_k^\rho}(\gamma) = L_k(\rho(\gamma)).\]
\end{lem}
\begin{proof}
Since $\xi_\rho$ is $\rho$-equivariant, it follows that $B_k^\rho$ is $\pi_1(S)$-invariant.
Symmetry and additivity are computations.

For $\F=\R$ and $\upsilon=-\log$ in the proof of \Cref{thm_MZHitchinPosRatioed},  Martone--Zhang prove that already the expression $\tilde{B}_k^\rho(x)$ is positive for all $x \in \Fix(S)^{[4]}$.
Equivalently $\widetilde{M}_k^\rho(x) \geq 1$ for all $x \in \Fix(S)^{[4]}$.
Thus for fixed $x \in \Fix(S)^{[4]}$ the set
\[S_x \coloneqq \big\{ \rho \in \Hom_\Hit(\pi_1(S),\PSL(n,\R)) \mid \widetilde{M}_k^\rho(x) \geq 1 \big\} \]
is semi-algebraic by \Cref{lem_StableFlagSemiAlg}, and equal to $\Hom_\Hit(\pi_1(S),\PSL(n,\R))$.
Hence by the Tarski--Seidenberg transfer principle (\Cref{thm_TarskiSeidenberg}) the same holds true for its $\F$-extension,
i.e.\ $(S_x)_\F=\Hom_\Hit(\pi_1(S),\PSL(n,\F))$ for all $x \in \Fix(S)^{[4]}$.
Since $\F$-positive representations are $\PGL(n,\F)$-conjugate to $\F$-Hitchin representations (\Cref{thm_FHitchinEquivFPositive}),  $B_k^\rho$ is $\PGL(n,\F)$-invariant, and $\upsilon$ is order-compatible this concludes the positivity of $B_k^\rho$.

To show that $\ell_{B_k^\rho}(\gamma) = L_k(\rho(\gamma))$ we use similar arguments as in the proof of \Cref{propo_ClosedLeafInequality}.
Hence for $y=(\gamma^+,\gamma^-,x,\gamma x)$ we have using the same observation as in the aforementioned proof that
\begin{align*}
	\widetilde{M}_k^\rho(y)  &= \frac{\xi_\rho(\gamma^+)^{(n-k)} \wedge \xi_\rho(x)^{(k)}}{\xi_\rho(\gamma^+)^{(n-k)} \wedge \rho(\gamma) \xi_\rho(x)^{(k)}} \cdot \frac{\xi_\rho(\gamma^-)^{(n-k)} \wedge \rho(\gamma) \xi_\rho(x)^{(k)}}{\xi_\rho(\gamma^-)^{(n-k)} \wedge \xi_\rho(x)^{(k)}} \\
	&= \frac{\frac{1}{\lambda_{k+1} \cdot \ldots \cdot \lambda_{n}}}{\frac{1}{\lambda_1 \cdot \ldots \cdot \lambda_{n-k}}}
	= \frac{\lambda_1 \cdot \ldots \cdot \lambda_{k}}{\lambda_{n-k+1} \cdot \ldots \cdot \lambda_{n}},
\end{align*}
where $\lambda_1 > \ldots > \lambda_n$ are the eigenvalues (of the same sign) of a lift of $\rho(\gamma)$ to $\SL(n,\F)$.
With the definition of $B_k^\rho$ and the order-compatibility of $\upsilon$, the claim follows.
\end{proof}

We are ready to prove \Cref{thm_FHitchinCR}.

\begin{proof}[Proof of \Cref{thm_FHitchinCR}]
\Cref{lem_PosCRFHitchin} shows that for every $k=1,\ldots,n-1$ we can associate to $\rho$ a positive cross-ratio $B_k^\rho$, whose period is equal to the $k$-length.
\Cref{thm_BIPPGeodCurrAssToPosCR} applied to $X=\Fix(S)$ and $B=B_k^\rho$ provides a geodesic current $\mu_\rho^k$ with the desired intersection property.

Left to show is the last statement.
We only need to consider the non-Archimedean case.
If $g \in \SL(n,\F)$ is totally hyperbolic with positive eigenvalues $\lambda_1(g) > \ldots > \lambda_n(g)>0$,  then by order-compatibility of $\upsilon$ we have $\upsilon(\lambda_1(g)) \leq \ldots \leq \upsilon(\lambda_n(g))$.
Thus by an iterated application of the last statement of \Cref{lem_ValuationsProperties} we have
\[ \upsilon(\tr(g)) = \upsilon \bigg(\sum_{j=1}^n \lambda_j(g) \bigg) = \min_{j=1,\ldots,n} \upsilon(\lambda_j(g)) = \upsilon(\lambda_1(g)).\]
Furthermore, since $\det(g)=1 $ we have $\lambda_n(g)<1$ and thus $\upsilon(\lambda_n(g))\geq 0$.

We first show one direction of the claim for $k=1$.
Recall that $\upsilon(x)=\upsilon(-x)$ for all $x \in \F$ (\Cref{lem_ValuationsProperties}), and hence $\upsilon(\tr(\rho(\gamma))$ is independent of a choice of lift of $\rho(\gamma)$ to $\SL(n,\F)$.
If there exists $\gamma \in \pi_1(S)$ with $\upsilon(\tr(\rho(\gamma))) <0$, then by the above remark, we have
\[i(\mu_\rho^1, \gamma) = L_1(\rho(\gamma)) = - \upsilon(\lambda_1(\rho(\gamma))) + \upsilon(\lambda_n(\rho(\gamma))) \geq -\upsilon(\tr(\rho(\gamma))) > 0,\]
and hence $\mu_\rho^1$ is non-zero.
For $k > 1$ we note that $L_k(\rho(\gamma)) \geq L_1(\rho(\gamma))$ and hence $\mu_\rho^k \neq 0$.

Conversely, assume that $\mu_\rho^1$ is non-zero.
Otal \cite[Th\'eor\`eme 2]{Otal_SpectreMarqueLongSurfacesCourNeg} proved that geodesic currents are determined by their intersection function, and hence there exists $e \neq \gamma \in \pi_1(S)$ with
\[0 < i(\mu_\rho^1,\gamma) = L_1(\rho(\gamma))= -\upsilon \left( \tfrac{\lambda_1(\rho(\gamma))}{\lambda_n(\rho(\gamma))}\right).\]
Now assume by contradiction that $\upsilon(\tr(\rho(\gamma^s))) \geq 0$ for all $s \in \N$.
Newton's identities, see e.g.\ \cite{Kalman_NewtonId}, imply that the coefficients of the characteristic polynomial of $\rho(\gamma)$ belong to the ring $\mathcal{O} \coloneqq \{x \in \F \mid \upsilon(x)\geq 0\}$.
Since $\mathcal{O} \subseteq \F$ is a valuation ring, it is integrally closed in $\F$, see \cite[Theorem 3.1.3.(1)]{EnglerPrestel_ValuedFields}, and hence $\lambda_1(\rho(\gamma)), \ldots, \lambda_n(\rho(\gamma)) \in \mathcal{O}$.
Since $\prod_{j=1}^n \lambda_j(\rho(\gamma))=1$, it follows that $\tfrac{\lambda_1(\rho(\gamma))}{\lambda_n(\rho(\gamma))} = \lambda_1(\rho(\gamma))^2 \lambda_2(\rho(\gamma)) \cdots \lambda_{n-1}(\rho(\gamma)) \in \mathcal{O}$, a contradiction.
The same argument works for general $k>1$.
\end{proof}

The above theorem allows us to assign to closed points in the real spectrum compactification of the Hitchin component $\rsp^\cl(\Hit(S,n))$ a projective class of a geodesic current.
Let us recall the following characterization from \Cref{subs_RSCCharVar}:
\[
\rsp^\cl(\Hit(S,n))\cong \left\{ (\rho, \F_\rho) \ \middle\vert 
\begin{array}{l}
	\rho \from \pi_1(S) \to \PSL(n,\F_\rho) \textrm{ is } \F_\rho\textrm{-Hitchin,} \\
	\F_\rho \supseteq \R \textrm{ real closed, } \rho\textrm{-minimal,}\\
	\F_\rho \textrm{ Archimedean over } \R[\tr(\Ad(\rho))]
\end{array}
\right\}_{\Big/ \sim} ,
\]
where $\F_\rho$ is the $\rho$-minimal field; see \Cref{dfn_RhoMinField} and \Cref{thm_RSpHitchinComponent}.

\begin{lem} \label{lem_PosCrAssociatedtoFHitchin}
For all $k=1,\ldots,n-1$ the map
\begin{align*}
\varphi_k \from \rsp^\cl(\Hit(S,n)) &\to \PP \CR^+(\Fix(S)),\\
[(\rho,\F_\rho)] &\mapsto [ B_k^\rho]
\end{align*}
is well-defined and continuous.
\end{lem}

Let us show how this lemma implies \Cref{corol_RSpHitGeodCurr}.

\begin{proof}[Proof of \Cref{corol_RSpHitGeodCurr}]
For every $k=1,\ldots,n-1$ the map is given as the following composition of maps
\[\rsp^\cl(\Hit(S,n)) \to \PP\CR^+(\Fix(S)) \to  \PP \mathscr{C}(S),\]
where the first map is continuous by \Cref{lem_PosCrAssociatedtoFHitchin}, and the second map is given by \Cref{thm_BIPPGeodCurrAssToPosCR} (before projectivizing) and proven to be continuous.
\end{proof}

%% file: sections/associatinggeodcurrentsHitchin_continuity.tex
Before we prove \Cref{lem_PosCrAssociatedtoFHitchin} we need some preliminary considerations.

\begin{lem}
Let $(\rho, \F)$ represent a point in $\rsp^\cl(\Hit(S,n))$ with $\F=\F_\rho$ the $\rho$-minimal field (\Cref{dfn_RhoMinField}).
Let $F$ be a finite symmetric generating set for $\pi_1(S)$ and $E=F^{2^n-1}$.
Set
\begin{align}
\label{eqn_BigElements}
b_\rho \coloneqq \sum_{\gamma \in E} \tr(\rho(\gamma))^2 \in \F.
\end{align}
Then $b_\rho$ is positive and a big element in $\F$.
\end{lem}
\begin{proof}
Since $b_\rho$ is a sum of squares, it is clearly positive.
If $(\rho,\F)$ represents a closed point then $\F$, the minimal field of definition, is non-Archimedean by \Cref{thm_RSpHitchinComponent} and of finite transcendence degree over $\R$, hence contains a big element (\Cref{lem_FinTransDegreeBigElements}).
Since the trace determines the representation (see the considerations at the beginning of \Cref{subsection_SemiAlgModels}),  we have that $\tr(\rho(\gamma))^2$ is big for some $\gamma \in \pi_1(S)$.
A result of Procesi \cite{Procesi_InvariantsNTimesNMatrices} tells us that there is a finite set of traces that we need to consider (here the traces of the images of $\gamma \in E$ under $\rho$ suffice).
Thus $b_\rho$ is big.
\end{proof}

Thus to $(\rho,\F)$ we can associate the non-trivial order-compatible valuation $\upsilon_\rho \coloneqq \upsilon_{b_\rho}$ from \Cref{lem_ValBigElem}, which is unique up to scaling by \Cref{lem_ValUnique}.

\begin{lem}
\label{lem_CRWellDef}
If $(\rho,\F_\rho)$ and $(\rho',\F_{\rho'})$ represent the same point in \linebreak
$\rsp^\cl(\Hit(S,n))$, then $B_k^\rho =  B_k^{\rho'}$.
\end{lem}
\begin{proof}
Let us abbreviate $\F\coloneqq \F_\rho$ and $\F'\coloneqq \F_{\rho'}$.
Also write $b\coloneqq b_\rho$ and $b' \coloneqq b_{\rho'}$ for the corresponding big elements defined in \Cref{eqn_BigElements}.
Since $(\rho,\F)$ and $(\rho',\F')$ represent the same point, we can without loss of generality assume that there exists an order-preserving isomorphism $\alpha \from \F \to \F'$ and $g \in \PSL(n,\F')$ such that for all $\gamma \in \pi_1(S)$
\[ \alpha(\rho(\gamma))= g \rho'(\gamma) g^{-1}.\]
We note that if $\xi_{\rho'}$ is the limit map of $\rho'$ then $g \xi_{\rho'}$ is the limit map of $\alpha \circ \rho$.
It is easy to see that $\alpha \widetilde{M}_k^{ \rho} = \widetilde{M}_k^{\alpha \circ \rho} = \widetilde{M}_k^{ \rho'}$.
Since the trace is conjugation invariant we also have $b_{\rho'}=b_{\alpha \circ \rho} = \alpha (b)$.
Hence 
\begin{align*}
B_k^\rho &= \tfrac{1}{2} \big(\log_{b} (\widetilde{M}_k^{ \rho}) +\log_b (\widetilde{M}_{n-k}^{ \rho})\big)\\
&=\tfrac{1}{2} \big(\log_{\alpha(b)} (\widetilde{M}_k^{\rho'} ) +\log_{\alpha(b)} (\widetilde{M}_{n-k}^{\rho'})\big)\\
&= \tfrac{1}{2} \big(\log_{b'} (\widetilde{M}_k^{\rho'} ) +\log_{b'} (\widetilde{M}_{n-k}^{\rho'})\big)\\
&= B_k^{\rho'}. \qedhere
\end{align*}
\end{proof}

\begin{proof}[Proof of Lemma \ref{lem_PosCrAssociatedtoFHitchin}]
For all $k=1,\ldots,n-1$ the map $\varphi_k$ is well-defined by \Cref{lem_CRWellDef}.
We would like to show that $\varphi_k$ is continuous, where $\rsp^\cl(\Hit(S,n))$ is endowed with the spectral topology,  $\CR^+(\Fix(S))$ with the topology of pointwise convergence and $\PP \CR^+(\Fix(S))$ with the quotient topology.
The following sets form a subbasis of closed sets for the topology on $\CR^+(\Fix(S))$: For $U \subseteq \R_{\geq 0}$ closed and $x=(x_1,x_2,x_3,x_4) \in \Fix(S)^{[4]}$,  we define the closed set
\[C(U,x) = \{b \in \CR^+(\Fix(S)) \mid b(x) \in U\}.\]
A basic closed set is of the form
\[C([0,t],x) = \{b \in \CR^+(\Fix(S)) \mid b(x) \leq t\}\]
for $t >0$.

By \Cref{lem_CRWellDef} we can lift $\varphi_k$ to a map \[\widetilde{\varphi}_k \from \rsp^\cl(\Hit(S,n)) \to \CR^+(\Fix(S)), \quad [(\rho,\F_\rho)] \mapsto B_k^{\rho}\] such that the diagram
\[
\begin{tikzcd}
\rsp^\cl(\Hit(S,n)) \arrow[r,"\widetilde{\varphi}_k"] \arrow[rd,swap,"\varphi_k"] & \CR^+(\Fix(S))\arrow[d]\\
& \PP \CR^+(\Fix(S))
\end{tikzcd}
\]
commutes.
It suffices thus to show that $\widetilde{\varphi}_k$ is continuous.
We show that $\widetilde{\varphi}_k^{-1}(C([0,t],x))$ is closed in $\rsp^\cl(\Hit(S,n))$ for all $t>0$ and $x \in \Fix(S)^{[4]}$.
We have
\begin{align*}
\widetilde{\varphi}_k^{-1}(C&([0,t],x)) = \big\{ [(\rho,\F)] \in \rsp^\cl(\Hit(S,n)) \mid \widetilde{\varphi}_k([(\rho, \F)]) \in C([0,t],x) \big\} \\
&=\big\{ [(\rho,\F)] \in \rsp^\cl(\Hit(S,n)) \mid B_k^\rho(x) \leq t \big\} \\
&=\big\{ [(\rho,\F)] \in \rsp^\cl(\Hit(S,n)) \mid \log_{b_\rho} \big( (\widetilde{M}_k^\rho(x) \widetilde{M}_{n-k}^\rho(x))^{1/2}\big) \leq t \big\} \\
&=\bigcap_{\frac{p}{q}>t} \big\{[(\rho,\F)] \in \rsp^\cl(\Hit(S,n)) \mid \big(\widetilde{M}_k^\rho(x) \widetilde{M}_{n-k}^\rho(x)\big)^{1/2} \leq b_\rho^{\ p/q} \big\} \\
&=\bigcap_{\frac{p}{q}>t} \big\{[(\rho,\F)] \in \rsp^\cl(\Hit(S,n)) \mid
\big(\widetilde{M}_k^\rho(x) \widetilde{M}_{n-k}^\rho(x) \big)^q \leq b_\rho^{2p}\big\}
\end{align*}
The left hand side of the expression
\[\big(\widetilde{M}_k^\rho(x) \widetilde{M}_{n-k}^\rho(x) \big)^q \leq b_\rho^{2p}\]
depends rationally on $\rho$, since all involved flags that are needed to define the left hand side are the stable flags of $\rho(\gamma_1), \ldots, \rho(\gamma_4)$,  where $x_i = \gamma_i^+$.
By \Cref{lem_StableFlagSemiAlg},  we know that the stable flag depends rationally on the positive hyperbolic element.
Thus
\[\bigcap_{\frac{p}{q}>t} \big\{[(\rho,\F)] \in \rsp^\cl(\Hit(S,n)) \mid 
\big(\widetilde{M}_k^\rho(x) \widetilde{M}_{n-k}^\rho(x) \big)^q \leq b_\rho^{2p}\big\}\]
is closed, which was to prove.
\end{proof}

%% file: sections/11_proofapplications.tex
\section{%
\texorpdfstring{%
Proof of \Cref{corol_PiecewiseSemiAlgPath} and \Cref{corol_DualSpace}}%
{Proof of Corollary 1.11 and Corollary 1.12}}
\label{section_ProofApplications}

Before we prove \Cref{corol_PiecewiseSemiAlgPath} we introduce the Weyl chamber length compactification \cite{Parreau_CompactificationsReprSpacesFinGenGroups},  which generalizes the Morgan--Shalen compactification \cite{MorganShalen_ValuationsTreesDegHypStructures} to higher rank.
For more details we refer to \cite{Parreau_CompactificationsReprSpacesFinGenGroups} and \cite{BurgerIozziParreauPozzetti_RSCCharacterVarieties2},  especially the later for its connection to the real spectrum compactification.
Again let us simplify only to the case where $G=\SL(n,\R)$ or $\PSL(n,\R)$.
The following definitions and result, however, apply more generally.

Recall from the introduction the Weyl chamber length function.
Let $\F$ be a real closed field with an order-compatible valuation $\upsilon \from \F \to \R \cup \{\infty\}$ (if $\F=\R$,  we take $\upsilon=-\log$). 
Given $\rho \from \pi_1(S) \to G_\F$ a representation,  then $\rho(\gamma)$ has (generalized) eigenvalues $\lambda_1(\rho(\gamma)), \ldots, \lambda_n(\rho(\gamma))$ in $\F[\sqrt{-1}]$ (potentially of a lift of $\rho(\gamma)$ to $\SL(n,\F)$) sorted such that their $\F$-valued absolute values satisfy $|\lambda_1(\rho(\gamma))| \geq \ldots \geq |\lambda_n(\rho(\gamma))|>0$.
The \emph{Weyl chamber length function} is the function
\begin{align*}
L_\rho \from \pi_1(S)&\to \overline{\mathfrak{a}}^+,\\
\gamma &\mapsto (-\upsilon(|\lambda_1(\rho(\gamma))|),\ldots,-\upsilon(| \lambda_n(\rho(\gamma))|)),
\end{align*}
where $\mathfrak{a}$ denotes the Cartan subalgebra of the Lie algebra of $\SL(n,\R)$ consisting of diagonal matrices of trace zero,  and $\overline{\mathfrak{a}}^+$ the closed Weyl chamber consisting of those elements $\textnormal{diag}(a_1,\ldots,a_n)\in \mathfrak{a}$ with $a_1\geq \ldots\geq a_n$.
Note that $L_\rho$ only depends on the conjugacy class of $\rho$.
If $L_\rho$ does not vanish identically,  we have a well-defined map $\PP L_\rho \from \pi_1(S) \to \PP\left(\overline{\mathfrak{a}}^+\right)$.

Let $\Hom_X \subseteq \Hom_\red(\pi_1(S),G)$ be a closed $G$-invariant semi-algebraic set,  and $X$ its image under a semi-algebraic model for $\chi(S,G)$,  which is closed and semi-algebraic as well,  see \Cref{subsection_SemiAlgModels}.
We denote by $\widehat{X}$ the Alexandrov one-point compactification of $X$,  which is the topological space $X \cup \{\infty\}$,  where a fundamental set of neighborhoods of $\infty$ is given by the complements of the compact sets of $X$.
The \emph{Weyl chamber length compactification} $\textnormal{WL}(X)$ of $X$ is the closure of the image of the embedding
\[ X \to \widehat{X} \times \PP\left(\big(\overline{\mathfrak{a}}^+\big)^{\pi_1(S)}\right),  \quad [\rho] \mapsto ([\rho],\PP L_\rho).\]
We denote by $\partial \textnormal{WL}(X) \coloneqq  \textnormal{WL}(X) \setminus X$ the set of length functions in the boundary of the Weyl chamber length compactification of $X$.

We have the following relation between the real spectrum and the Weyl chamber length compactification.

\begin{theor}[{\cite[Theorem 8.2]{BurgerIozziParreauPozzetti_RSCCharacterVarieties2}}]
\label{thm_ProjWLComp}
For any $G$-invariant,  closed,  semi-algebraic subset $\Hom_X \subseteq \Hom_\red(\pi_1(S),G)$ avoiding representations for which $L$ vanishes identically,  the map
\[ \widehat{\PP L} \from \rsp^\cl(X) \to  \textnormal{WL}(X), \quad [\rho] \mapsto \begin{cases} ([\rho], \PP L_\rho),  \textnormal{ if } [\rho] \in X, \\
(\infty,  \PP L_\rho),  \textnormal{ if } [\rho] \in \partial\rsp^\cl(X),
\end{cases}
 \]
is continuous and surjective.
\end{theor}

That the map is well-defined even on the closed points of the boundary of the real spectrum compactification follows from \cite[Theorem 1.2 and Proposition 5.21]{BurgerIozziParreauPozzetti_RSCCharacterVarieties2}.
By construction,  the map $\widehat{\PP L}$ is injective on $X$ (but not in general) and sends boundary points to boundary points.
We have the following accessibility result,  that follows from general properties of the real spectrum \cite[Section 3.2]{BurgerIozziParreauPozzetti_RSCCharacterVarieties2}.

\begin{corol}[{\cite[Corollary 8.4]{BurgerIozziParreauPozzetti_RSCCharacterVarieties2}}]
\label{corol_Accessibility}
For every $[L] \in \partial \textnormal{WL}(X)$ there exists a continuous locally semialgebraic path $\rho_t \from [0,\infty) \to \Hom_X$ such that, for every $\gamma \in \pi_1(S)$ we have
\[\lim_{t\to \infty} \frac{L_{\rho_t}(\gamma)}{t}= L(\gamma).\]
\end{corol}

With this result at hand we can finish the proof of \Cref{corol_PiecewiseSemiAlgPath}.

\begin{proof}[Proof of \Cref{corol_PiecewiseSemiAlgPath}]
Let $\Hom_X=\Hom_\Hit(\pi_1(S),\PSL(n,\R))$ be the closed $\PSL(n,\R)$-invariant semi-algebraic subset of $\Hom_\red(\pi_1(S),\PSL(n,\R))$ of Hitchin representations.
Since these are positively hyperbolic,  their length functions do not vanish identically.
The semi-algebraic set $X$ is homeomorphic to the Hitchin component $\Hit(S,n)$.

By assumption,  the length function $L_\rho$ does not vanish identically,  thus $(\rho,\F_\rho)$ represents a closed point in $\rsp(X)$,  see \Cref{thm_RSpHitchinComponent} and the above characterization of closed points in terms of not identically vanishing length functions.
By \Cref{thm_ProjWLComp},  it follows that $\PP L_\rho \in \partial \textnormal{WL}(X)$.
Thus \Cref{corol_Accessibility} implies that there exists a continuous locally semi-algebraic path $\rho_t \from [0,\infty) \to \Hom_\Hit(\pi_1(S),\PSL(n,\R))$ and $c>0$ such that for every $\gamma \in \pi_1(S)$ we have
\[\lim_{t\to \infty} \frac{L_{\rho_t}(\gamma)}{t}= c L_\rho(\gamma),\]
which is what we wanted to prove.
\end{proof}

\medskip

Let us now turn to the second application concerning dual spaces of geodesic currents.
These should be thought of as the generalization of an $\R$-tree associated to a measured lamination.
For a detailed introduction we refer to \cite{DeRosaMartinezGranado_DualSpacesGeodesicCurrents}.
Choose an auxiliary hyperbolic structure on $S$ and denote by $\tilde{S}$ the universal cover of $S$.

\begin{dfn}[{\cite[Definition 3.1]{DeRosaMartinezGranado_DualSpacesGeodesicCurrents}}]
A geodesic current $\mu$ induces a pseudo-distance on $\tilde{S}$ given by
\[d_\mu(x, y) = \frac{1}{2}(\mu(G[x, y)) + \mu(G(x, y]),\]
where $G[x, y)$ denotes the set of hyperbolic geodesics of $\tilde{S}$ transverse to the geodesic segment $[x,y)$,  for all $x, y \in \tilde{S}$.
The \emph{dual space} of $\mu$, denoted by $X_\mu$, is defined as the quotient metric space of $\tilde{S}$,  where two points are identified if their pseudo-distance is equal to $0$.
\end{dfn}

The set of all dual spaces of geodesic currents will be denoted by $\mathcal{D}(S)$ and is equipped with the equivariant Gromov--Hausdorff topology, first introduced and studied by Paulin \cite{Paulin_TopGromovEquivStructuresHypArbresReels}.
This is a variation of the Gromov--Hausdorff topology that takes the action of a group into account.
Recall that $\mathscr{C}(S)$ denotes the space of geodesic currents equipped with the weak$^\star$-topology.

\begin{theor}[{\cite[Theorem 10.3]{DeRosaMartinezGranado_DualSpacesGeodesicCurrents}}]
The map 
\[\Psi \from \mathscr{C}(S) \to \mathcal{D}(X), \quad \mu \mapsto X_\mu\]
is continuous.
\end{theor}

\Cref{corol_DualSpace} is then a direct consequence of  \Cref{corol_PiecewiseSemiAlgPath} and this theorem.

\begin{proof}[Proof of \Cref{corol_DualSpace}]
From \Cref{corol_PiecewiseSemiAlgPath},  we already know of the existence of the sequence of real representations $\rho_t \from \pi_1(S) \to \PSL(3,\R)$,  whose length functions converge to the $c$-multiple of the ones of $\rho$.
Thus the sequence of geodesic currents $\tfrac{1}{t}\mu_{\rho_t}$ converges to $c \mu_\rho$,  since geodesic currents are determined by their length functions \cite[Th\'eor\`eme 2]{Otal_SpectreMarqueLongSurfacesCourNeg} and the intersection form is continuous.
By continuity of the map $\Psi$ from the above theorem,  we have that $\Psi(\tfrac{1}{t}\mu_{\rho_t})$ converges to $\Psi(c \mu_\rho)$ in the equivariant Gromov--Hausdorff topology.
Now \cite[Proposition 5.13]{DeRosaMartinezGranado_DualSpacesGeodesicCurrents} implies that for all $t>0$ the dual metric space $X_{\mu_{\rho_t}/t}$ is isometric to $(\Omega_{\rho_t},\tfrac{1}{t} d_{\Omega_{\rho_t}})$,  where $\Omega_{\rho_t}$ is the properly convex subset of the real projective space invariant under $\rho_t$,  endowed with its Hilbert metric $d_{\Omega_{\rho_t}}$.
The dual space of $c \mu_\rho$ is isometric to the dual space of $\mu_\rho$,  where we scale the distance by $c$.
This finishes the proof.
\end{proof}